\documentclass[11pt]{article}
\usepackage[margin=1.1in,a4paper]{geometry}
\usepackage[dvipsnames]{xcolor}
\usepackage[compress]{natbib}
\usepackage{bibunits}
\usepackage{amsmath, amssymb, bbm, graphicx, url, amsthm, subcaption, float}
\usepackage{amsfonts,bm} 
\usepackage[mathcal]{eucal}
\usepackage{tikz}

\usepackage[ruled,linesnumbered]{algorithm2e}
\usepackage{amsfonts}        \usepackage{xcolor}
\usepackage[colorlinks=true,citecolor=blue]{hyperref}    

\DeclareMathOperator{\cov}{cov}
\DeclareMathOperator{\var}{var}

\DeclareMathOperator{\tr}{tr}
\newcommand{\mat}{\textsc{\small mat}}
\renewcommand{\vec}{\textsc{\small vec}}
\newcommand{\pair}{\textsc{\small pair}}

\newcommand{\otK}{\otimes_{\textsc{k}}}

\newcommand{\cL}{\mathcal{L}}  
\newcommand{\B}{\mathbb{B}} 
\newcommand{\E}{\mathbb{E}} 
\newcommand{\eps}{{\mathcal{E}}}   
\renewcommand{\P}{\mathbb{P}}  
\newcommand{\N}{\mathbb{N}}  
\newcommand{\cN}{\mathcal{N}}  
\newcommand{\R}{\mathbb{R}}  
\renewcommand{\S}{\mathbb{S}}  
 
\newcommand{\U}{\mathbb{U}} 
\newcommand{\V}{\mathbb{V}}  
  
\newcommand{\x}{\mathbf{x}} 
\newcommand{\X}{\mathbf{X}}  
\newcommand{\Z}{\mathbf{Z}}

\newcommand{\cT}{\mathcal{T}} 
\newcommand{\sF}{\mathsf{F}} 
\newcommand{\sS}{\mathsf{S}} 
\newcommand{\sT}{\mathsf{T}} 
\newcommand{\sM}{\mathsf{M}} 
\newcommand{\cM}{\mathcal{M}} 
\newcommand{\sV}{\mathsf{V}} 

\newcommand{\<}{\langle}
\renewcommand{\>}{\rangle}

\newcommand{\s}[1]{\textsc{\tiny #1}} 
\newcommand{\ST}{\mathsf{S}}

\newcommand{\I}{\mathbb{I}}

\newcommand{\bs}{\boldsymbol}

\newtheorem{thm}{Theorem}[section]
\newtheorem{lem}[thm]{Lemma}
\newtheorem{prop}[thm]{Proposition}
\newtheorem{cor}[thm]{Corollary}

\theoremstyle{definition}
\newtheorem{defn}[thm]{Definition}

\newtheorem{ex}[thm]{Example}
\newtheorem{assp}{Assumption}

\theoremstyle{remark}
\newtheorem{rem}[thm]{Remark}

\newcommand{\piotr}[1]{{\small{\color{BrickRed} \sf $\clubsuit$ Piotr: #1}}}

\title{Probabilistic PCA on tensors}
\author{Yaoming Zhen\thanks{School of Data Science, The Chinese University of Hong Kong, Shenzhen \texttt{yaomingzhen@cuhk.edu.cn}}
\and
Piotr Zwiernik\thanks{Department of Economics and Business, Universitat Pompeu Fabra, and Barcelona School of Economics \texttt{piotr.zwiernik@upf.edu}}
}

\begin{document}

\maketitle
\begin{abstract}
In probabilistic principal component analysis (PPCA), an observed vector is modeled as a linear transformation of a low-dimensional Gaussian factor plus isotropic noise. We generalize PPCA to tensors by constraining the loading operator to have Tucker structure, yielding a probabilistic multilinear PCA model that enables uncertainty quantification and naturally accommodates multiple, possibly heterogeneous, tensor observations. We develop the associated theory: we establish identifiability of the loadings and noise variance and show that—unlike in matrix PPCA—the maximum likelihood estimator (MLE) exists even from a single tensor sample. We then study two estimators. First, we consider the MLE and propose an expectation–maximization (EM) algorithm to compute it. Second, exploiting that Tucker maps correspond to rank-one elements after a Kronecker lifting, we design a computationally efficient estimator for which we provide provable finite-sample guarantees. Together, these results provide a coherent probabilistic framework and practical algorithms for learning from tensor-valued data.
\end{abstract}

\begin{bibunit}[apalike]

\section{Introduction}

Tensors are ubiquitous in contemporary data science, arising in fields as diverse as computer vision \citep{panagakis2021tensor}, recommender systems \citep{bi2018multilayer, zhang2021dynamic}, neuroscience \citep{zhou2013tensor, zhou2023partially}, genomics \citep{jing2021community}, and chemometrics \citep{cichocki2015tensor}.  By keeping multi‐way structure intact, tensor methods can reveal interactions that are invisible to vector- or matrix-based approaches, yet the resulting high dimensionality demands careful statistical modeling.

\medskip
\noindent\textbf{From PCA to probabilistic tensor models.}
Classical Principal Component Analysis (PCA) \citep{pearson1901liii,hotelling1933analysis,greenacre2022principal} is one of the most widely used techniques for linear dimension reduction. Probabilistic PCA (PPCA; \citealp{tipping1999probabilistic}) recasts PCA as a latent-variable model, enabling likelihood-based inference, uncertainty quantification, principled handling of missing data, and natural extensions to hierarchical Bayesian modeling. PPCA, however, is fundamentally matrix-centric; naively vectorizing a tensor destroys multilinear structure and inflates the number of parameters. Our goal is to endow the widely used Tucker decomposition \citep{tucker1966some} with a fully probabilistic counterpart. 

We introduce \emph{Tensor Probabilistic PCA} (TPCA), a generative latent-variable model that factorizes an order-$r$ tensor $X\in\R^{n_1\times\cdots\times n_r}$ as 
\[
X \;=\; (A_1,\ldots,A_r)\cdot Z \;+\; \eps, 
\quad 
Z_{i_1,\ldots,i_r}\stackrel{\text{i.i.d.}}{\sim} N\bigl(0,1\bigr),
\quad
\eps_{i_1,\ldots,i_r}\stackrel{\text{i.i.d.}}{\sim} N\!\bigl(0,\sigma^{2}\bigr),
\]
where each factor matrix $A_k\in\R^{n_k\times m_k}$ has $m_k\ll n_k$ columns and ``$\cdot$'' denotes the standard Tucker product (formal definition in Section~\ref{sec:tensors}). Setting  $r=1$ reduces TPCA to PPCA.  With $r\ge2$, TPCA becomes a \emph{probabilistic} version of multilinear PCA (MPCA; \citealp{lu2008mpca}) and Tucker-PPCA \citep{chu2009ptucker}, providing an integrated framework for inference, hypothesis testing and model selection. For simplicity of exposition we assume throughout this section that $X$ has mean zero. With repeated samples, a general mean $\mu$ can always be estimated separately and subtracted, while in small-sample regimes TPCA also accommodates structured means of the form $(A_1,\ldots,A_r)\cdot\nu$ (see Section~\ref{sec:tpca}).

\medskip
\noindent\textbf{Related work.}
Deterministic tensor decompositions are well established—CP \citep{carroll1970analysis}, Tucker \citep{tucker1966some}, tensor-train \citep{oseledets2011tensor}, CUR \citep{JMLR:v22:21-0287}, and their orthogonally constrained variants—but lack a likelihood.  
Early probabilistic attempts include Bayesian Tucker and CP models with MCMC or variational Bayes \citep{zhao2015bayesian,lyu2019tensor,mai2022doubly}; structured matrix/tensor normal models \citep{matsuda2022estimation,hoff2023core} share a Gaussian likelihood yet treat Kronecker covariances rather than low-rank factors.  
Our work differs in that (i) the latent ``core'' $Z$ is \emph{isotropic} allowing for sharp identifiability results; (ii) the factor loading matrices are the only free parameters apart from $\sigma^{2}$; and (iii) TPCA nests many popular models, including tensor-normal models, multilinear subspace models, CP factor models and Spiked tensor models under a single probabilistic umbrella.

On the algorithmic side, EM algorithms for Gaussian factor models are classical \citep{rubin1982em}; tensor-specific EM variants appear in \citet{hannani2019EMtensor} and, more recently, in high-dimensional settings \citep{wu2024sharp,chattopadhyay2024blessing}. Spectral and power-iteration approaches have also been explored for tensor PCA \citep{han2022tensor,wang2023exponentiated}, and distributed or federated implementations are emerging \citep{chen2025distributed}. We build on these ideas in two complementary ways: we develop an \emph{exact-likelihood} EM algorithm tailored to Tucker-structured PPCA for computing the MLE, and we design a separate, computationally efficient estimator that exploits the Kronecker structure. The latter comes with finite-sample performance guarantees and markedly lower computational cost, while the EM algorithm targets the MLE and enjoys standard asymptotic properties. Together, these methods bridge likelihood-based modeling and scalable learning for tensor-valued data. 
\medskip

\noindent\textbf{Why TPCA in practice?}
We list four concrete benefits here that TPCA offers.  
First, it handles \emph{repeated observations}: unlike most tensor decompositions that work with a single array, TPCA is defined for $N\!\ge 1$ independent samples, enabling population inference and prediction.  
Second, it copes gracefully with \emph{missing data}: because the model is Gaussian, the EM algorithm replaces every unobserved entry by its exact conditional mean given the data that are recorded, so TPCA can be fit no matter which entries are absent—entire slices, scattered pixels, or anything in between.  
Third, TPCA comes with built-in \emph{uncertainty quantification}: posterior (or observed-Fisher) covariances for the factor loadings measure sampling variability that deterministic methods ignore.  
Finally, the likelihood enables principled \emph{model selection and completion}: marginal likelihoods lead to AIC, BIC, and Bayes factors for choosing multilinear ranks $(m_1,\dots,m_r)$, and the same framework supports low-rank tensor completion \citep{kressner2014low}. Such situations arise naturally in practice: fMRI studies involve repeated scans across many patients with missing or corrupted voxels; recommender systems generate sparse user–item–time tensors; and in chemometrics or genomics, experiments yield multiway measurements whose uncertainty must be quantified.

\medskip
\noindent\textbf{Main contributions.}
Our contributions are threefold. 
\emph{First}, we introduce TPCA --- a fully probabilistic analogue of the Tucker decomposition --- and show that it nests PPCA ($r=1$), multilinear PCA, tensor-normal models, and symmetric variants under mild constraints. 
\emph{Second}, we provide subtle algorithms tailored to the model: (i) an exact-likelihood EM algorithm that computes the MLE and whose E-step exploits Kronecker structure for $O\!\bigl(\sum_k n_k^3\bigr)$ cost per iteration; and (ii) a scalable rank-1, power-iteration–type estimator that leverages the Tucker-to–rank-one correspondence, together with finite-sample performance guarantees and convergence analysis. 
\emph{Third}, we develop a statistical theory for the model: TPCA is identifiable up to the usual orthogonal and scaling indeterminacies, and its maximum-likelihood estimator (MLE) exists even for a single observation ($N=1$). We also derive information-theoretic lower bounds and non-asymptotic error rates for a computationally efficient rank-1 power iteration estimator. Particularly, the rank-1 power iteration estimator would be consistent even after the first iteration when almost arbitrary initialization is given.  

\medskip
\noindent\textbf{Notation.} Throughout we adopt the concise language of multilinear algebra: tensors are elements of a finite Cartesian product of Euclidean spaces; mode-$k$ multiplication and the Tucker product are defined via linear maps, letting us avoid cumbersome index gymnastics. The details are provided in the next section.

\medskip
\noindent The remainder of the paper is organized as follows.  Section~\ref{sec:tensors} fixes tensor notations and recalls basic definitions.  Section~\ref{sec:tpca} introduces TPCA and its basic properties.  Sections~\ref{sec:EM} and~\ref{sec:algorithm} develop the MLE and power-iteration estimators together with some numerical and theoretical results.  Numerical experiments in Section~\ref{sec:experiments} complement the theory, and Section~\ref{sec:conclusion} concludes the paper.

\section{Tensor spaces and notation}\label{sec:tensors}

The proofs of this section are rather straightforward and standard, and they have been moved to Appendix~\ref{app:tensors_proofs} to streamline the discussion. 

\subsection{The inner product space of tensors}

For any $n\in\N=\{1,2,\ldots\}$, let $[n]:=\{1,\ldots,n\}$.  
An $r$-way array of real numbers with dimensions $(n_1,\ldots,n_r)$ is written as $T=(T_{i_1,\ldots,i_r})$ with $i_k\in[n_k]$.  
The set of such arrays is denoted by
\[
  \V:=\R^{n_1\times\cdots\times n_r}.
\]
Rank-one tensors are outer products of the form
$\bs v=v_1\otimes\cdots\otimes v_r$ with $v_k\in\R^{n_k}$, that is
\[
  \bs v_{i_1,\ldots,i_r}= (v_1)_{i_1}\cdots (v_r)_{i_r}.
\]
With $e^{(k)}_j$ the $j$-th canonical vector in $\R^{n_k}$,  
$\bs e^{\,i_1,\ldots,i_r}:=e^{(1)}_{i_1}\otimes\cdots\otimes e^{(r)}_{i_r}$ forms an orthonormal basis under the standard inner product
\begin{equation}\label{eq:inner}
  \langle S,T\rangle
  \;=\;\sum_{i_1=1}^{n_1}\!\cdots\sum_{i_r=1}^{n_r}
    S_{i_1,\ldots,i_r}\,T_{i_1,\ldots,i_r},
\end{equation}
whose induced norm is the Frobenius norm
$\|S\|_{\s{f}}:=\sqrt{\langle S,S\rangle}$.  
Throughout we identify $\R^{n_1\times\cdots\times n_r}$ with the algebraic tensor product
$\R^{n_1}\otimes\cdots\otimes\R^{n_r}$; see Section 3 of \cite{comon2008symmetric} for details.

\subsection{Linear maps between tensor spaces}\label{sec:slm}

Let $\U=\R^{m_1\times\cdots\times m_r}$ and
$\V=\R^{n_1\times\cdots\times n_r}$.  
Write $\mathcal L(\U,\V)$ for the space of linear maps $\U\to\V$.  
Exactly as in the matrix case ($r=1$), $\mathcal L(\U,\V)$ is in bijection with the order-$2r$ tensor space
$\R^{n_1\times\cdots\times n_r\times m_1\times\cdots\times m_r}$:
for $\sF\in\mathcal L(\U,\V)$ we set
\[
  \sF_{i_1,\ldots,i_r,j_1,\ldots,j_r}
  :=\bigl(\sF(\bs e^{\,j_1,\ldots,j_r})\bigr)_{i_1,\ldots,i_r}.
\]
Therefore, specifying $\sF$ on the basis tensors $\bs e^{\,j_1,\ldots,j_r}$ determines it uniquely.

\begin{rem}
Throughout, capital letters $S,T$ denote \emph{tensors};
bold letters $\bs v$ denote rank-one tensors;
sans-serif $\sF,\sM,\Sigma$ denote tensors in \emph{operator spaces} like $\cL(\U,\V)$. 
\end{rem}

For linear maps
$\mathsf G:\U\to\R^{k_1\times\cdots\times k_r}$
and
$\sF:\R^{k_1\times\cdots\times k_r}\to\V$
their composition $\sF\mathsf G\in\mathcal L(\U,\V)$ satisfies
\[
  (\sF\mathsf G)_{i_1,\ldots,i_r,j_1,\ldots,j_r}
  =\sum_{\ell_1=1}^{k_1}\!\cdots\sum_{\ell_r=1}^{k_r}
      \sF_{i_1,\ldots,i_r,\ell_1,\ldots,\ell_r}\,
      \mathsf G_{\ell_1,\ldots,\ell_r,j_1,\ldots,j_r}.
\]

\begin{defn}[Tucker maps]
Given matrices $A_k\in\R^{n_k\times m_k}$ ($k=1,\dots,r$), let
$\mathbf A=(A_1,\dots,A_r)$.  
The \emph{Tucker map}
$\sF_{\mathbf A}\in\mathcal L(\U,\V)$ is defined as a map $T\mapsto \bs A\cdot T$, where
\[
  \bigl(\mathbf A\cdot T\bigr)_{i_1,\ldots,i_r}
  \;:=\;\sum_{j_1=1}^{m_1}\!\cdots\sum_{j_r=1}^{m_r}
        (A_1)_{i_1j_1}\cdots(A_r)_{i_rj_r}\,T_{j_1,\ldots,j_r}.
\]
\end{defn}

\begin{rem}
Many authors write the same operation as
$T\times_1A_1\times_2\cdots\times_rA_r$;  we keep ``$\mathbf A\cdot T$'' to emphasise linear-algebraic structure.
\end{rem}

Tucker maps respect rank-one structure:
\[
  (A_1,\dots,A_r)\cdot
    (v_1\otimes\cdots\otimes v_r)
  =(A_1v_1)\otimes\cdots\otimes(A_rv_r).
\]
Their parameter dimension is
$\sum_{k=1}^r m_k n_k\ll\prod_{k=1}^r m_k n_k
      =\dim\mathcal L(\U,\V)$, 
which makes them statistically attractive (see Proposition~\ref{prop:spmap}).

\subsection{Two spectral norms}

For $\sF\in\mathcal L(\U,\V)$ set its operator norm as 
\begin{equation}\label{eq:spnorm0}
      \|\sF\|
  :=\max_{\|S\|_{\s{f}}\le1}\|\sF(S)\|_{\s{f}}
  =\max_{\|S\|_{\s{f}},\,\|T\|_{\s{f}}\le1}
     \langle\sF(S),T\rangle.
\end{equation}
However, a weaker norm is often useful.  
With
\begin{equation}\label{eq:Bv}
      \B_{\s{u}}:=\{\bs u=u_1\otimes\cdots\otimes u_r:
                  \|u_k\|\le1\},
  \qquad
  \B_{\s{v}}:=\{\bs v=v_1\otimes\cdots\otimes v_r:
                  \|v_k\|\le1\},
\end{equation}
define
\begin{equation}\label{eq:spnorm}
  \|\sF\|_{\mathrm{ml}}
  \;:=\;\max_{\bs u\in\B_{\s{u}}}
        \max_{\bs v\in\B_{\s{v}}}
        \langle\sF(\bs u),\bs v\rangle.
\end{equation}
Clearly $\|\sF\|_{\mathrm{ml}}\le\|\sF\|$, and the gap can be large.  
For example, take $\U=\V=\R^{n_1\times\cdots\times n_1}$ and
\[
  \sF(T)=\langle T,\I\rangle\,J,
\]
where $\I$ is the identity tensor ($\I_{i\cdots i}=1$ for $i\in [n_1]$ and zeros otherwise) and $J$ is the all-ones tensor.
Then
\[
  \|\sF\|_{\mathrm{ml}}=n_1^{r/2},
  \qquad
  \|\sF\|=n_1^{(r+1)/2},
\]
so the ratio $\|\sF\|/\|\sF\|_{\mathrm{ml}}=\sqrt{n_1}$ diverges with mode size $n_1$. The following lemma shows that no such gap exists for Tucker maps.

\begin{lem}\label{lem:normsequal}
    For  Tucker maps $\|\sF_{\bs A}\|_{\rm ml}= \|\sF_{\bs A}\|=\prod_{k=1}^r\|A_k\|$.
\end{lem}

\subsection{Computations with tensors}\label{sec:tensorcomp}

While the language of abstract linear algebra is convenient for most of the theoretical derivations throughout, in the actual implementations we  work with tensor reshaping of various kinds. In this subsection, we present several reshaping operators and their properties. 
\medskip

\noindent\textbf{Operator} $\vec$\textbf{.} Denote $n=n_1\cdots n_r$. Let $\vec:\V\to \R^{n}$ be a mapping that flattens tensor $T$ to a vector $\vec(T)$ with entries one below another in an arbitrary but fixed order. Particularly, we fix this order by defining the ranking function that takes an $r$-tuple in $[n_1]\times \cdots\times [n_r]$ and outputs a number in $[n]$:
\begin{equation}\label{eq:ranking}
\overline{i_1,\ldots,i_r}\;:=\;1 + \sum_{k=1}^r (i_k - 1) \prod_{\ell=1}^{k-1} n_\ell,\quad\mbox{for any }  i_1\in  [n_1],\ldots, i_r \in  [n_r], \text{ and } r \in \N.
\end{equation}
 This simply gives the colexicographic ordering (see e.g. Section~1.4 in \cite{liu2022tensor}) of the $r$-tuples, which is consistent with how it done in Matlab, R or Python. For example, the induced total ordering on $[2]\times [2]\times [2]$ is $$
(1,1,1)\;\succ\;(2,1,1)\;\succ\;(1,2,1)\;\succ\;(2,2,1)\;\succ\;(1,1,2)\;\succ\;(2,1,2)\;\succ\;(1,2,2)\;\succ\;(2,2,2).
$$
With this ranking function, we have
$\vec(T)_{\overline{i_1,\ldots,i_r}}\;=\; T_{i_1,\ldots, i_r}$ and $\<S,T\>=\vec(S)^\top \vec(T)$.    
\medskip

\noindent\textbf{Operators }$\mat$ \textbf{and }$\pair$\textbf{.}
The $\sF$ be the $2r$-tensor representing a mapping $\sF\in \mathcal L(\U,\V)$. This tensor can be flattened in various useful ways. We consider two. First, we denote by $\mat(\sF)$ the underlying matrix with $n=n_1\cdots n_r$ rows and $m=m_1\cdots m_r$ columns. Formally, $\mat(\sF)\in \R^{n\times m}$ is defined via
\begin{equation}\label{eq:matF}
\mat(\sF)_{\overline{i_1,\ldots,i_r}, \overline{j_1,\ldots, j_r} } \;=\; \sF_{i_1,\ldots,i_r,j_1,\ldots,j_r}. 
\end{equation}

The following lemma is well-known and easily checked. 

\begin{lem}\label{lem:matop}Suppose $\sF\in \mathcal L(\U,\V)$, $\mathsf{G} \in \mathcal{L}(\mathbb{W}, \U)$ for any tensor space $\mathbb{W}$, and $S\in \U$,  it then holds that:
\begin{enumerate}
    \item [(i)] $\vec(\sF S)\;=\;\mat(\sF)\vec(S)$,
    \item [(ii)] $\vec(\mat(\sF))=\vec(\sF)$,
    \item [(iii)] $\mat(\sF\mathsf{G})=\mat(\sF)\mat(\mathsf{G})$.
\end{enumerate}
Moreover, (i) and (iii) together imply that $\mat(\sF^{-1})=\mat(\sF)^{-1}$ if $\sF$ is invertible with inverse $\sF^{-1}$. 
\end{lem}

 For any linear map $\sF:\U\to \V$, we define its Frobenius norms as  $\|\sF\|_{\s{f}}:=\|\mat(\sF)\|_{\s{f}}$. Note also that the spectral norm defined in \eqref{eq:spnorm0} satisfies $\|\sF\|=\|\mat(\sF)\|$.

The second useful flattening of $\sF\in \mathcal L(\U,\V)$ pairs the corresponding modes in $\U$ and $\V$. Formally, $\pair(\sF)\in \R^{n_1m_1\times\cdots\times n_rm_r}$ is defined as
$$
\pair(\sF)_{\overline{i_1,j_1},\ldots, \overline{i_r,j_r}}\;=\;\sF_{i_1,\ldots,i_r,j_1,\ldots,j_r}.
$$
\medskip

\noindent\textbf{The special property of Tucker maps.} As we will see in the following proposition, the $\pair(\cdot)$ operator gives a powerful characterization of Tucker maps. 
\begin{prop}\label{prop:spmap}
    A linear mapping $\sF\in \cL(\U,\V)$ is a Tucker map if and only if $\pair(\sF)$ is a rank one tensor in the tensor space $\R^{m_1n_1\times \cdots\times m_rn_r}$.
\end{prop}

In the terminology of tensor CP decomposition \citep{kolda2009tensor}, we say that $\pair(\sF)$ has CP rank $1$. Besides, the flattening $\mat(\sF)$ also has a useful representation for Tucker maps in terms of tensor Kronecker product defined as follows.
\begin{defn}\label{def:Kron}
    The Kronecker product between  tensors $A \in \V$ and $B \in \U$ is denoted as $A \otK B \in \mathbb{R}^{n_1m_1 \times \cdots \times n_r m_r}$, with entry 
\begin{align*}
[A \otK B]_{\overline{j_1i_1}, \ldots,  \overline{j_ri_r} } \;=\; A_{i_1, \ldots, i_r} B_{j_1, \ldots, j_r}. 
\end{align*}
In particular, when $r=2$ we obtain the familiar matrix Kronecker product. 
\end{defn}

\begin{lem}\label{lem:tuckermat}
    If $\sF_{\bs A}$ is a Tucker map with $\bs A=(A_1,\ldots,A_r)$ then 
    $$
\mat(\sF_{\bs A})\;=\;A_r\otK\cdots \otK A_2\otK A_1.
$$
\end{lem}

\subsection{Self-adjoint mappings}
\label{sec:self-adjoint}

For $\sF\in\mathcal L(\U,\V)$ the adjoint  
$\sF^*\in\mathcal L(\V,\U)$ satisfies
$\langle\sF S,T\rangle=\langle S,\sF^*T\rangle$.
Componentwisely,
$
  \sF^*_{i_1,\ldots,i_r,j_1,\ldots,j_r}
  =\sF_{j_1,\ldots,j_r,i_1,\ldots,i_r}.
$
Denote $\mathbf A^\top=(A_1^\top,\ldots,A_r^\top)$. We then have $\sF^*_{\mathbf A}(T)=\mathbf A^\top\!\cdot T$. A map $\sM:\V\to\V$ is \emph{self-adjoint} when
$\langle\sM T_1,T_2\rangle=\langle T_1,\sM T_2\rangle$.
The identity $\mathbb I_\V$, and compositions such as $\sF\sF^*$, are fundamental examples.

For any linear map $\sM:\V\to \V$ we define its trace and determinant as the trace/determinant of the associated matrix flattening $\tr(\sM)=\tr(\mat(\sM))$, $\det(\sM)=\det(\mat(\sM))$. In particular, using Lemma~\ref{lem:matop}, we easily confirm that $\tr(\sF\sF^*)=\tr(\sF^*\sF)$. Besides, we have the following result.  

\begin{lem}\label{lem:sqSigma}
Suppose that $\sM$ is self-adjoint mapping. Then $\mat(\sM)$ is a symmetric matrix.  Moreover, if $T$ is an eigenvector of $\sM$ with eigenvalue $\lambda$, then $\vec(T)$ is an eigenvector of $\mat(\sM)$ with the same eigenvalue $\lambda$.  
\end{lem}

\subsection{Gaussian tensors}\label{sec:gaussian}

 We say that a random tensor $X \in \V$ folows a normal distribution with mean $\mu \in \V$ and covariance {operator} $\Sigma$, a self-adjoint linear map on $\V$, if, for every $T \in \V$, the real-valued random variable $\langle X, T \rangle$ follows a normal distribution with mean $\langle \mu, T \rangle$ and variance $\langle \Sigma T, T \rangle$. We write $X\sim\mathcal N_\V(\mu,\Sigma)$. This definition aligns with the standard definition of a Gaussian distribution on an arbitrary inner product space; see \cite{eaton1983multivariate}.


\begin{rem}
 Note that $\Sigma$ formally represents a self-adjoint mapping on $\V$ and not a matrix. Vectorizing $X$ translates everything to a more familiar format; see Proposition~\ref{prop:allold} below.
   \end{rem}

When $\Sigma$ is the identity mapping, denoted by $\I_{\s{v}}$, and $\mu = 0$, we say that $X$ follows the standard Gaussian distribution on $\V$. In this case, for any $T \in \V$, we have $\langle X, T \rangle \sim N(0, \|T\|_{\s{f}}^2)$. The following result is standard and is not unique to tensor spaces.
\begin{lem}\label{lem:linearGauss}
Consider a linear mapping $\sF: \U \to \V$. If $Z$ has a Gaussian distribution on $\U$ with mean $\mu \in \U$ and covariance $\Sigma$, then $\sF Z$ has a Gaussian distribution on $\V$ with mean $\sF\mu$ and covariance $\sF \Sigma \sF^*$.
\end{lem}


The following result allows us to rewrite everything in the language of the canonical multivariate Gaussian distribution.

\begin{prop}\label{prop:allold}
    If $X\in \V$ follows Gaussian distribution with mean $\mu\in \V$ and covariance $\Sigma$, then $\vec(X)$ follows Gaussian distribution with mean $\vec(\mu)\in \R^{n}$ and covariance matrix $\mat(\Sigma)$.
 \end{prop}


\section{Tensor probabilistic PCA}\label{sec:tpca}

\subsection{Model setup}

A probabilistic principal component analysis (PPCA) model, as introduced independently in  \cite{tipping1999probabilistic} and Section~7.3 of \cite{anderson1956statistical}, expresses a random vector $X \in \mathbb{R}^n$ as the sum of three components: (i) a deterministic vector $\mu \in \mathbb{R}^n$, (ii) a linear image $AZ$ of a standard normal vector $Z \in \mathbb{R}^m$, where $A \in \mathbb{R}^{n \times m}$ with $m \le  n$, and (iii) an isotropic noise term $\varepsilon \sim N(\mathbf{0}, \sigma^2 I_n)$. Geometrically, the samples of $X$ are concentrated around the $m$-dimensional affine subspace $\mu + \mathrm{Im}(A)$.

We extend this construction to the tensor case, where $X \in \V = \mathbb{R}^{n_1 \times \cdots \times n_r}$ and $Z \in \U = \mathbb{R}^{m_1 \times \cdots \times m_r}$ with $m_k \le  n_k$, for $k = 1, \dots, r$. A direct naive extension is to consider $X\in \V$ that admit stochastic representation 
\begin{equation}\label{eq:naive}
    X\;=\;\mu+\sF(Z)+\eps, \quad \mbox{where } \sF\in \cL(\U,\V).
\end{equation}
This corresponds to the standard PPCA after vectorizing both $X$ and $Z$. The dimension of this model is prohibitively large and also it does not exploit the tensor structure. In this paper, we restrict the class of linear transformations from $Z \in \U$ to $\V$ to be Tucker maps as introduced in Section~\ref{sec:slm}. 

\begin{defn}
\label{def: PPCA_tensor}
The random tensor $X \in \V$ follows a probabilistic tensor PCA model with parameters $\mu \in \V$, $A_1 \in \mathbb{R}^{n_1 \times m_1}$, \dots, $A_r \in \mathbb{R}^{n_r \times m_r}$, and $\sigma \geq 0$ if 
\begin{equation}\label{eq:XPPCA}
X = \mu + \mathbf{A} \cdot  Z + \eps,   
\end{equation}
where $\mu=\bs A\cdot  \nu$ for some $\nu\in \U$, $Z$ has a standard normal distribution in $\U$, and $\eps$ has a Gaussian distribution in $\V$ with mean zero and covariance $\sigma^2 \I_{\s{v}}$, and $\mathbf{A} = (A_1, \ldots, A_r)$. We also assume that $\eps$ and $Z$ are independent.
\end{defn}
\begin{rem}\label{rem:normalization}
Because the core $Z$ is isotropic, we may, without loss of generality, absorb arbitrary orthogonal rotations of its modes into the factor matrices. We hence fix the right singular factors in the SVD of each $A_k$ to be the identity; this removes the mode-wise orthogonal indeterminacy. More details will be provided in Section~\ref{sec:identifiability}.   
\end{rem}

Geometrically, the situation mirrors the general PPCA model. We model $X$ as approximately lying in a lower-dimensional affine subspace of $\V$. Given repeated observations (or a single observation), the goal is to recover $\mu$ and the underlying linear subspace. However, unlike the general PPCA model, we do not consider all subspaces of a given dimension, but only those that arise as the image of the linear mapping $\sF_{\mathbf{A}}$ for some $\mathbf{A} =(A_1,\ldots,A_r)$. The family of such linear maps has dimension $\sum_{k} m_k n_k$, as opposed to $\prod_{k} m_k n_k$, which is the dimension of all linear maps between $\U$ and $\V$. As we will see, constraining the linear transformations to Tucker maps results in stronger identifiability results than in the trivial PPCA model (\ref{eq:naive}). In particular, the family of all linear subspaces of this form is a proper subvariety in the ambient space, which we call the Tucker variety; see Section~\ref{sec:tuckerdim}.

\subsection{Special cases and related models}

The TPCA family subsumes, refines, or parallels many models that have
appeared under different names. We highlight five connections.

\paragraph{Matrix/tensor–normal distributions.}
If we set  $m_k=n_k$ for all $k$, drop the noise
($\sigma=0$), and take $\nu=\mathbf 0$,  
then $X$ in \eqref{eq:XPPCA} becomes
$\mathcal N_\V(0,\Sigma)$ with
\(
  \Sigma=\sF_{\bs A}\sF_{\bs A}^{*}
        =\Sigma_1\otimes\cdots\otimes\Sigma_r,
  \;
  \Sigma_k:=A_kA_k^{\top}.
\)
This is exactly the \emph{tensor–normal} (or matrix–normal when $r=2$)
distribution studied in
\citep{dawid1981tensornormal,lyu2019tensor,hoff2023core, ren2024transfer}.
TPCA therefore extends those models by (i) allowing
$m_k<n_k$ (dimension reduction) and (ii) re-introducing isotropic
noise~$\sigma^2\I_\V$.

\paragraph{Deterministic Tucker–rank models.}
If we set the latent core to be zero,  
the TPCA family splits into two familiar special cases:
\begin{itemize}
  \item[\textit{(i)}]  \emph{Low-Tucker-rank factor model}  
        \citep{chen2024semi}:  
        \(X=\bs A\cdot \nu+\varepsilon\) with isotropic noise
        \(\varepsilon\sim N(0,\sigma^{2}\mathbb I_{\s{v}})\).
  \item[\textit{(ii)}] \emph{Multilinear PCA (MPCA)}  
        \citep{lu2008mpca}:  
        \(X=\bs A\cdot \nu\) \emph{without} an error term
        (\(\sigma=0\) after centring).
\end{itemize}
Thus MPCA is simply the noise-free limit of the low-Tucker-rank factor
model, and TPCA generalises both by re-introducing a \emph{random}
Gaussian core \(Z\) and/or allowing measurement noise
\(\varepsilon\). 

\paragraph{CP-rank factor models.}
Fix a common latent rank $m_0$ so that $m_1=\cdots=m_r=m_0$ and set
$\nu=0$.  
Instead of a \emph{full} Gaussian core $Z\in\R^{m_0\times\cdots\times m_0}$, suppose that $Z$ is diagonal almost surely with independent standard normal diagonal entries $z_1,\ldots,z_{m_0}$.
Choose factor matrices 
$A_k\in \R^{n_k\times m_0}$ with column vectors $A_{k,1},\ldots,A_{k,m_0}\in \R^{n_k}$. We assume $\|A_{k,j}\|=\omega_j^{1/r}$ for all $k=1,\ldots,r$ and $j=1,\ldots,m_0$ allowing us to rewrite $A_{k,j}=\omega_j^{1/r}\tilde A_{k,j}$, where $\tilde A_{k,j}$ has unit norm. With this degenerate core, the TPCA representation \eqref{eq:XPPCA}
reduces to
\[
   X
   \;=\;\bs A\cdot Z+\varepsilon\;=\;
   \sum_{j=1}^{m_0}\omega_j z_j\;
         \tilde A_{1,\,j}\otimes \tilde A_{2,\,j}\otimes\cdots\otimes \tilde A_{r,\,j}
        +\varepsilon,
\]
where the positive weights \(\omega_j\) absorb the usual CP scaling
indeterminacy.
This leads to the CP-rank factor (or multilinear PCA) model studied
in \citet{han2022tensor,ouyang2023multiway}.
If we regard the latent scores $z_j$'s as fixed unknown constants rather
than Gaussian, we obtain the dynamic-factor specification of
\citet{ghysels2024tensor}.  Thus CP-rank models can be viewed as TPCA
with an \emph{extremely sparse} (diagonal) core: all stochastic
variation is concentrated in $m_0$ rank-one directions, whereas TPCA
allows a full Gaussian core and hence richer mode-to-mode interactions.

\paragraph{Bayesian Tucker and CP literature—relation to TPCA.}
A number of papers develop \emph{Bayesian} Tucker or CP decompositions by putting hierarchical priors on the factor matrices (and sometimes on the core or the noise variance); see, e.g., \citet{zhao2015bayesian,kolda2019bayesiantucker, mai2022doubly}.  These works are not equivalent to TPCA.  First, they usually treat the core as deterministic—or endow it with a sparsity-inducing prior—whereas TPCA fixes the core distribution to be i.i.d.\ Gaussian.  Second, their observation models often omit a measurement-noise term or adopt a different error structure, while TPCA explicitly includes isotropic Gaussian noise, \(\varepsilon\sim N(0,\sigma^{2}\mathbb I_\V)\).  In spite of these differences, one can combine our likelihood  with any of the priors proposed in that Bayesian literature; TPCA then acts as a likelihood ``back-end’’ to which those prior layers may be attached, rather than reproducing the same model.

\paragraph{Spiked tensor models.}
For cubic tensors ($n_k=n_0$, $m_k=1$) and $r\ge3$, set
$A_k=\lambda^{1/r}u$ with a unit vector $u\in\R^{n_0}$ and
$\nu=0$.  Then
\(
  X = \lambda\,u^{\otimes r} + \varepsilon,
\)
the symmetric \emph{spiked tensor} model of
\cite{montanari2014statistical,arous2019landscape,wu2024sharp}.
Thus spiked tensors are the rank-one, symmetric boundary point of
TPCA.

\medskip
These links make it easy to import existing estimation or inference
techniques by specialising the TPCA likelihood.

\subsection{Mean and covariance in the probabilistic tensor PCA model}\label{sec:covariance}

Under the proposed probabilistic tensor PCA model $X=\mu+\sF_{\bs A}Z+\eps$, for any $T \in \V$, we have
$$
\<X,T\>\;=\;\<\mu,T\>+\<Z,\sF_{\bs A}^* T\>+\<\eps,T\>,
$$
which is Gaussian as it is a sum of Gaussian variables. We have the following result.
\begin{lem}\label{lem:Sigma}
If $X\in \V$ satisfies \eqref{eq:XPPCA}, then $X$ is Gaussian with mean $\mu$ and covariance
\begin{equation}\label{eq:Sigma}
\Sigma\;=\;\sigma^2 \I_{\s{v}} + \sF_{\bs A}\sF_{\bs A}^*. 
\end{equation}
\end{lem}
\begin{proof}
    Since
$\E\< X,T\>=\< \mu,T\>$, the mean of $X$ is $\mu$. Similarly, 
\begin{eqnarray*}
\var(\< X,T\>)&=&\var(\< Z,\sF_{\bs A}^*T\>)+\var(\< \eps,T\>)\\
&=&\<\sF_{\bs A}^* T,\sF_{\bs A}^* T\>+\sigma^2 \<T,T\>\\
&=& \<\sigma^2 T+\sF_{\bs A}\sF_{\bs A}^* T,T\>, 
\end{eqnarray*}
where the second equality follows from the assumption that both $Z$ and $\eps$ are normal variables in their corresponding subspaces with means zero and covariance operators $\mathbb{I}_{\s{u}}$ and $\sigma^2\I_{\s{v}}$ respectively. By definition, the covariance of $X$ is then given by \eqref{eq:Sigma}. 
\end{proof}


\subsection{Spectrum of the covariance}\label{sec:spectrum}

In this subsection, we study the spectrum of the covariance operation $\Sigma$ defined in (\ref{eq:Sigma}). Let $\sigma_1^{(k)}\geq \cdots\geq \sigma_{m_k}^{(k)}$ be the singular values of $A_k$. The eigenvalues of $B_k: =A_kA_k^\top\in \R^{n_k\times n_k}$ are thus 
\begin{align*}
\lambda_1^{(k)}=(\sigma_1^{(k)})^2\;\geq\; \cdots\;\geq\; \lambda_{m_k}^{(k)}=(\sigma_{m_k}^{(k)})^2\;\geq\; \lambda_{m_k+1}^{(k)}\;=\;\cdots\;=\;\lambda_{n_k}^{(k)}\;=\;0.
\end{align*}

Recall that $\sF_{\bs A}:\U\to \V$ is given by $\sF_{\bs A}(S)=\mathbf A\cdot  S$ and $\sF_{\bs A}^*:\V\to \U$ is given by $\sF_{\bs A}^*(T)=\mathbf A^\top \cdot  T$. In consequence, $\sF_{\bs A} \sF_{\bs A}^*:\V\to \V$ is given by $\sF_{\bs A} \sF_{\bs A}^* T= \mathbf{B}\cdot  T$, where
$$
\mathbf B = (B_1, \ldots,B_r) \;=\;(A_1A_1^\top,\ldots, A_r A_r^\top).
$$

\begin{prop}\label{prop:spectrum}
  Denote by $\mathbf u_1^{(k)},\ldots,\mathbf u_{n_k}^{(k)}$ the eigenvectors of $B_k$ corresponding to eigenvalues $\lambda_1^{(k)},\ldots,\lambda_{n_k}^{(k)}$, that is, $B_k \mathbf u_j^{(k)}=\lambda_j^{(k)}\mathbf u_j^{(k)}$. Then, for any $1\leq i_1\leq n_1$,\ldots, $1\leq i_r\leq n_r$ the rank one tensor $\mathbf u_{i_1}^{(1)}\otimes \cdots \otimes \mathbf u_{i_r}^{(r)}$ satisfies
$$
\sF_{\bs A} \sF_{\bs A}^*(\mathbf u_{i_1}^{(1)}\otimes \cdots \otimes \mathbf u_{i_r}^{(r)})\;=\;\lambda_{i_1}^{(1)}\cdots \lambda_{i_r}^{(r)} \mathbf u_{i_1}^{(1)}\otimes \cdots \otimes \mathbf u_{i_r}^{(r)}.
$$
Consequently, the  tensors $\mathbf u_{i_1}^{(1)}\otimes \cdots \otimes \mathbf u_{i_r}^{(r)}$ form the set of (unit) eigenvectors of $\sF_{\bs A} \sF_{\bs A}^*$.  
\end{prop}
This result is well known (see, e.g., Exercise~4.128 in \cite{hackbusch2012tensor}). We provide proof for completeness.  
\begin{proof} For any $i_1, \ldots , i_r$, we have
    \begin{eqnarray*}
\sF_{\bs A}\sF^*_{\bs A}(\mathbf u_{i_1}^{(1)}\otimes \cdots \otimes \mathbf u_{i_r}^{(r)})&=& \bs B\cdot \mathbf u_{i_1}^{(1)}\otimes \cdots \otimes \mathbf u_{i_r}^{(r)}\\
&=&  (B_1\mathbf u_{i_{1}}^{(1)})\otimes\cdots\otimes (B_r\mathbf u_{i_{r}}^{(r)})\\
&=& \lambda_{i_1}^{(1)}\cdots \lambda_{i_r}^{(r)} \mathbf u_{i_1}^{(1)}\otimes \cdots \otimes \mathbf u_{i_r}^{(r)}.
    \end{eqnarray*}
    To show that these vectors form the set of eigenvectors, we need to show that they are mutually orthogonal. This follows from the fact that 
    $$
    \<\mathbf u_{i_1}^{(1)}\otimes \cdots \otimes \mathbf u_{i_r}^{(r)},\mathbf u_{j_1}^{(1)}\otimes \cdots \otimes \mathbf u_{j_r}^{(r)}\>\;=\; \<\mathbf u_{i_1}^{(1)},\mathbf u_{j_1}^{(1)}\>\cdots \<\mathbf u_{i_r}^{(r)},\mathbf u_{j_r}^{(r)}\>,
    $$
    which is zero unless $i_1=j_1$,\ldots, $i_r=j_r$.
\end{proof}

\begin{cor}\label{cor:specSig}With notations as above, the covariance $\Sigma$ in \eqref{eq:Sigma} has eigenvectors $\mathbf u_{i_1}^{(1)}\otimes \cdots \otimes \mathbf u_{i_r}^{(r)}$,  for $1\leq i_1\leq n_1$, \ldots, $ 1\leq i_r\leq n_r$ and the corresponding eigenvalues are
$$
\sigma^2+\lambda_{i_1}^{(1)}\cdots\lambda_{i_r}^{(r)}.
$$
Among all $n=n_1\cdots n_r$ eigenvalues, there will be $m=m_1\cdots m_r$ greater than or equal to $\sigma^2$ and the remaining ones are equal to $\sigma^2$.
\end{cor}

\subsection{Model identifiability}\label{sec:identifiability}

In this subsection, we study the identifiability of the parameter tuple $(A_1,\ldots,A_r,\nu,\sigma^2)$. Since the distribution of $X$ in the proposed probabilistic tensor PCA model is uniquely given by the mean $\mu=\bs A\cdot  \nu$ and the covariance $\Sigma$ in \eqref{eq:Sigma}, it is enough to study conditions under which the parameters can be uniquely recovered from these two first population moments.

Since $\mu=\bs A\cdot  \nu$ is the mean of $X$, $\bs A\cdot  \nu$ is identifiable while  $\nu$ is not identifiable without further restrictions. This is because $Z$ has the same distribution as $\bs O \cdot Z$ for any collection of orthogonal matrices $\bs{O}=(O_1,\ldots,O_r)$. In consequence, the parameters
$A_1,\ldots,A_r,\nu,\sigma$ and $A_1O_1^\top,\ldots,A_rO_r^\top, \bs{O}\cdot  \nu,\sigma$ define exactly the same distribution in the probabilistic tensor PCA model. We thus make the following assumption; cf. Remark~\ref{rem:normalization}.
\begin{assp}
\label{assp:colorthogonal} 
The right singular matrix in the SVD decomposition of each $A_k$ is the identity matrix. 
\end{assp}
\noindent Note that Assumption~\ref{assp:colorthogonal} does not assume $A_k$'s to have full column ranks. Under Assumption~\ref{assp:colorthogonal}, the spectral decomposition $B_k=U_k \Lambda_k U_k^\top$ identifies $A_k$ via $A_k=U_k \Lambda_k^{1/2}$. 

To identify $\bs B=(B_1,\ldots,B_r)$ and $\sigma$ from the covariance matrix $\Sigma$ in \eqref{eq:Sigma}, note first that the $r$-tuple $\mathbf B=(B_1,\ldots,B_r)$ defines the same linear transformation as the $r$-tuple $(\alpha_1B_1,\ldots,\alpha_r B_r)$ whenever $\alpha_1\cdots \alpha_r=1$. Thus we can only hope to identify $B_i$'s up to some fixed scaling. It is then a natural question whether the parameters $\mathbf B=(B_1,\ldots,B_r)$ in the tensor PPCA model are identifiable {up to scaling}. We fix the scales of $B_i=A_iA_i^\top$'s as follows. 
\begin{assp}
\label{assp:scale}
Assume that 
$\|B_1\|_{\s{f}}=\ldots=\|B_r\|_{\s{f}}$.
\end{assp}
Consider the special case when each $B_k=c_kI_{n_k}$ is a scalar matrix, which implies $c_k\ge 0$ and $m_k = n_k$, for $k \in [r]$. In particular, we need $c_k=\tfrac{c}{\sqrt{n_k}}$ for some $c>0$ to accommodate for Assumption~\ref{assp:scale}. Then, by \eqref{eq:Sigma}, the covariance operator satisfies $\Sigma=(\sigma^2+c_1\cdots c_r)\I_{\s{v}}$. In this case the parameters are not identifiable as there are many pairs $(\sigma,c)$ that give the same value of $\sigma^2+\tfrac{c^r}{\sqrt{n_1\cdots n_r}}$. We thus add another minor condition.
\begin{assp}
\label{assp:nonscalar} 
Either $n_k>m_k$ for at least one $k$ or at least two of the $B_k$'s are not scalar matrices.
\end{assp}

The following theorem shows that the above assumptions are sufficient to ensure the identifiability of the proposed probabilistic tensor PCA model. 

\begin{thm}\label{th:ident}Consider the probabilistic tensor PCA model with $r\geq 2$, $n_k\geq m_k$ for $k \in [r]$ and with parameters satisfying Assumptions~\ref{assp:colorthogonal}, \ref{assp:scale}, and \ref{assp:nonscalar}. Then the parameters $\bs A$ and $\sigma^2$ are identifiable. If, in addition, all $A_k$'s have full column rank then $\nu$ is also identifiable.
\end{thm}
\begin{proof}
Suppose that we have a covariance operator in \eqref{eq:Sigma} that can be equivalently given by two sets of parameters
    $$
    \sigma^2\I_{\s{v}}+\sF_{\mathbf B}\;=\;{\sigma'}^2\I_{\s{v}}+\sF_{\mathbf B'}.
    $$
    Since the eigenvectors of $ \sigma^2\I_{\s{v}}+\sF_{\mathbf B}$ are the same as those of $\sF_{\mathbf B}$, it follows that $\sF_{\mathbf B}$ and $\sF_{\mathbf B'}$ have the same eigenvectors. By Proposition~\ref{prop:spectrum}, these eigenvectors are tensor products of eigenvectors of $B_1$,\ldots,$B_r$. Denote $S^{n_k-1}=\{v\in \R^{n_k}: \|v\|=1\}$ as the unit sphere in $\R^{n_k}$ and consider the mapping
\begin{equation}\label{eq:segre}
    S^{n_1-1}\times \cdots\times S^{n_r-1}\to \R^{n_1\times \cdots \times n_r}\qquad (v^{(1)},\ldots, v^{(r)})\;\mapsto \;v^{(1)}\otimes\cdots\otimes v^{(r)}.    
\end{equation}
Given any point $\mathbf  v=v^{(1)}\otimes\cdots\otimes v^{(r)}$ in the image, we can uniquely identify  each component $v^{(i)}$ up to scaling. Indeed, fix $i_1,\ldots,i_r$ such that $\mathbf v_{i_1,\ldots, i_r}\neq 0$ (there will be always at least one such coordinate). Computing 
$$
\frac{\mathbf v_{j i_2 \cdots i_r}}{\mathbf v_{i_1 i_2 \cdots i_r}}\;=\;\frac{v^{(1)}_{j}v^{(2)}_{i_2}\cdots v^{(r)}_{i_r}}{v^{(1)}_{i_1}v^{(2)}_{i_2}\cdots v^{(r)}_{i_r}}\;=\;\frac{v^{(1)}_{j}}{v^{(1)}_{i_1}},\qquad\mbox{for }j\in [n_1]\setminus \{i_1\},
$$
recovers $v^{(1)}$ up to scaling and the same argument applies to all other components. Note also that since $v^{(k)}\in S^{n_k-1}$, recovering up to scale means recovering up to sign, which implies that the eigenvectors of $B_k$ are equal to the eigenvectors of $B_k'$ for each $k\in [r]$. 

Let $\lambda_{1}^{(k)},\ldots,\lambda_{n_k}^{(k)}$ be the eigenvalues of $B_k$ and ${\lambda_{1}'}^{(k)},\ldots,{\lambda'_{n_k}}^{\!\!\!\!(k)}$ be the eigenvalues of $B_k'$ for $k=1,\ldots,r$. By Corollary~\ref{cor:specSig}, we get 
\begin{equation}\label{eq:eigvalsi}
    \sigma^2+\lambda_{i_1}^{(1)}\cdots \lambda^{(r)}_{i_r}\;=\;{\sigma'}^2+{\lambda_{i_1}'}^{(1)}\cdots {\lambda'_{i_r}}^{\!\!\!(r) } \qquad\mbox{for all }\; i_1 \in  [n_1],\ldots, i_r\in [n_r].
\end{equation}
By Assumption~\ref{assp:nonscalar} we either have $n_k>m_k$ for some $k$ or at least two $B_k$'s are not scalar matrices. Suppose first that $n_{k_0}>m_{k_0}$ for a $k_0 \in [r]$. Then $B_{k_0}$ has some eigenvalues equal to zero and we conclude from \eqref{eq:eigvalsi} that $\sigma=\sigma'$. Arguing as in the first part of the proof, we conclude that we can identify the vector of eigenvalues of $B_k$ up to scaling, that is, $B_k'=c_k B_k$ for all $k \in [r]$. By Assumption~\ref{assp:scale} and the positive semi-definiteness of the $B_k$'s, there exists $c\geq 0$ such that $c_k=c$ for all $k\in [r]$. In this case \eqref{eq:eigvalsi} gives
\begin{equation}\label{eq:auxident}
    \lambda_{i_1}^{(1)}\cdots \lambda^{(r)}_{i_r}(1-c^r)\;=\;0 \qquad\mbox{for all }\;\; i_1\in [n_1],\ldots,  i_r\in [n_r].    
\end{equation}
If $c=1$, then \eqref{eq:auxident} holds and $B_k=B_k'$ for all $k \in [r]$. If $c\neq 1$ then \eqref{eq:auxident} holds only if some $B_k$ is zero. Indeed, because each index tuple can be chosen with non-zero $\lambda_{i_k}^{(k)}$ unless $B_k=0$, the only possibility is that at least one $B_k$ is the zero matrix, contradicting Assumption~\ref{assp:scale}. We conclude that  $B_k=B_k'$ for all $k\in [r]$.

Suppose now that $n_k=m_k$ for $k\in [r]$ but two of the matrices $B_i$, say $B_1$ and $B_2$, are not scalar matrices. From \eqref{eq:eigvalsi}, it follows that for any $i_1,j_1,i_2,\ldots,i_r$,
$$
(\lambda_{i_1}^{(1)}-\lambda_{j_1}^{(1)})\lambda_{i_2}^{(2)}\cdots \lambda^{(r)}_{i_r}\;=\;({\lambda_{i_1}'}^{(1)}-{\lambda_{j_1}'}^{(1)}){\lambda_{i_2}'}^{(2)}\cdots {\lambda_{i_r}'}^{(r)}.
$$
Moreover, as $B_1$ is not a scalar matrix and no $B_k$ is a zero matrix (by Assumption~\ref{assp:scale}), there exist $i_1, j_1$ such that $\lambda^{(1)}_{i_1}> \lambda^{(1)}_{j_1}$ and $i_2,\ldots,i_r$ such that $\lambda_{i_2}^{(2)}\cdots \lambda^{(r)}_{i_r}\neq 0$. From this it follows that for every $l_2=1,\ldots,n_2$, 
$$
\frac{\lambda^{(2)}_{l_2}}{\lambda^{(2)}_{i_2}}\;=\;\frac{{\lambda_{l_2}'}^{(2)}}{{\lambda_{i_2}'}^{(2)}}
$$
showing that $B_2'=c_2 B_2$ for some $c_2>0$. Similarly, $B_k'=c_k B_k$ for all $k=3,\ldots,r$. By Assumption~\ref{assp:scale}, we also conclude $c_2=\ldots=c_r=:c$. Since $B_2$ is also non-scalar, we can use a similar argument to conclude $B_1'=cB_1$. In this case \eqref{eq:eigvalsi} gives
$$
    \lambda_{i_1}^{(1)}\cdots \lambda^{(r)}_{i_r}(1-c^r)\;=\;{\sigma'}^2-\sigma^2 \qquad\mbox{for all }\;\; i_1\in [n_1],\ldots,  i_r\in [n_r].
$$
As not all eigenvalues of $B_1$ are the same, this is only possible if $c=1$ and $\sigma^2={\sigma'}^2$.

Finally, recall that $\bs A\cdot  \nu$ is always identifiable. Under Assumption~\ref{assp:colorthogonal}, identifying all $B_i$'s uniquely allows us to identify all $A_i$'s uniquely. If all $A_i$'s have full rank, the kernel of $\sF_{\bs A}$ is trivial and so $\nu$ is identifiable too. 
\end{proof}

\subsection{Tucker variety and model dimension}\label{sec:tuckerdim}

For a tensor $T\in \V$ denote by ${\rm rank}(T)=(d_1,\ldots,d_r)$ as its multilinear rank, also known as Tucker rank, where $d_k$ is the rank of the $n_k \times \prod_{l \ne k} n_l$ matrix whose columns are all the $k$-th mode fibers of $T$. The Tucker manifold $\cT(d_1,\ldots,d_r)\subseteq \V$ is the set of all tensors $T\in \V$ of the fixed multilinear rank. For the proof of the next result see Section~2 in \cite{kressner2014low}.
\begin{lem}\label{lem:kressner}
    The Tucker manifold $\cT(d_1,\ldots,d_r)$ is a smooth manifold and its dimension is $\sum_k (d_k n_k-d_k^2)+\prod_{k} d_k$. If $T\in \cT(d_1,\ldots,d_r)$,  then $T$ possesses the so-called Tucker decomposition $T=\bm{U}\cdot  S$ for some $S\in \R^{d_1\times \cdots \times d_r}$ and $\bs U=(U_1,\ldots,U_r)$ with $U_k\in \R^{n_k \times d_k}$ satisfying $U_k^\top U_k = I_{d_k}$, for $k \in [r]$.
\end{lem}
Here we also consider the Tucker variety $\overline{\cT}(d_1,\ldots,d_r)$, defined as
$$
\overline{\cT}(d_1,\ldots,d_r)\;:=\;\bigcup_{d_1'\leq d_1,\ldots, d_r'\leq d_r} {\cT}(d_1',\ldots,d_r'),
$$
which is obtained by taking the closure  of ${\cT}(d_1,\ldots,d_r)$. The Tucker variety is known in algebraic geometry as the subspace variety; see Section 3.4.1 in \cite{landsberg2011tensors}. This results in a non-smooth algebraic variety with defining equations given by setting all $(d_k+1)\times (d_k+1)$ minors of the $k$-th mode matricization of the underlying tensor to be $0$ for every $k$ such that $d_k<n_k$. The dimension of the Tucker variety is equal to the dimension of $\cT(d_1,\ldots,d_r)$. 

Lemma~\ref{lem:kressner} suggests a link between Tucker manifolds and our construction. The exact relation is more subtle.
\begin{prop}\label{prop:tuckervar}
    Consider the set 
    $\mathcal{M}(m_1,\ldots,m_r)\;=\;\{T\in \V:\;T=\bs A\cdot  S, \;S\in \U\}$.
    Then 
    $$
\mathcal{M}(m_1,\ldots,m_r)\;=\;\overline{\cT}(d_1,\ldots,d_r),
    $$
    where $d_k=\min\{m_k,\prod_{l \neq k}m_l\}$. In particular,
    $$
    \dim(\mathcal{M}(m_1,\ldots,m_r))\;=\;\sum_k (d_k n_k-d_k^2)+\prod_{k}d_k.
    $$
\end{prop}
\begin{proof}
    If $T\in \mathcal{M}(m_1,\ldots,m_r)$ then its $1$-st mode matricization $T_{(1)}$ satisfies 
    $$
    T_{(1)}\;=\;A_1 S_{(1)} (A_r\otK \cdots\otK A_2)^\top
    $$
    and it has rank at most $\min\{m_1,\prod_{j\neq 1} m_j\}=d_1$. Considering other modes in the same way allows us to conclude that $\mathcal{M}(m_1,\ldots,m_r)\subseteq \overline{\cT}(d_1,\ldots,d_r)$. 
    
    To prove the opposite inclusion, suppose $T\in \overline{\cT}(d_1,\ldots,d_r)$. Then $T\in {\cT}(d_1',\ldots,d_r')$ for some $d_1'\leq d_1$, \ldots, $d_r'\leq d_r$. By Lemma~\ref{lem:kressner}, $T=\bs U'\cdot  S'$ where $S'\in \R^{d_1'\times\cdots \times d_r'}$ and $U_k'\in \R^{n_k\times d_k'}$ is a column orthogonal matrix, for $k \in [r]$. Since $d_k'\leq d_k \le  m_k$ we can complete each $U_k'$ to $\tilde{U}_k\in \R^{n_k\times m_k}$ (append an orthonormal complement) and embed $S'$ to $S\in \U$ by adding zero entries so that $T=\bs A\cdot  S$ showing that $T\in \mathcal{M}(m_1,\ldots,m_r)$.
\end{proof}

The following example will reappear later.
\begin{ex}\label{ex:simple}
    Consider order three tensors with $(n_1,n_2,n_3)=(6,3,3)$ and $(m_1,m_2,m_3)=(5,2,2)$. Then $(d_1,d_2,d_3)=(4,2,2)$, and thus $\dim(\mathcal{M}(5,2,2))=\dim(\cT(4, 2, 2))=28$ with codimension $54-28=26$.
\end{ex}

Note that, by construction, the tensors in the probabilistic Tensor PCA model, concentrate around the corresponding Tucker variety whose dimension is provided in Proposition~\ref{prop:tuckervar}. We next contrast this dimension formula with the dimension of the tensor PCA model. 
\begin{prop}
    The dimension of the probabilistic tensor PCA model in (\ref{eq:XPPCA}) is 
    $$
    2+\sum_{k=1}^r \left(n_k m_k-\binom{m_k}{2}-1\right)\qquad\mbox{or }\qquad 2+\sum_{k=1}^r \left(n_k m_k-\binom{m_k}{2}-1\right)+\prod_{k=1}^r m_k
    $$
    depending on whether the mean tensor $\nu$ is assumed to be zero or not. 
\end{prop}
\begin{proof} Suppose first that $\nu=0$. Consider the mapping from the parameter space to the space of the underlying distributions. This mapping is not one to one. The proof of Theorem~\ref{th:ident} implies that this mapping becomes one to one if we restrict this mapping to the set of $(\bs A,\nu,\sigma^2)$ satisfying Assumptions~\ref{assp:colorthogonal}, \ref{assp:scale}, and \ref{assp:nonscalar}. The set of matrices satisfying Assumption~\ref{assp:colorthogonal} has dimension $\sum_{k=1}^r(m_kn_k-\binom{m_k}{2})$. Fixing the scale in Assumption~\ref{assp:scale} further reduces this dimension by $r-1$ as we employ $r-1$ constraints $\|B_k\|_{\s{f}}=\|B_1\|_{\s{f}}$ for $k=2,\ldots,r$. Adding the additional dimension coming from the choice of $\sigma^2$ gives the first dimension formula. If $\nu$ is not fixed, an additional
$\prod_{k=1}^{r} m_k$ coordinates are free,
yielding the second dimension formula.
\end{proof}

For fair comparison we consider the deterministic
low-Tucker-rank factor model $X=\boldsymbol A\cdot\nu+\varepsilon$
\citep{chen2024semi}.  Assume $m_k\le\prod_{l\ne k}m_l$, so
the multilinear rank $(m_1,\dots,m_r)$ would be possible.  If the columns of each
$A_k$ are orthonormal, its parameter space has dimension
\[
  1+\sum_{k=1}^{r}\left(n_k m_k-\binom{m_k+1}{2}\right)
    +\prod_{k=1}^{r} m_k .
\]
Because TPCA augments this model with a
random Gaussian core $\bs A\!\cdot\!Z$, the version with a free mean
$\nu$ contains
$\sum_{k=1}^{r}m_k+1-r$
more parameters.
When $\nu$ is fixed to zero, TPCA instead has
$\prod_{k}m_k-\sum_{k}m_k+r-1$
fewer parameters than the deterministic factor model; the gap
can grow exponentially in the tensor order~$r$.

\section{Maximum likelihood estimator}\label{sec:EM}

Given that $X\in \V$ is Gaussian with mean $\mu$ and covariance $\Sigma$ satisfying \eqref{eq:Sigma}, it is natural to estimate the parameters $\nu,\sigma^2$, $B_1,\ldots,B_r$ in the tensor PPCA model by maximizing the log-likelihood of the data $X^{(1)},\ldots,X^{(N)}$, where $N$ is the sample size. Let $\sS_N=\tfrac1N\sum_{i=1}^N (X^{(i)} - \mu)\otimes (X^{(i)}-\mu)\in \cL(\V,\V)$. Then the Gaussian log-likelihood (up to some additive constants) is simply
\begin{equation}\label{eq:llike}
    \tfrac{N}{2}\log\det \Sigma^{-1}-\tfrac{N}{2}\tr(\sS_N \Sigma^{-1}),
\end{equation}
where we also recall from Section~\ref{sec:self-adjoint} that the trace and determinant on $\mathcal L(\V, \V)$ is defined through the mapping $\mat$.

\subsection{EM algorithm for tensor PPCA}

In this section we discuss the general EM algorithm for the tensor PPCA. Large parts of this follow the standard results; e.g. Section~12.2.2 in \cite{bishop2006pattern}. We use the fact that the joint distribution on $(X,Z)$ can be easily specified by first specifying the distribution of $Z$ and then the conditional distribution of $X$ given $Z$. The standard expression for the conditional mean of $Z$ given $X$ becomes
$$
\E(Z|X)\;=\;\Sigma_{\s{ZX}} \Sigma_{\s{XX}}^{-1}(X-\mu),
$$ 
where $\Sigma_{\s{XX}}$ is just $\Sigma$ given in \eqref{eq:Sigma} and $\mu=\bs A\cdot  \nu$. Writing $\Sigma_{\s{ZX}} \Sigma_{\s{XX}}^{-1}$ we mean a composition of two linear maps: a map $\Sigma_{\s{XX}}^{-1}\in \mathcal L(\V,\V)$, which is the inverse of $\Sigma_{\s{XX}}$ and the map $\Sigma_{\s{ZX}}\in \mathcal L(\V,\U)$. Moreover, for any $S\in \U$ and $T\in \V$
$$
\<\Sigma_{\s{ZX}}T,S\>\;=\;\cov(\<X,T\>,\<Z,S\>)\;=\;\cov(\<Z,\sF_{\bs A}^* T\>,\<Z,S\>)\;=\;\<\sF_{\bs A}^*  T,S\>,
$$
where we used the fact that $\Sigma_{\s{ZZ}}=\I_{\s{u}}$. In other words $\Sigma_{\s{ZX}}=F^*_{\mathbf A}$. We also note that the inverse of $\Sigma_{\s{XX}}=\sF_{\bs A}\sF^*_{\mathbf A} + \sigma^2 \I_{\s{v}}$ can be efficiently computed using the Woodbury identity. We have
\begin{equation}\label{eq:SXXinv}
\Sigma_{\s{XX}}^{-1}\;=\;(\sF_{\bs A}\sF^*_{\mathbf A} + \sigma^2 \I_{\s{v}})^{-1}\;=\;\tfrac{1}{\sigma^2}\left[\I_{\s{v}}-\sF_{\bs A} (\sF_{\bs A}^*\sF_{\bs A}+\sigma^2 \I_{\s{u}})^{-1}\sF_{\bs A}^*\right],    
\end{equation}
and the advantage of the expression on the right is that inverting the transformation 
\begin{equation}\label{eq:M_definition}
\sM\;:=\;\sF_{\bs A}^*\sF_{\bs A}+\sigma^2\I_{\s{u}}    
\end{equation}
can be done more efficiently than inverting the map $\sF_{\bs A}\sF_{\bs A}^* + \sigma^2 \I_{\s{v}}$ when $\prod_{k} m_k$ is much smaller than $\prod_{k} n_k$. 

With these reformulations, we can write
\begin{equation}\label{eq:EZX}
    \E(Z|X)\;=\;\sM^{-1} \sF_{\bs A}^*(X-\mu).
\end{equation}
Moreover, 
\begin{eqnarray*}
\var(Z|X)&=& \Sigma_{\s{ZZ}}-\Sigma_{\s{ZX}}\Sigma_{\s{XX}}^{-1}\Sigma_{\s{XZ}}\\
&=&\I_{\s{u}}-\sF^*_{\bs A}\left[\frac{1}{\sigma^2}(\I_{\s{v}}-\sF_{\bs A}\sM^{-1}\sF^*_{\bs A})\right]\sF_{\bs A}\\
&=&\I_{\s{u}}-\frac{1}{\sigma^2}\sF^*_{\bs A}\sF_{\bs A}+\frac{1}{\sigma^2}\sF^*_{\bs A}\sF_{\bs A}\sM^{-1}\sF^*_{\bs A}\sF_{\bs A}\\
&=&\I_{\s{u}}-\frac{1}{\sigma^2}\sM\sM^{-1}\sF^*_{\bs A}\sF_{\bs A}+\frac{1}{\sigma^2}\sF^*_{\bs A}\sF_{\bs A}\sM^{-1}\sF^*_{\bs A}\sF_{\bs A}\\
&=& \I_{\s{u}}-\sM^{-1}\sF^*_{\bs A}\sF_{\bs A}\;=\;\sigma^2 \sM^{-1}.
\end{eqnarray*}

The EM-algorithm starts by fixing some starting values for $\sigma^2_0, A_{0,1},\ldots,A_{0,r},\nu_0$ with $\mu_0=\bs A_0\cdot  \nu_0$. Then we proceed iteratively.

First, in the E-step, we compute $\sM_0$ using \eqref{eq:M_definition}, the conditional means from \eqref{eq:EZX}
\begin{equation}\label{eq:E1}
\bs E_i\;:=\;\E_{\theta_0}(Z^{(i)}|X^{(i)})\;=\;\sM^{-1}_0 \sF_{\mathbf A_0}^*(X^{(i)}-\mu_0)    \;\in \;\U,
\end{equation}
and the conditional covariance from the formula for $\var(Z|X)$ above
\begin{equation}\label{eq:E2}
\sV\;:=\;\sigma^2_0 \sM_0^{-1}\;\in\;\mathcal L(\U,\U).
\end{equation}
\begin{rem}\label{rem:impl-reg}
    Note that $\bs E_i$ can be equivalently characterized as the optimal $U$ in
    $$
    \|X^{(i)}-\mu_0-\sF_{\bs A_0}U\|^2_{\s{f}}+\sigma^2_0 \|U\|^2_{\s{f}},
    $$
    which corresponds to a ridge regression problem linking our problem to regularized Tucker decomposition.
\end{rem}
Second, in the M-step, we update $\mu$ and $\bs A$ by minimizing 
\begin{equation}\label{eq:M}
    \tfrac1N\sum_{i=1}^N\|X^{(i)}-\mu-\sF_{\bs A} \bs E_i\|^2_{\s{f}}+\tr(\sF_{\bs A}\sV \sF_{\bs A}^*).
\end{equation}
 Note that (\ref{eq:M}) is a block multi-convex function. Precisely, as a function of a single $A_k$, this is a strictly convex function. It admits a closed form optimizer as for the generalized ridge regression. Once $\bs A$ is updated, we update $\nu$ with the formula
\begin{equation}\label{eq:nu}
\nu\;=\;(\sF^*_{\bs A} \sF_{\bs A})^{-1}\sF^*_{\bs A}\left(\frac1N \sum_{i=1}^N (X^{(i)}-\sF_{\bs A}\bs E_i)\right).    
\end{equation}

Finally, $\sigma^2$ is updated via
\begin{equation}\label{eq:M2}
\sigma^2\;=\;\frac{1}{nN} \sum_{i=1}^N\|X^{(i)}-\mu-\sF_{\bs A} \bs E_i\|_{\s{f}}^2+\frac1n \tr(\sF_{\bs A}\sV \sF^*_{\bs A}),
\end{equation}
where $n=n_1\cdots n_r$ is the dimension of $\V$.

\begin{thm}
    The proposed procedure defined by equations \eqref{eq:E1}--\eqref{eq:M2} performs the EM algorithm for the proposed probabilistic tensor PCA.
\end{thm}
\begin{proof}The calculations are similar to classical derivations; see Section~12.2.2 in \cite{bishop2006pattern}.
    The joint distribution of $X,Z$ is conveniently expressed in terms of the marginal distribution of $Z$ and the conditional distribution of $X$ given $Z$: $Z\sim N(0,\I_{\s{u}})$ and $X|Z\sim N\left(\sF_{\bs A} (\nu+Z),\sigma^2 \I_{\s{v}}\right)$. Both distributions have isotropic covariances, which gives the joint log-likelihood of the form
$$
\log p(X,Z)\;=\;{\rm const}-\tfrac{1}{2}\|Z\|_{\s{f}}^2-\tfrac{n}{2}\log \sigma^2-\tfrac{1}{2\sigma^2}\|X-\mu-\sF_{\bs A} (Z)\|_{\s{f}}^2,
$$
where the term ${\rm const}$ is actually $-\frac{(n+m)}{2}\log (2\pi)$ with $m = m_1 \ldots m_r$ and $n = n_1 \ldots n_r$, and it does not depend on the data or parameters $\theta=(\sigma^2,\nu,\bs A)$. The joint log-likelihood is then given as $\ell(\theta)=\sum_{i=1}^N\log p(X^{(i)},Z^{(i)})$. The conditional expectation of the log-likelihood, given the observed data $\mathcal D=\{X^{(1)},\ldots,X^{(N)}\}$ {and a fixed parameter vector $\theta_0=(\sigma_0^2,\nu_0,\bs A_0)$} satisfies
\begin{eqnarray*}
\E_{\theta_0}(\ell|\mathcal D)&=&{\rm const}-\tfrac12\sum_{i=1}^N\E_{\theta_0}\left(\|Z^{(i)}\|_{\s{f}}^2|X^{(i)}\right)-\tfrac{nN}{2}\log \sigma^2-\tfrac{1}{2\sigma^2}\sum_{i=1}^N\|X^{(i)}-\mu\|_{\s{f}}^2\\
&+&\tfrac{1}{\sigma^2}\sum_{i=1}^N \left\<X^{(i)}-\mu,\sF_{\bs A} \E_{\theta_0}(Z^{(i)}|X^{(i)})\right\> - \tfrac{1}{2\sigma^2}\sum_{i=1}^N\E_{\theta_0} \left(\|\sF_{\bs A} Z^{(i)}\|_{\s{f}}^2|X^{(i)}\right).  
\end{eqnarray*}
 Using basic algebra, we get
\begin{eqnarray*}
\E_{\theta_0}\left(\|Z\|_{\s{f}}^2|X\right)&=& \E_{\theta_0}\left(\|Z-\E_{\theta_0}(Z|X)\|_{\s{f}}^2|X\right)+\left\|\E_{\theta_0}(Z|X)\right\|_{\s{f}}^2\\
&=& \tr\left(\var_{\theta_0}(Z|X)\right)+\left\|\E_{\theta_0}(Z|X)\right\|_{\s{f}}^2.
\end{eqnarray*}
Moreover, it holds that
\begin{eqnarray*}
\E_{\theta_0}\left(\|\sF_{\bs A} Z\|_{\s{f}}^2|X\right)&=&\E_{\theta_0}\left(\<\sF_{\bs A} Z,\sF_{\bs A} Z\>|X\right)
\\&=&\tr\left(\sF_{\bs A} \var_{\theta_0}(Z|X)\sF_{\bs A}^*\right)+\left\|\sF_{\bs A} \E_{\theta_0}(Z|X)\right\|_{\s{f}}^2.
\end{eqnarray*}
This allows to rewrite the expected full log-likelihood as
\begin{eqnarray*}
\E_{\theta_0}(\ell|\mathcal D)&=&{\rm const}-\tfrac12\sum_{i=1}^N\|\bs E_i\|_{\s{f}}^2-\tfrac{nN}{2}\log \sigma^2-\tfrac{1}{2\sigma^2}\sum_{i=1}^N\|X^{(i)}-\mu-\sF_{\bs A} \bs E_i\|_{\s{f}}^2\\
&-&\tfrac{N}{2}\tr(\sV)-\tfrac{N}{2\sigma^2}\tr(\sF_{\bs A}\sV \sF_{\bs A}^*). 
\end{eqnarray*}
Optimizing this with respect to $\bs A$ is equivalent to minimizing \eqref{eq:M}. Optimizing with respect to $\sigma^2$ leads to the formula in \eqref{eq:M2}.
\end{proof}
We provide some compotational details for implementing the EM algorithm in Appendix~\ref{sec:comp_details}.

\subsection{Existence of the MLE}

We start by noting that in the standard probabilistic PCA (PPCA) applied to tensor data (after vectorization), the sample size needed for the existence of the MLE may be large.
\begin{prop}
    The MLE for standard PPCA model in \eqref{eq:naive} exists (with probability one) if and only if the sample size $N$ satisfies $N>m_1\cdots m_r$.
\end{prop}
\begin{proof}
This PPCA problem corresponds to the standard PPCA with $x\in \R^n$, $x=\mu+Az+\epsilon$, $z\sim N_m(0, I_m)$ and $\epsilon \sim N_n(0, \sigma^2 I_n)$. Where $n=n_1\cdots n_r$ and $m=m_1\cdots m_r$. Let $\sS_N$ be the sample covariance matrix.     The MLE is given in closed form; see \cite{tipping1999probabilistic} or Section~7.3 of \cite{anderson1956statistical}. In particular, $\widehat\mu$ is the sample average, and $\widehat\sigma^2$ is obtained as the average of $n-m$ smallest eigenvalues of the sample covariance matrix. This is generically positive if and only if $N>m$. The MLE for all other parameters is always well-defined.  
\end{proof}

On the other hand, under mild conditions on the underlying dimensions, in the proposed probabilistic tensor PCA model, the MLE exists even with a single sample.

\begin{thm}\label{th:MLE1}
   Suppose $n_k>m_k$ for at least one $k \in [r]$. Then a single sample is enough for the MLE in the proposed probabilistic tensor PCA model to exists with probability 1. 
\end{thm}
\begin{proof}
    Let $X\in \V$ be the single observation from our model. By Proposition~\ref{prop:tuckervar}, the Tucker variety $\overline{\cT}(d_1,\ldots,d_r)$ describes all tensors that can be written as $\bs A\cdot  U$ for some $\bs A$ and some $U\in \U$. Here $d_k=\min\{m_k,\prod_{l \neq k}m_l\}$. If one $m_k<n_k$, then $d_k<n_k$ and it is always a proper subvariety of $\V$; see Section~\ref{sec:tuckerdim}. Therefore, with probability $1$, a random $X$ will lie outside of this variety. Recall that the EM-algorithm is a hill climbing procedure. The update \eqref{eq:M2} implies that without loss of generality, optimizing the likelihood, we can assume
    $$
    \sigma^2\;\geq\;\frac1n {\rm dist}^2\left(X,\overline{\cT}(d_1,\ldots,d_r)\right)\;>\;0.
    $$
    Denote this bound by $\sigma_*^2$. This shows that the EM algorithm can restrict $\widehat{\Sigma}$ to the set where $\Sigma\succeq \sigma_*^2\I_{\s{v}}$ (recall that eigenvalues of $\Sigma$ are greater or equal to $\sigma^2$ by Corollary~\ref{cor:specSig}). Since this condition defines a closed (random) subset of $\S^n_+$, the resulting values of the parameters $(\nu,\bs A,\sigma^2)$ also lie in the closed set as the  mapping \eqref{eq:Sigma} from the parameters to $\Sigma$ is continuous. 
    Since this set is not bounded, the optimum could still not be achieved. To conclude that the MLE exists, we note that $\|\Sigma\|\to\infty$ implies that the log-likelihood in \eqref{eq:llike} (here $N=1$) diverges to $-\infty$. Indeed, if the maximal eigenvalue of $\Sigma$ diverges (recall that the minimal eigenvalue is guaranteed to be at least $\sigma^2_*$) then $\log\det\Sigma^{-1}\to -\infty$ while  the term $-\tr(\sS_N\Sigma^{-1})$ is always non-positive. Similarly, if $\|\mu\|\to \infty$, then $-\tr(\sS_N\Sigma^{-1})\to -\infty$ almost surely and $\log\det\Sigma^{-1}$ is bounded above.
\end{proof}

The problem of existence is separate from the problem of the quality of the maximum likelihood estimator. For small or moderate sample sizes we expect that the MLE is not a good estimator and some forms of regularization must be necessary. We also note that for small samples the MLE may be quite degenerate.

\begin{ex}
Recall the toy setting of Example~\ref{ex:simple}:
$(n_1,n_2,n_3)=(6,3,3)$ and $(m_1,m_2,m_3)=(5,2,2)$.
We drew one or more i.i.d.\ observations from the TPCA model with
$\nu=0$ and factor matrices
$A_1\in\mathbb R^{6\times5}$, $A_2,A_3\in\mathbb R^{3\times2}$
generated at random (columns orthogonal, then rescaled). With a \textbf{single} tensor ($N=1$) the numerical maximizer
$\widehat A_1$ produced by the EM algorithm always had
${\rm rank}(\widehat A_1)=3$ in 1,000 independent trials.
Once a second sample was added ($N\ge2$) the estimated rank jumped to
the full possible value~$5$ almost surely. From the closed-form update in appendix \ref{sec:comp_details},
\[
\widehat A_k
=\Bigl(\tfrac1N\sum_{i=1}^{N}\cM_{k}(X^{(i)})\,U_{k,i}^{\top}\Bigr)
\Bigl(
W_k+\tfrac1N\sum_{i=1}^{N}U_{k,i}U_{k,i}^{\top}
\Bigr)^{-1}
\qquad\text{(see~\eqref{eq: solution_A_k})},
\]
take $k=1$ and $N=1$.  Here
$U_{1,1}=\cM_{1}\bigl(\bs A^{(-1)}E_{1}\bigr)$ is a
$5\times(3\cdot3)=5\times9$ matrix, but
$\operatorname{rank}(U_{1,1})\le m_2 m_3 = 4$.
Consequently the first factor in the above display has rank at most~$4$, hence
$\operatorname{rank}(\widehat A_1)\le4$.  The simulations suggest that the generic rank is
in fact~$3$. Understanding precisely \emph{which} ranks occur for small $N$,
and how they depend on $(n_k,m_k)$'s, seems an interesting open
problem.
\end{ex}

\section{Rank-1 power iteration algorithm}\label{sec:algorithm}

In this section, we propose a novel rank-1 power iteration algorithm as an alternative procedure for estimating the model parameters. The main motivation is for efficient computation while addressing the potential computational-statistical gap. The proposed algorithm fully exploits both the low-rank structure of $\pair(\Sigma)$ as a tensor (cf.\ Section~\ref{sec:tensorcomp}) and the low-rank structure of its factors as matrices after appropriate reshaping. From the implementation perspective, it directly leverages the low-rank decomposition of the sample Kronecker covariance tensor, making it computationally efficient. Similar to spectral methods, it is a model-free procedure, but unlike those methods, it yields estimators of the $\widehat{A}_k$’s themselves rather than just their column spaces. Theoretically, we show that under the probabilistic tensor PCA model, this augmented rank-1 power iteration algorithm converges and provides consistent estimators—even after the first iteration. For ease of presentation, we consider the case $\mu = 0$ in this section.

\subsection{Rank-1 power iteration}

\textbf{Exploiting the structure in $\pair(\Sigma)$.} 
To begin with, we take a closer look at the covariance in the probabilistic tensor PCA model. 
\begin{prop}
\label{prop: Sigma}
    Suppose $\Sigma\in \mathcal L(\V,\V)$ is a covariance for the probabilistic tensor PCA model with parameters $\sigma^2,\bs A=(A_1,\ldots,A_r)$, then $\pair(\Sigma)\in \R^{n_1^2\times \cdots \times n_r^2}$ satisfies
    $$
    \pair(\Sigma)\;=\;  \vec(A_1A_1^\top)\otimes \cdots\otimes \vec(A_rA_r^\top)+\sigma^2 \vec(I_{n_1})\otimes \cdots\otimes \vec(I_{n_r}).
    $$
\end{prop}
\begin{proof}
    This follows immediately from Proposition~\ref{prop:spmap} and its proof.
\end{proof}

Proposition~\ref{prop: Sigma} states that $\pair(\Sigma)$ admits a rank-2 CP decomposition, where the first component captures the Tucker map in the tensor probabilistic PCA model, while the second component characterizes the noise level. Intuitively, if the entries of $A_k$'s are of constant order, the Frobenius norm of the first component would be much larger than that of the second one due to the highly sparsity of the second component. This motivates us to estimate the parameters $\mathbf{A} = (A_1, \ldots, A_r)$ from the rank-1 approximation of the noisy version of $\pair(\Sigma)$.  

Recall that we denote $B_k=A_kA_k^\top$ for all $k=1,\ldots,r$. Moreover, $\bs B=(B_1,\ldots,B_r)$ and $\sF_{\bs B}=\sF_{\bs A}\sF_{\bs A}^*$. An important geometric property implied by this model is that both components in the decomposition of $\Sigma$ in \eqref{eq:Sigma} are nearly orthogonal in the regime when $n_k \gg m_k$. Indeed, note that, from Cauchy-Schwarz inequality, for every $k=1,\ldots,r$, 
\begin{equation}\label{eq:lowrankFbound}
\langle \frac{B_k}{\|B_k\|_{\s{f}}}, \frac{I_{n_k}}{\|I_{n_k}\|_{\s{f}}}\rangle \;=\; \frac{\tr(B_k)}{\|B_k\|_{\s{f}} \sqrt{n_k}} \;\le\; \sqrt{\frac{m_k}{n_k}}. 
\end{equation}
As a result, the cosine between the $\sF_{\bs B}$ and $\sigma^2\I_{\s{v}}$ is provably small: 
$$
\frac{\<\sF_{\bs B},\sigma^2\I_{\s{v}}\>}{\|\sF_{\bs B}\|_{\s{f}}\|\sigma^2\I_{\s{v}}\|_{\s{f}}}\;=\;\frac{\sigma^2\tr(\sF_{\bs B})}{\|\sF_{\bs B}\|_{\s{f}}\sigma^2 \sqrt{n} }\;=\;\prod_{k=1}^r \frac{\tr(B_k)}{\|B_k\|_{\s{f}}\sqrt{n_k}}
\;\leq\;\prod_{k=1}^r \sqrt{\frac{m_k}{n_k}}.
$$

\noindent\textbf{Input of the algorithm.} Denote $\sS_N = \frac{1}{N} \sum_{i = 1}^N X_i \otimes X_i\in \mathcal L(\V,\V)$ to be the sample covariance. For a covariance $\Sigma$ in \eqref{eq:Sigma}, write
\begin{equation}\label{eq:SNandPhi}
    \sS_N\;=\;\Sigma+(\sS_N-\Sigma)\;=\;\sF_{\bs B}+\Phi,\qquad \Phi:=\sigma^2\I_{\s{v}}+\sS_N-\Sigma.    
\end{equation}
Denote $b_k=\vec(B_k)/\|B_k\|_{\s{f}}$ for $k=1,\ldots,r$ to be the unit vector obtained from $B_k$ by vectorizing and divinding by its norm. By Proposition~\ref{prop: Sigma}, we can write 
\begin{equation}\label{eq:S_Ndecomp}
\pair(\sS_N) \;=\; \prod_{k=1}^r \|B_k\|_{\s{f}} \;b_1\otimes \cdots\otimes b_r+\pair(\Phi). 
\end{equation}
If $\sS_N$ is good estimator of $\Sigma$ in a sense that will be made clear later, then the ``signal'' term $\prod_{k=1}^r \|B_k\|_{\s{f}}\;b_1 \otimes \ldots \otimes b_r$ dominates. A detailed concentration study of $\sS_N - \Sigma$ is provided in Section~\ref{sec:theory}. 
\medskip

\noindent\textbf{A single iteration.} We run a rank-1 power iteration over $\pair(\sS_N)$ to estimate the $b_k$'s. Suppose the current estimates $b_k^{(l)}=\vec(B_k^{(l)})$ satisfy $\|b_k^{(l)}\| = 1$ and $B_k^{(l)} \succeq 0$ with $\mathrm{rank}(B_k^{(l)}) \le m_k$ for all $k\in [r]$. In the $(l+1)$-st iteration we first compute, for each $k\in[r]$,
\begin{align}
\label{equ: b_k_l_tilde}
\tilde b_k \;=\; \bigl({b_1^{(l)}}^\top, \ldots, {b_{k-1}^{(l)}}^\top, I_{n_k^2}, {b_{k+1}^{(l)}}^\top, \ldots , {b_r^{(l)}}^\top \bigr)\cdot \pair(\sS_N),
\end{align}
and set $\tilde B_k=\vec^{-1}(\tilde b_k)$. We then take the best rank-$m_k$ truncation of $\tilde B_k$ and re-normalize:
\begin{align}
\label{equ: b_k_l_hat}
B_k^{(l+1)} \;=\; \frac{ \mathrm{Eig}_{m_k}(\tilde B_k) }{\bigl\|\mathrm{Eig}_{m_k}(\tilde B_k)\bigr\|_{\s{f}}}
\qquad\text{and}\qquad 
b_k^{(l+1)} := \vec\!\bigl(B_k^{(l+1)}\bigr),
\end{align}
where $\mathrm{Eig}_{m_k}(\tilde B_k)$ keeps the top $m_k$ eigenpairs of $\tilde B_k$ and zeroes out the rest.

\medskip
The next lemma shows that $\tilde B_k$ is automatically positive semi-definite (PSD), so as its top-$m_k$ truncation, and thus the update \eqref{equ: b_k_l_hat} preserves feasibility and unit norm.

\begin{lem}\label{lem:PSD-tildeBk}
If $B_j^{(l)} \succeq 0$ for all $j\neq k$, then $\tilde B_k$ defined by \eqref{equ: b_k_l_tilde} is positive semi-definite. Consequently, $\mathrm{Eig}_{m_k}(\tilde B_k)\succeq 0$ with $\mathrm{rank}\!\bigl(\mathrm{Eig}_{m_k}(\tilde B_k)\bigr)\le m_k$, and the update \eqref{equ: b_k_l_hat} satisfies
\[
B_k^{(l+1)} \succeq 0,\qquad \mathrm{rank}\bigl(B_k^{(l+1)}\bigr)\le m_k,\qquad \|b_k^{(l+1)}\|=\|B_k^{(l+1)}\|_{\s{f}}=1.
\]
\end{lem}

\begin{proof}
For any $u\in\R^{n_k}$, by definition
\[
u^\top \tilde B_k\, u
\;=\; \langle \tilde B_k,\, uu^\top\rangle
\;=\; \langle \mat(\sS_N), B^{(l)}_r \otK \ldots \otK  B^{(l)}_{k+1} \otK uu^\top \otK B^{(l)}_{k-1}\otK \ldots  \otK B^{(l)}_1 \rangle.
\]
Since $\mat(\sS_N)\succeq 0$ and each $B_j^{(l)}\succeq 0$ for $j\neq k$, the last inner product can be rewritten as a quadratic form after the Cholesky decomposition of the $B_j^{(l)}$'s and is therefore nonnegative. Hence $u^\top \tilde B_k\, u\ge 0$ for all $u$, that is, $\tilde B_k\succeq 0$. A PSD matrix is symmetric, so $\tilde B_k$ is symmetric as well. Keeping the top $m_k$ eigenpairs of a PSD matrix yields another PSD matrix with rank at most $m_k$, proving the claim for $\mathrm{Eig}_{m_k}(\tilde B_k)$. Finally, the normalization in \eqref{equ: b_k_l_hat} enforces $\|B_k^{(l+1)}\|_{\s{f}}=1$, and since $\|\vec(X)\|=\|X\|_{\s{f}}$, we also have $\|b_k^{(l+1)}\|=1$.
\end{proof}

\noindent\textbf{Estimates. }
The above two steps will be applied recursively. The algorithm terminates after $L$ iterations. Denote $\widehat{b}_k=b_k^{(L)}$ for $k=1,\ldots,r$. We estimate the signal strength 
\begin{equation}\label{eq:omega}
    \omega\;=\;\prod_{k=1}^r \|B_k\|_{\s{f}}
\end{equation} 
via
\begin{equation}\label{equ:omega_hat}
\widehat{\omega} = \left(\widehat{b}_1^\top, \ldots, \widehat{b}_r^\top \right) \cdot \pair(\sS_N).
\end{equation}
Moreover, we estimate the $A_k$'s via
\begin{align}
\label{equ:A_hat}
\widehat{B}_k:= \widehat{\omega}^{\frac1r}\vec^{-1}(\widehat{b}_k), \quad \widehat{A}_k =  \widehat U_k\widehat{\Lambda}_k^{1/2},
\end{align}
where $\widehat{B}_k=\widehat U_k\widehat\Lambda_k\widehat U_k^\top$ is the spectral decomposition of $\widehat{B}_k$. 
Finally, we apply the residual learning to estimate $\sigma^2$. Specifically, $\widehat{\sigma^2}$ is computed as 
\begin{align}
\label{equ: sigma_2_hat}
\widehat{\sigma^2} \;:=\; \frac1n \tr(\sS_N-\sF_{\widehat{\bs B}})\;=\;\frac1n\tr(\sS_N)-\frac1n\prod_{k=1}^r \tr(\widehat{B}_k).
\end{align}

\begin{algorithm}[!htp]
	\SetKwInOut{Input}{Input}
	\SetKwInOut{Output}{Output}
	\caption{Rank-1 power iteration} \label{algorithm1}
	\Input{Data $X^{(1)}, \ldots, X^{(N)} \in \V$, initializers $\widehat{b}_k^{(0)}$ such that $\|b_k^{(0)}\| = 1$, and $\vec^{-1}(b^{(0)}_k)$ is positive semi-definite and has rank at most $m_k$, for $k \in [r]$, number of iteration $L$\\}
	\Output{Estimators $\widehat{A}_1, \ldots, \widehat{A}_r$, $\widehat{\sigma^2}$\\}
	Compute the sample Kronecker covariance tensor as $\pair(\sS_N)=\frac{1}{N} \sum_{i = 1}^N X_i \otK X_i$.  \\
    \For {$l =1, \ldots, L$}{
   Power iteration step: Compute $\tilde{b}_k$ according to (\ref{equ: b_k_l_tilde}), for all $k \in [r]$. \\
   PCA step: Update $b_k^{(l)}$ according to (\ref{equ: b_k_l_hat}), for all $k \in [r]$.\\
    }
   Compute $\widehat{\omega}$ and $\widehat{A}_k$ according to \eqref{equ:omega_hat} and \eqref{equ:A_hat}. \\
   Compute $\widehat{\sigma^2}$ according to (\ref{equ: sigma_2_hat}). 
\end{algorithm}


The proposed rank-1 power iteration algorithm is summarized in Algorithm~\ref{algorithm1}.

\paragraph{Practical cost.}
Forming the sample Kronecker–covariance tensor
\(\sS_N=\tfrac1N\sum_{i=1}^N X^{(i)}\otimes X^{(i)}\in\cL(\V,\V)\)
and then carrying out the power-iteration step
\eqref{equ: b_k_l_tilde} can, in principle, be expensive:
explicitly materialising \(\sS_N\) needs
\(O(Nn^2)\) arithmetic operations and \(O(n^2)\) memory
(\(n=\prod_{k}n_k\)).
A direct evaluation of
\begin{equation}\label{eq:auxefficient}    \left(\vec(H_1H_1^\top)^\top,\dots,\vec(H_rH_r^\top)^\top\right)
   \cdot\pair(\sS_N),
   \qquad
   H_k\in\R^{n_k\times\ell_k},
\end{equation}
would add a further \(O(rn^2+\sum_{k}n_k^2\ell_k)\) flops, for any \(\bs H=(H_1,\dots,H_r)\). However, a simple identity yields a much cheaper alternative:
\[
\left(\vec(H_1H_1^\top)^\top,\dots,\vec(H_rH_r^\top)^\top\right)
\cdot\pair(\sS_N)
=\frac1N\sum_{i=1}^N
      \|\sF_{\bs H}^{*}(X^{(i)})\|_{\s{f}}^{2},
\]
where
\(\sF_{\bs H}^{*}\) is the adjoint Tucker map
(cf.\ Section~\ref{sec:slm}).  Computing each
\(\sF_{\bs H}^{*}(X^{(i)})\) costs
\(O(\prod_{k} n_k\ell_k)\) flops, so the total workload is
\(O\!\bigl(N\,n\,\ell\bigr)\) with
\(\ell=\prod_{k}\ell_k\).
In the special case used in \eqref{eq:auxefficient} we take
\(\ell_k=m_k\), giving \(O(Nnm)\) flops but \emph{without}
ever storing the full \(n^2\)-dimensional tensor.
Thus the algebraic reformulation avoids the memory bottleneck and,
for tall–thin \(H_k\) (small \(\ell_k\)), reduces the runtime by up to
a factor \(n/\ell\).

\subsection{Concentration}\label{sec:theory}

\color{black}

In this section, we provide statistical guarantee on the concentration behavior of $\sS_N -\Sigma$. Denote $n=\prod_{i=1}^r n_i$, $m=\prod_{i=1}^r m_i$ to be the dimensions of $\V$ and $\U$ respectively. Also, we denote $n_{\max} = \max_{k = 1, \ldots, r} n_k$ and $m_{\max} = \max_{k = 1, \ldots, r} m_k$ throughout. Recall the definition of the spectral norm $\|\cdot\|_{\rm ml}$ in \eqref{eq:spnorm}. The following theorm provides a sharp concentraion inequality for the sample covariance operator.

\begin{thm}
\label{theorem:concentration_spectral}
Under the TPCA model,  for any $\delta_1,\delta_2,\delta_3>0$, we have
\begin{align*}
\|\sS_N - \Sigma\|_{\rm ml} \le\; t_1 + t_2 +t_3 \qquad\mbox{with probability at least }1- \delta_1-\delta_2-\delta_3,
\end{align*}
where
\begin{align*}
& t_1 := 2\sigma^2\max\{\alpha_1,\sqrt{\alpha_1}\} \quad\quad\quad\;\text{ with } \alpha_1: = \frac{1}{0.145 N}\left(  2\log(10r) \sum_{k=1}^r n_k+\log\tfrac{1}{\delta_1}\right),\\
&t_2 := 2\|\bs A\|^2\max\{\alpha_2,\sqrt{\alpha_2}\} \quad\quad\text{ with } \alpha_2: = \frac{1}{0.145  N}\left(2\log(10r)  \sum_{k=1}^r m_k+\log\tfrac{1}{\delta_2}\right), \\
&t_3 := 4\sigma\|\bs A\|\max\{\alpha_3,\sqrt{\alpha_3}\} \quad\quad\;\text{ with } \alpha_3: = \frac{3}{N}\left( \log(10r)\sum_{k=1}^r (n_k + m_k )+\log\tfrac{1}{\delta_3}\right).
\end{align*}
\end{thm}
Here \(\|\bs A\|:=\prod_{k=1}^{r}\|A_k\|\) is the product of the spectral norms of the factor matrices. All the proofs in this section are moved to Appendix~\ref{app:proofsalg}.
\smallskip

The following corollary shows a version of high-dimensional consistency.
\begin{cor}\label{cor:HDcons}
Taking in Theorem~\ref{theorem:concentration_spectral} $$\delta_1=N^{-2 \log(10r)\sum_{k = 1}^r n_k},\quad \delta_2=N^{-2 \log(10r)\sum_{k = 1}^r m_k},\text{ and } \delta_3=N^{- \log(10r)\sum_{k = 1}^r (n_k +m_k)},$$
we conclude that 
$$\frac{\|\sS_N - \Sigma\|_{\rm ml}}{\|\bm{A}\|^2}\to 0, 
$$
as long as $\max\left\{\tfrac{n_{\max}}{\|\bs A\|^2}, m_{\max}, \tfrac{n_{\max}}{\|\bs A\|}\right\}\tfrac{\log(N)}{N}\to 0$ (here $r$ and $\sigma^2$ are treated as fixed). In particular, the condition can be reduced to $\max\left\{ m_{\max}, \tfrac{n_{\max}}{\|\bs A\|}\right\}\tfrac{\log(N)}{N}\to 0$ when $\|\bs A\|>1$, and to $\tfrac{m_{\max}\log(N)}{N}\to 0$ when $\|\bs A\|>\tfrac{n_{\max}}{m_{\max}}$. 
\end{cor}

\medskip
\noindent\textit{Sketch of the proof and scaling remarks.}
The bound in Theorem~\ref{theorem:concentration_spectral}
is obtained by isolating the three components
\(
N^{-1}\!\sum_{i=1}^{N}\eps_i\!\otimes\!\eps_i-\sigma^{2}\I_{\s{v}},
\;
N^{-1}\!\sum_{i=1}^{N}Y_i\!\otimes\!Y_i-\sF_{\bs A}\sF_{\bs A}^{*},
\;
2N^{-1}\!\sum_{i=1}^{N}Y_i\!\otimes\!\eps_i,
\)
where \(Y_i=\bs A\!\cdot\! Z_i\).
Each term is handled by a covering argument combined with a Hanson–Wright
inequality for Gaussian chaos (Lemma~\ref{lem: hanson-wright}):
the isotropic noise part yields \(t_1\),
the structured signal part yields \(t_2\),
and the cross term yields \(t_3\).
Importantly, treating them separately allows a sharper result than a direct bound on
\(\|\sS_N-\Sigma\|_{\rm ml}\). This is because we exploit the TPCA whitening technique for the dominant signal term and the cross term, which effectively reduces the covering number from
\(\exp(\sum_{k}n_k)\) to \(\exp(\sum_{k}m_k)\).

\smallskip
To see which term dominates, fix \(r\) and \(\sigma^{2}\) and choose
\(\delta_1,\delta_2,\delta_3\) as in Corollary~\ref{cor:HDcons}.
Then
\(\alpha_1\simeq n_{\max}/N\),
\(\alpha_2\simeq m_{\max}/N\),
\(\alpha_3\simeq n_{\max}/N\), where $\cdot \simeq \cdot$ denotes the two sequences are asymptotically of the same order, up to logarithm factors.
If furthermore $\|A_k\|\simeq \sqrt{n_k}$ (this holds with high
probability when the entries are i.i.d.\ \(N(0,1)\), see
Theorem 6.1 of~\cite{Wainwright_2019}),
we obtain \(t_1\ll t_3\).
When \(N\ge n_{\max}\),  we have
\(t_3\ll t_2\), so $t_2$ dominates. Similarly, when $N< n_{\max}$, we also have $t_3 \ll t_2$ unless the mode dimensions are \emph{extremely}
unbalanced (precisely, unless
\(n_{\max}\gg n_{-\max} m_{\max}\max\{m_{\max},N\}\), where $n_{-\max} = n/n_{\max}$). 
Hence the spectral error is driven by the structured part of the model.

\smallskip
Finally, we remark that if \(\sigma^{2}\) is small relative to
\(\prod_{k}\|A_k\|^{2}\)
and the preceding balance condition holds,
then Theorem~\ref{theorem:concentration_spectral} gives
\[
\|\sS_N-\Sigma\|_{\rm ml}
      =\widetilde{O}\!\left(
           \|\Sigma\|\,\sqrt{\tfrac{\sum_{k}m_k}{N}}
        \right),
\]
where the \(\widetilde{O}\) hides poly-logarithmic factors.
Treating each \(X_i\) as a length-\(n\) vector and applying a standard
Wishart bound (Example 6.4 in~\cite{Wainwright_2019}) would instead yield
\(\widetilde{O}(\|\Sigma\|\sqrt{n/N})\).
Thus respecting the TPCA structure and using the multilinear spectral
norm replaces the \(\sqrt{\prod n_k}\) factor by \(\sqrt{\sum_k m_k}\),
removing any dependence on the (possibly huge) mode sizes.
This improvement is substantial and highlights a key advantage of the
tensor viewpoint; we flag the point already in the Introduction.

\subsection{One-step consistency}

As before, let \(b_k := \vec(B_k)/\|B_k\|_{\s{f}}\) for \(k=1,\ldots,r\). Suppose the initializer \(B_k^{(0)}\) is positive semidefinite with \(\operatorname{rank}(B_k^{(0)})\le m_k\) and \(\|B_k^{(0)}\|_{\s{f}}=1\). Then, at each iteration \(l\), the update \(b_k^{(l)}\) is a unit vector and \(B_k^{(l)}:=\vec^{-1}(b_k^{(l)})\) is positive semidefinite with \(\operatorname{rank}(B_k^{(l)})\le m_k\).

Let \(\theta_k^{(l)}\) denote the angle between \(b_k\) and \(b_k^{(l)}\). Using \(\|B_k^{(0)}\|_{\s{f}}=1\),
\[
\cos \theta_k^{(0)}
= \langle b_k, b_k^{(0)} \rangle
= \frac{1}{\|B_k\|_{\s{f}}}\,\tr\!\big(B_k B_k^{(0)}\big).
\]
Since \(B_k\) and \(B_k^{(0)}\) are positive semidefinite, \(\tr(B_k B_k^{(0)})\ge 0\), hence \(0\le \cos \theta_k^{(0)}\le 1\) and \(0\le \sin \theta_k^{(0)}\le 1\). Moreover,
\[
\|b_k-b_k^{(l)}\|^2 \;=\; 2\big(1-\cos \theta_k^{(l)}\big).
\]
Thus, to control the one-step estimation error \(\|b_k-b_k^{(1)}\|\), it suffices to bound \(\sin \theta_k^{(1)}\) (equivalently, \(\cos \theta_k^{(1)}\)).

Our probabilistic bounds are derived from deterministic arguments on a high-probability event. For \(\alpha>0\), define
\[
\bs E(\alpha)\;:=\;\big\{ \|\Phi\|_{\rm ml}\le \alpha \big\}.
\]
Given \(\delta_1,\delta_2,\delta_3\in(0,1)\), let \(t_1,t_2,t_3\) be as in Theorem~\ref{theorem:concentration_spectral}, and set
\begin{equation}\label{eq:psi}
\psi\;:=\;\sigma^2+t_1+t_2+t_3.
\end{equation}
Since
\[
\|\Phi\|_{\rm ml}
= \big\|\sigma^2 \I_{\s{v}} + \sS_N - \Sigma\big\|_{\rm ml}
\le \sigma^2 + \big\|\sS_N - \Sigma\big\|_{\rm ml},
\]
the event \(\bs E(\psi)\) is implied by \(\{\|\sS_N-\Sigma\|_{\rm ml}\le t_1+t_2+t_3\}\), which holds with probability at least \(1-\delta_1-\delta_2-\delta_3\) by Theorem~\ref{theorem:concentration_spectral}. Throughout this subsection we work assuming \(\bs E(\psi)\).

Our analysis depends on two quantities: the signal strength \(\omega\) defined in~\eqref{eq:omega}, and the noise level \(\psi\) in~\eqref{eq:psi}; typically \(\omega \gg \psi\). Recall that \(m = m_1\cdots m_r\) and consider the function

\newcommand{\Cnoise}{32 m \big(\tfrac{\psi}{\omega}\big)^2}
\begin{equation}\label{eq:gfun}
g(x)\;=\;x^2(1-x^2)^{r-1}-\Cnoise,\qquad 0<x<1,
\end{equation}
which is negative at both endpoints. The function $x\mapsto x^2(1-x^2)^{r-1}$ attains its maximum over $(0,1)$ at $x=\sqrt{1/r}$ with value $\tfrac1r(1-\tfrac1r)^{r-1}$. Hence $g(x)\le 0$ for all $0<x<1$ unless
\begin{equation}\label{eq:assnoneptyint}
\big(\tfrac{\psi}{\omega}\big)^2
\;<\;\frac{1}{32m}\,\frac1r\Big(1-\frac1r\Big)^{r-1}.
\end{equation}
Moreover, if \eqref{eq:assnoneptyint} holds, then $g(x)>0$ precisely on an interval $(\tau_0,\tau_1)$, where the roots $\tau_0 = \tau_0(\tfrac{\psi}{\omega}),\tau_1 = \tau_1(\tfrac{\psi}{\omega})$ of $g$ depend on $\psi/\omega$ and satisfy, as $\psi/\omega\to 0$,
\begin{enumerate}
\item $\displaystyle \tau_0=4\sqrt{2m}\,\tfrac{\psi}{\omega}
\;+\;O\!\big((\tfrac{\psi}{\omega})^3\big),$
\item $\displaystyle 1-\tau_1
=\tfrac12\,(32m)^{1/(r-1)}\big(\tfrac{\psi}{\omega}\big)^{2/(r-1)}
\;+\;O\!\big((\tfrac{\psi}{\omega})^{4/(r-1)}\big).$
\end{enumerate}
The above equalities can be obtained by expressed $x^2(1-x^2)^{r-1}$ as $x^2 - (r-1)x^4 + \ldots + (-1)^{r-1}x^{2r}$ and $(t-1)^2 t^{r-1}(t+2)^{r-1}$ with $t = 1 -x$, respectively.  
In particular, for small $\psi/\omega$ we have $\tau_0\downarrow 0$ and $\tau_1\uparrow 1$ (although the latter convergence is much slower for large values of $r$). The point $\tau_1$ is crucially used in our next assumption.

\begin{assp}[Initialization]\label{assp: tau}
Assume \eqref{eq:assnoneptyint} and suppose the initial matrices $B_k^{(0)}$ satisfy
$\sin \theta_k^{(0)} \le \tau$ for all $k=1,\ldots,r$, where $\tau>0$ is such that $\tau < \tau_1\!\big(\tfrac{\psi}{\omega}\big)$.
\end{assp}

Assumption~\ref{assp: tau} is mild: $\omega/\psi$ is the signal-to-noise ratio. When this ratio is large (e.g., $\psi/\omega\to 0$ along an asymptotic sequence), the asymptotic behaviour of $\tau_1=\tau_1(\tfrac\psi\omega)$ shows that we can choose $\tau$ arbitrarily close to $1$.

\begin{thm}\label{theorem:rank-1poweriteration}
Suppose Assumption~\ref{assp: tau} and the event $\bs E(\psi)$ hold. After the first iteration of Algorithm~\ref{algorithm1},
\begin{equation}\label{equ: sin_theta_k_1}
\sin\theta^{(1)}_k
\;\le\;
\frac{\psi}{\omega}\,
\frac{4\sqrt{2m}}{(1-\tau^2)^{\frac{r-1}{2}}}
\qquad\text{for all }k=1,\ldots,r,
\end{equation}
and
\[
\frac{|\omega-\widehat{\omega}^{(1)}|}{\omega}
\;\le\;
\Big(\tfrac{\psi}{\omega}\Big)^{\!2}
\frac{16mr}{(1-\tau^2)^{r-1}}
\;+\;
\frac{\psi}{\omega}\sqrt{m}\,.
\]
\end{thm}

Theorem~\ref{theorem:rank-1poweriteration} gives a \emph{deterministic} one-step error bound that depends only on two scalars computed from the fixed signal $\{A_k\}$ and the sample: the \emph{signal size} $\omega$ in \eqref{eq:omega}, and the \emph{effective noise level} \(\psi\) defined in~\eqref{eq:psi} via the event \(\bs E(\psi)=\{\|\Phi\|_{\rm ml}\le\psi\}\) with \(\Phi=\sigma^2\I_{\s v}+\sS_N-\Sigma\) (see~\eqref{eq:SNandPhi} and~\eqref{eq:spnorm}). In words: the first power iteration is accurate whenever the ratio \(\psi/\omega\) is small and the initialization angle \(\tau\) satisfies Assumption~\ref{assp: tau}.

\begin{ex}[Gaussian-design]\label{ex:gaussiandesign}
The bounds in Theorem~\ref{theorem:rank-1poweriteration} are deterministic given $\{A_k\}_{k = 1}^r$, but it is useful to understand their typical sizes when $A_k$'s are random.
Assume $A_k\in\R^{n_k\times m_k}$ have i.i.d.\ $N(0,1)$ entries (independent across $k$) and the aspect ratios $m_k/n_k$ stay bounded away from $1$.
Then, with high probability as $\min_k n_k\to\infty$, the singular values of $A_k$ satisfy
\begin{equation}\label{eq:gaussS4Frob}
c \sqrt{n_k} \le \sigma_1(A_k), \ldots, \sigma_{m_k}(A_k) \le C\sqrt{n_k}, 
\end{equation}
for universal constants $c,C>0$. 
Consequently, by Lemma~\ref{lem:normsequal},
\begin{equation}\label{eq:gaussOmega}
\|\bs A \| \;=\;\Theta_{\P}\!\Big(\sqrt{n}\Big) \text{ and } \omega \;=\;\prod_{k=1}^r \|B_k\|_{\s{f}}
\;=\;\Theta_{\P}\!\Big(n\sqrt{m}\Big).
\end{equation}
We now look at the behaviour of $\tfrac{\psi}{\omega}$. By definition, 
$$
\psi = \sigma^2 + 2\sigma^2\max\{\alpha_1,\sqrt{\alpha_1}\} + 2\|\bs A\|^2\max\{\alpha_2,\sqrt{\alpha_2}\} + 4\sigma\|\bs A\|\max\{\alpha_3,\sqrt{\alpha_3}\}. 
$$ Taking $\delta_1, \delta_2, \delta_3$ appropriately as in Corollary \ref{cor:HDcons}, we have
\[
\alpha_1 = \Theta\left(\frac{n_{\max} \log N}{N}\right), \alpha_2 = \Theta\left(\frac{m_{\max} \log N}{N}\right), \text{ and } \alpha_3 = \Theta\left(\frac{n_{\max} \log N}{N}\right).
\]
When $n_{\max} \lesssim \frac{N}{\log N}$ and fixed $\sigma^2$,  with high probability, we have,
\begin{align*}
\frac{\psi}{\omega} & \;\lesssim\; \frac{\sigma^2}{n\sqrt{m}} \;+\; \frac{\sigma^2}{n}\sqrt{\frac{n_{\max} \log N}{mN}}
\;+\; \,\sqrt{\frac{m_{\max}\log N}{mN}}
\;+\; \sigma\,\sqrt{\frac{n_{\max} \log N}{nmN}}\\
& \lesssim \frac{1}{\sqrt{m}} \max\left\{\frac{1}{n}, \sqrt{\frac{m_{\max} \log N}{N}}\right\}\xrightarrow\;0. 
\end{align*}
When $\tfrac{N}{\log N} \lesssim n_{\max}$ and fixed $\sigma^2$,  with high probability, we have,
\begin{align*}
\frac{\psi}{\omega} & \;\lesssim\; \frac{\sigma^2}{n\sqrt{m}} \;+\; \frac{\sigma^2n_{\max} \log N}{n\sqrt{m}N}
\;+\; \, \max \left\{\sqrt{\frac{m_{\max}\log N}{mN}}, \frac{m_{\max} \log N}{ \sqrt{m}N }\right\}
\;+\; \sigma\,\frac{n_{\max} \log N}{\sqrt{nm}N}\\
& \lesssim \frac{1}{\sqrt{m}} \max\left\{\frac{1}{n}, \sqrt{\frac{m_{\max}\log N}{N}}, \frac{m_{\max} \log N}{N}, \frac{n_{\max} \log N}{ \sqrt{n}N}\right\}, 
\end{align*}
which also converges to $0$ when $n_k$'s and $N$ diverge if we further assume the mode sizes and their latent dimensions are balanced, such as, $\tfrac{n}{n_{max}} \ge n_{\max}$ and $\tfrac{m}{m_{\max}} \ge m_{\max}$.

Hence Theorem~\ref{theorem:rank-1poweriteration} gives, after a single iteration and a feasible initialization (Assumption~\ref{assp: tau}),
\[
\sin\theta_k^{(1)} \;=\; o_{\P}(1)
\quad\text{and}\quad
\frac{|\widehat{\omega}^{(1)}-\omega|}{\omega} \;=\; o_{\P}(1),
\]
with the explicit rates shown in \eqref{equ: sin_theta_k_1}. Finally, we remark that the first term $\sigma^2$ in $\Phi$, or $\tfrac{\sigma^2}{n\sqrt{m}}$ in $\tfrac{\psi}{\omega}$ comes from the fact that we apply a rank-1 power iteration algorithm to estimate a tensor whose signal part has rank 2. We think this term could be potentially avoided after more careful algorithm design and then more subtle theoretical analysis, so that the estimator would be consistent when $N$ diverge and $n$ is given.  
\end{ex}

We now analyze the guarantees of the algorithm after subsequent runs. Define, for later use, functions giving the bounds in Theorem~\ref{theorem:rank-1poweriteration}
\begin{equation}\label{eq:f12}
f_1(\tau)
=\frac{\psi}{\omega}\frac{4\sqrt{2m}}{(1-\tau^2)^{\frac{r-1}{2}}}
\qquad\text{and}\qquad
f_2(\tau)
=\Big(\tfrac{\psi}{\omega}\Big)^{\!2}\frac{16mr}{(1-\tau^2)^{r-1}}+\frac{\psi}{\omega}\sqrt{m}\,.
\end{equation}
We now analyze when the one-step bound $f_1(\tau)$ is smaller than the initialization bound $\tau$. A direct corollary of Theorem~\ref{theorem:rank-1poweriteration} is that the bound on $\sin\theta^{(1)}_k$ improves over that on $\sin\theta^{(0)}_k$ if and only if
$f_1(\tau)\le \tau$, or, equivalently, $g(\tau)\ge 0$ with $g$ from~\eqref{eq:gfun}.
As argued earlier, under~\eqref{eq:assnoneptyint} the function $g$ is positive on an open interval $(\tau_0,\tau_1)$ with $\tau_0<\sqrt{1/r}<\tau_1$.

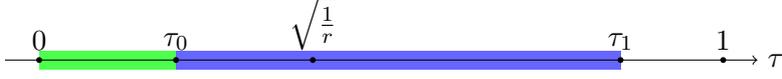
\begin{figure}[htp!]
    \centering
\begin{tikzpicture}[x=9cm,y=1cm]
  \def\xB{0.2} 
  \def\xD{0.4} 
  \def\xC{0.85} 
  \def\xE{1} 

  \fill[green!70] (0,-0.12) rectangle (\xB,0.12);   
  \fill[blue!60]  (\xB,-0.12) rectangle (\xC,0.12); 

  \draw[->] (-0.05,0) -- (1.05,0) node[right] {$\tau$};

  \fill (0,0)   circle (1.1pt);
  \fill (\xB,0) circle (1.1pt);
  \fill (\xD,0) circle (1.1pt);
  \fill (\xC,0) circle (1.1pt);
  \fill (\xE,0) circle (1.1pt);

  \node[above] at (0,0.0) {$0$};
  \node[above] at (\xB,0) {$\tau_0$};
  \node[above] at (\xD,0) {$\sqrt{\tfrac1r}$};
  \node[above] at (\xC,0) {$\tau_1$};
  \node[above] at (\xE,0) {$1$};
\end{tikzpicture}
\caption{Under \eqref{eq:assnoneptyint}, we have $f_1(\tau)\le \tau$ precisely for $\tau\in[\tau_0,\tau_1]$ (blue). Moreover, if $\tau$ lies in the green region $(0,\tau_0)$, then $f_1(\tau)$ also lies in $(0,\tau_0)$.}
\label{fig:placeholder}
\end{figure}

Under Assumption~\ref{assp: tau} we have $\tau<\tau_1$, which yields two regimes:

\smallskip
\noindent {(i) Contraction region:} If $\tau\in(\tau_0,\tau_1)$, then $f_1(\tau)<\tau$, so $\sin\theta^{(1)}_k<\sin\theta^{(0)}_k$ for all $k$.

\smallskip
\noindent {(ii) Noise floor region:} If $\tau<\tau_0$, then, since $f_1$ is strictly increasing on $(0,1)$,
\begin{equation}\label{eq:f1tau0}
f_1(\tau)\;<\;f_1(\tau_0)\;=\;\tau_0,    
\end{equation}
so the one-step bound lands in the interval $I_0=[0,\tau_0)$ and stays there thereafter. Thus, $\tau_0$ acts as a \emph{noise-determined floor} for this analysis: our proof technique, based on the inequality $f_1(\tau)\le\tau$ being equivalent to $g(\tau)\ge0$, cannot certify a bound below~$\tau_0$ without a sharper argument.\footnote{This phenomenon is standard for one-step spectral/power-type analyses: the certified error decreases until it hits a floor determined by the level of the deterministic control parameter (here encoded by $\psi/\omega$). The true error may well be smaller; the floor reflects the coarseness of the inequality, not a limitation of the algorithm per se.} For the reader’s convenience, the two regimes are depicted in Figure~\ref{fig:placeholder}.

\medskip
\noindent\textbf{Consequence for $\widehat\omega$.}
By Theorem~\ref{theorem:rank-1poweriteration}, if $\tau<\tau_0$, then
\[
\frac{|\omega - \widehat{\omega}^{(1)}|}{\omega}
\;\le\; f_2(\tau)\;=\;\frac{r}{2}f_1^2(\tau) + \frac{\psi}{\omega}\sqrt{m}.
\]
Thus, by \eqref{eq:f1tau0}, in the noise floor region we observe that
\[
\frac{|\omega - \widehat{\omega}^{(1)}|}{\omega}
\;\le\; \frac{r}{2}\,\tau_0^2 \;+\; \frac{\psi}{\omega}\sqrt{m}
\;=\; \frac{\psi}{\omega}\sqrt{m} \;+\; O\!\left((\tfrac{\psi}{\omega})^2\right).
\]

\begin{rem}[No incoherence needed]
Many tensor power iteration analyses assume \emph{incoherence} of factors (e.g., for a matrix with orthonormal columns $U\in\R^{n\times r}$, the condition $\max_i n\,\|e_i^\top U\|_2^2 \le \mu r$), which prevents ``spiky'' components and controls cross-term interference; see, e.g., \cite{anandkumar2014guaranteed}. In our \emph{rank-1} extraction, we work with the PSD matrices $B_k=A_kA_k^\top$ and their perturbed versions. The analysis relies only on spectral-norm control of these matrices (captured by the parameter $\psi$) and does not require any incoherence-type control on the singular vectors of $A_k$. In short: the rank-1 structure and PSD setting remove the usual need for incoherence.
\end{rem}

\begin{cor}\label{cor: cor_A_consistency}
Suppose Assumption~\ref{assp:scale}, Assumption~\ref{assp: tau} hold, and that $A_1,\ldots,A_r$ have full column rank. Let $\widehat A_k$ be the output after $L=1$ step. Then for each $k=1,\ldots,r$ there exists an orthogonal matrix $O_k\in\R^{m_k\times m_k}$ such that
\[
\frac{1}{\sqrt{n_k m_k}}\,
\|\widehat A_k - A_k O_k\|_{\s{f}}
\;\le\;
\frac{\big[f_2(\tau)+\sqrt{2}\,f_1(\tau)\big]\;\omega^{1/r}}
{2(\sqrt{2}-1)\,\sqrt{n_k m_k}\,\lambda_{m_k}(B_k)}\,,
\]
where $\lambda_{m_k}(B_k)$ denotes the smallest nonzero eigenvalue of $B_k$ and $f_1,f_2$ are given in~\eqref{eq:f12}.
\end{cor}

\begin{ex}
   We continue the Gaussian design Example~\ref{ex:gaussiandesign}. Under the random-design model, w.h.p.\ we have
\[
\omega^{1/r}=\Theta\!\big(n_k\sqrt{m_k}\big)
\quad\text{and}\quad
\lambda_{m_k}(B_k)=\Theta(n_k), \text{ for } k \in [r], 
\]
so the prefactor in Corollary~\ref{cor: cor_A_consistency} simplifies to $\Theta(1/\sqrt{n_k})$. Hence
\[
\frac{1}{\sqrt{n_k m_k}}\,
\|\widehat A_k^{(1)}-A_k O_k\|_{\s{f}}
\;=\;
\Theta\!\Big(\frac{f_2(\tau)+\sqrt{2}\,f_1(\tau)}{\sqrt{n_k}}\Big)=\Theta\!\Big(\frac{\psi}{\omega\sqrt{n_k}}\Big),
\qquad\text{w.h.p.}
\]
In particular, if $\psi/\omega\to 0$ in probability under this Gaussian design, then $f_1(\tau),f_2(\tau)\to 0$ and the one-step estimator is consistent, with rate governed by the vanishing of $\psi/\omega$ and the factor $1/\sqrt{n_k}$.
\end{ex}

\subsection{Consistency of $\widehat{\sigma^2}$}

Recall the plug-in estimator
\[
   \widehat{\sigma^2}
   \;=\;\frac1n\operatorname{tr}\bigl(\sS_N-\sF_{\widehat{\bs B}}\bigr)
   \;=\;\sigma^2+\frac1n\tr J_1+\frac1n\tr J_2,
\]
where
\[
   J_1:=\sF_{\bs B}-\sF_{\widehat{\bs B}},
   \qquad
   J_2:=\sS_N-\Sigma .
\]
Hence
\(
   |\widehat{\sigma^2}-\sigma^2|
   \le \tfrac1n|\tr J_1|+\tfrac1n|\tr J_2|
\)
and in our analysis we focus on controlling the two trace terms.

\begin{thm}
\label{theorem: sigma_2_consistency} 
Assume Assumptions \ref{assp:scale} and \ref{assp: tau}, and work on
the high-probability event $\bs E(\psi)$ introduced in
\eqref{eq:psi}.  After one iteration of Algorithm~\ref{algorithm1}, with probability at least $ 1- \frac{1}{N^2 n^2}$,
\[
   |\widehat{\sigma^2}-\sigma^2|
   \;\le\;
   \frac{\sqrt m\,\omega}{n}\left[\left(1+2f_1(\tau)+\sqrt2\,f_2(\tau)\right)^r-1\right]
   + t_4 ,
\]
where $f_1,f_2$ are given in \eqref{eq:f12} and
\[
   t_4
   \;=\;
   \frac{m\|\Sigma\|+(n-m)\sigma^2/\|\Sigma\|}{n}
   \,
   \max\{\alpha_4, \sqrt{\alpha_4}\},
   \qquad
   \alpha_4=
   \frac{2\|\Sigma\|\log(Nn)}{0.145N\bigl(m\|\Sigma\|+(n-m)\sigma^2/\|\Sigma\|\bigr)}.
\]

\end{thm}

The first term bounds $\tfrac1n|\tr J_1|$; it depends on
$\omega=\prod_k\|B_k\|_{\s{f}}$ and on the
post-iteration angles through $f_1,f_2$.
The second term $t_4$ controls $\tfrac1n|\tr J_2|$ via a Hanson–Wright
inequality; notice that $\alpha_4\asymp (Nm)^{-1}\log(Nn)$, so
$t_4=\tilde O\bigl(\sqrt{m/N}\bigr)$ as long as $Nm\to\infty$.

Suppose $r$ and $\sigma^2$ are fixed, $N\gtrsim m_{\max}$, and
$\|A_k\|_{\mathrm{op}}\asymp\sqrt{n_k}$ (e.g.\ Gaussian loading
matrices).  Consider the first case in Example \ref{ex:gaussiandesign} when $N/\log N \gg n_{\max}$, we have
$f_1(\tau)\asymp f_2(\tau)\asymp\sqrt{m_{\max}\log N/(Nm)}$ if $n$ is large enough, and 
\[
   \frac{\sqrt m\,\omega}{n}\left[
   \bigl(1+2f_1+\sqrt2f_2\bigr)^r-1\right]
   =\tilde O\!\left(\sqrt{\frac{m_{\max}m}{N}}\right), 
\]
which leads to 
\(
   |\widehat{\sigma^2}-\sigma^2|
   =\tilde O\!\bigl(\sqrt{m_{\max}m/N}\bigr),
\)
so $\widehat{\sigma^2}$ is consistent whenever
$N/\log N\gg \max\{m m_{\max}, n_{\max}\}$. The other case in Example \ref{ex:gaussiandesign} can be argue similarly with more careful assumptions on the balance of the tensor modes.

Finally, if the core factors are column-orthogonal ($A_k^\top A_k=I_{m_k}$ for all $k$), then $B_k=A_kA_k^\top$ is an orthogonal projector and therefore
\[
\|\sF_{\bs B}\|=1,
\qquad
\Sigma=\sigma^2 I+\sF_{\bs B},
\qquad
\|\Sigma\|=\sigma^2+1\asymp 1.
\]
Substituting this into the definition of $\alpha_4$ and $t_4$ in Theorem~\ref{theorem: sigma_2_consistency} yields
$$
\alpha_4
=\frac{2\|\Sigma\|\log(Nn)}{0.145N\bigl(m\|\Sigma\|+(n-m)\sigma^2/\|\Sigma\|\bigr)}
=\tilde O\!\Big(\frac{1}{Nn}\Big),
$$
and
$$
t_4
=\frac{m\|\Sigma\|+(n-m)\sigma^2/\|\Sigma\|}{n}\cdot
\max\{\alpha_4,\sqrt{\alpha_4}\}
=\tilde O\!\Big(\frac{1}{\sqrt{Nn}}\Big).
$$
Thus the sampling fluctuation term $\tfrac1n|\tr J_2|$ attains the standard parametric rate. The estimation term $\tfrac1n|\tr J_1|$ scales as
\[
\frac{\sqrt m\,\omega}{n}\Bigl[\bigl(1+2f_1(\tau)+\sqrt{2}\,f_2(\tau)\bigr)^r-1\Bigr],
\]
which under the Gaussian-loading calibration in this subsection is
$\tilde O\!\bigl(\sqrt{m_{\max}m/N}\bigr)$. Hence with orthogonal factors, the \emph{sampling} contribution is already parametric, and the accuracy of $\widehat{\sigma^2}$ after one iteration is driven by the post-iteration alignment (via $f_1,f_2$). If the factors were known (so $J_1\equiv 0$), the plug-in estimator would be parametric; with estimated factors, stronger initialization (smaller $\tau$) or additional iterations reduces $\tfrac1n|\tr J_1|$ toward the same $(Nn)^{-1/2}$ scale.

\section{Numerical examples}\label{sec:experiments}
We study the numerical performance of the proposed probabilistic tensor PCA estimators given by the EM algorithm and the rank-$1$ power-iteration method. Throughout the section we evaluate factor estimation quality with the same Procrustes-aligned, size-normalized error.

\paragraph{Error metric (used in all cases).}
For each mode $k\in[r]$ we report
\[
\mathrm{err}_k
\;=\;
\min_{O_k^\top O_k=I_{m_k}}
\frac{\bigl\lVert \widehat A_k O_k - A_k \bigr\rVert_{\mathrm F}}
     {\sqrt{n_k m_k}},
\qquad
\mathrm{Err}\;=\;\frac{1}{r}\sum_{k=1}^r \mathrm{err}_k.
\]
Here $\|\cdot\|_{\mathrm F}/\sqrt{n_k m_k}$ is the entrywise root-mean-square error; the Procrustes step removes the intrinsic right-rotation ambiguity, so $\mathrm{err}_k$ measures the per-entry discrepancy \emph{modulo rotation}. We standardize the truth to $\|A_k\|_{\mathrm F}=\sqrt{n_k m_k}$ in all simulations, making errors comparable across $(n_k,m_k)$ and consistent with our other experiments.

\medskip
\noindent\textbf{Case 1 (EM algorithm).}
We set $r=3$ with equal mode sizes $n_1=n_2=n_3=n\in\{5,15\}$. For each $n$ we consider $m_1=m_2=m_3=m\in\{\tfrac{2}{5}n,\tfrac{3}{5}n\}$ and noise $\sigma^2\in\{0.1,1,10\}$; the mean tensor is fixed at $\nu=0$.

\emph{Data generation.}
For each replication, draw $A_k^{\mathrm{raw}}\in\mathbb R^{n\times m}$ with i.i.d.\ $N(0,1)$ entries, write $A_k^{\mathrm{raw}}=U_kD_kV_k^\top$, and set $A_k=U_kD_k$ (right singulars fixed to identity). Rescale each factor so that $\|A_k\|_{\mathrm F}=\sqrt{nm}$.

\emph{Sampling.}
Generate $N=15$ observations $X^{(i)}=\mathbf A\!\cdot Z^{(i)}+\varepsilon^{(i)}$ with
$Z^{(i)}\sim\mathcal N(0,I_{\boldsymbol m})$ and
$\varepsilon^{(i)}\sim\mathcal N(0,\sigma^2 I_{\boldsymbol n})$.

\emph{Estimation.}
Run the fast EM routine (Section~\ref{sec:EM}) with an HOSVD warm start, tolerance $10^{-3}$, and at most $100$ iterations, initialized at the true $\sigma^2$. For each scenario we report $\mathrm{Err}$ averaged over $10$ replications and its $95\%$ CI in Table \ref{tab:EM-sim}. Standard sanity checks have been applied; e.g. that the algorithm increases the log-likelihood with each step.

\begin{table}[ht]
\centering
\begin{tabular}{ccccc}
\hline
$n$ & $m$ & $\sigma^2$ & mean error & 95\% CI \\
\hline
 5 & 2 & 0.1 & 0.366 & [0.348, 0.384] \\
 5 & 2 & 1.0 & 0.218 & [0.199, 0.237] \\
 5 & 2 & 10  & 0.158 & [0.137, 0.179] \\
\hline
 5 & 3 & 0.1 & 0.411 & [0.397, 0.425] \\
 5 & 3 & 1.0 & 0.232 & [0.216, 0.248] \\
 5 & 3 & 10  & 0.162 & [0.146, 0.177] \\
\hline
15 & 6 & 0.1 & 0.612 & [0.609, 0.615] \\
15 & 6 & 1.0 & 0.465 & [0.459, 0.470] \\
15 & 6 & 10  & 0.334 & [0.327, 0.341] \\
\hline
15 & 9 & 0.1 & 0.633 & [0.631, 0.636] \\
15 & 9 & 1.0 & 0.494 & [0.488, 0.500] \\
15 & 9 & 10  & 0.321 & [0.314, 0.328] \\
\hline
\end{tabular}
\caption{EM results (Case~1): mean Procrustes-aligned error $\mathrm{Err}$ with $\sqrt{nm}$ normalization over $10$ replications.}
\label{tab:EM-sim}
\end{table}

\emph{Discussion.}
For $n=5$, EM attains $\mathrm{Err}\approx 0.16$–$0.41$ with only $N=15$ samples; for $n=15$, the mean error ranges from $\approx 0.33$ to $0.63$. A salient (and initially somewhat counterintuitive) pattern is that the error \emph{decreases} as $\sigma^2$ increases from $0.1$ to $10$. We observed the same phenomenon for the rank-$1$ power-iteration estimator (see Figure~\ref{fig:simu3} below). Two complementary explanations help understand this:

\noindent\textbf{Finite-sample regularization.} With small $N$ and moderate dimensions, the EM updates implicitly contain ridge-like terms that scale with $\sigma^2$; see, e.g. Remark~\ref{rem:impl-reg}. Larger $\sigma^2$ stabilizes the resulting inverses and reduces variance in the factor updates. The variance reduction can dominate the bias introduced by stronger regularization, lowering the \emph{Frobenius} estimation error in finite samples.

\noindent\textbf{Noise-aided separation in the lifted moment.} In the pairwise-lifted view used by the rank-$1$ method, $\pair(\Sigma)$ decomposes into two rank-$1$ components: a signal term $\vec(A_1A_1^\top)\otimes\vec(A_2A_2^\top)\otimes\vec(A_3A_3^\top)$ and an isotropic noise term $\sigma^2\,\vec(I_{n_1})\otimes\vec(I_{n_2})\otimes\vec(I_{n_3})$. As $\sigma^2$ grows, these components become more mutually incoherent, which makes the tensor power iteration better able to “peel off” the noise component and recover the signal component cleanly. This increased separability carries over to EM through a better-conditioned E-step. We stress that this is a \emph{finite-$N$} effect: for fixed $(n,m)$ and increasing $N$, higher $\sigma^2$ would eventually hurt; but at $N=15$ the conditioning/regularization benefits dominate. 
\medskip

\noindent\textbf{Case 2 (rank-$1$ power iteration: convergence).}
We fix $r=3$, $(n_1,n_2,n_3)=(15,15,15)$, $(m_1,m_2,m_3)=(3,3,3)$, $\sigma^2=1$, $N=400$, and vary the number of power-iteration steps $L\in\{1,\dots,10\}$. Factors are generated as in Case~1. Starting from random initialization, we report the same error $\mathrm{Err}$ and the estimate of $\sigma^2$, averaged over $50$ replications; see Figure~\ref{fig:simu2}. The three mode-wise error curves are nearly indistinguishable and stabilize after two iterations, in line with the theory that random-initialization effects dissipate quickly.

\begin{figure}[!htb]
\centering
\includegraphics[width=0.495\textwidth]{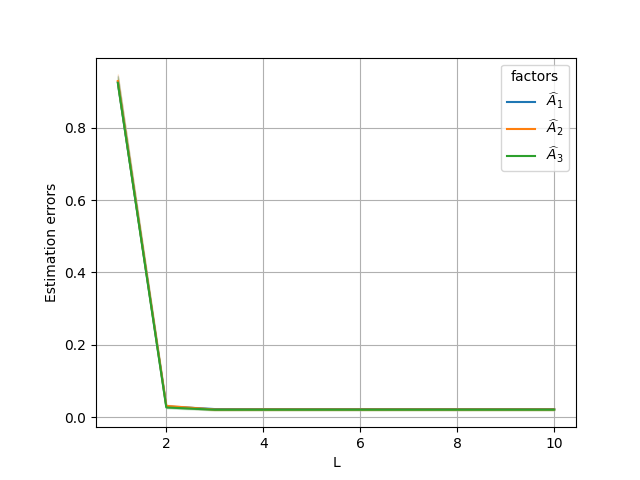}\;
\includegraphics[width=0.495\textwidth]{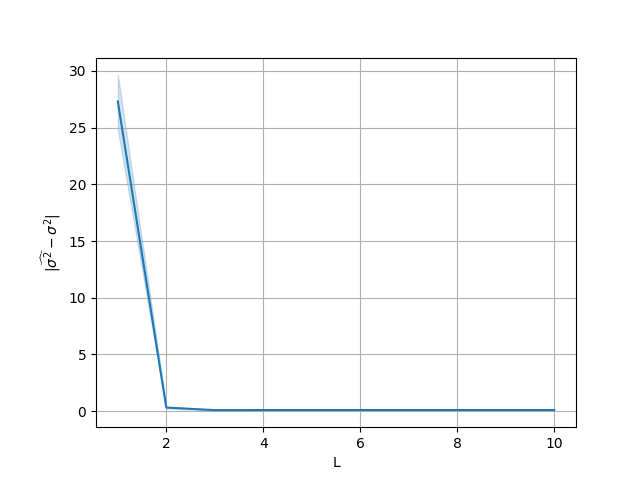}
\caption{Case~2 (power iteration): $\mathrm{Err}$ (left) and $\widehat{\sigma^2}$ (right) versus iteration $L$, with $95\%$ CIs over 50 replications. The error metric is the Procrustes-aligned $\sqrt{n_k m_k}$-normalized error defined at the start of the section.}
\label{fig:simu2}
\end{figure}

\medskip
\noindent\textbf{Case 3 (rank-$1$ power iteration: effect of $N$ and $\sigma^2$).}
We fix $r=3$, $(n_1,n_2,n_3)=(40,60,80)$, $(m_1,m_2,m_3)=(2,3,4)$ and vary $(\sigma^2,N)\in\{1,4,9,16,25\}\times\{100,200,300,400,500\}$. Factors are generated as before and scaled to $\|A_k\|_{\mathrm F}=\sqrt{n_k m_k}$. We report $\mathrm{Err}$ for each mode and the relative error $|\widehat{\sigma^2}-\sigma^2|/\sigma^2$, averaged over $50$ replications (Figure~\ref{fig:simu3}). As $N$ increases, all errors decrease. Consistent with Case~1, the mode-wise curves remain stable across $\sigma^2$ and tighten for larger modes; the relative error of $\widehat{\sigma}^2$ decreases for larger $\sigma^2$, reflecting easier separation of the isotropic component in the lifted moment.

\begin{figure}[!htb]
\centering
\includegraphics[width=0.49\textwidth]{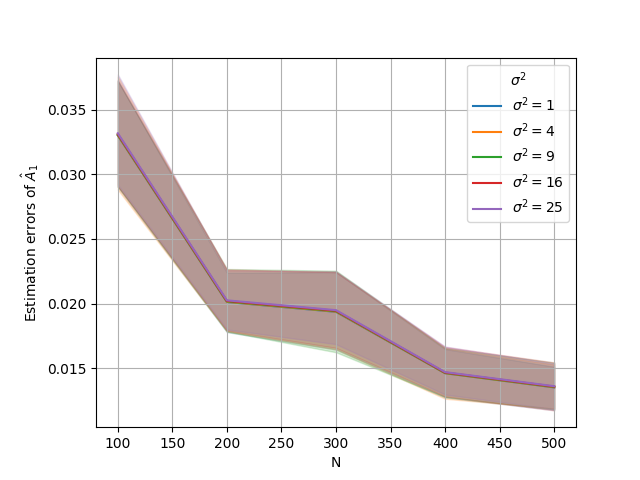}
\includegraphics[width=0.49\textwidth]{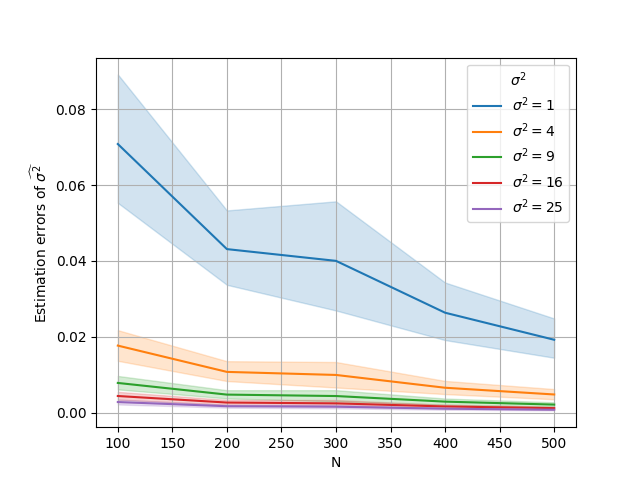}\\
\includegraphics[width=0.49\textwidth]{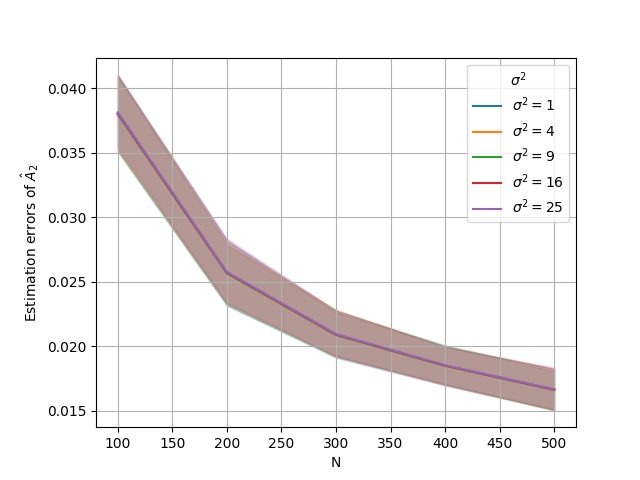}
\includegraphics[width=0.49\textwidth]{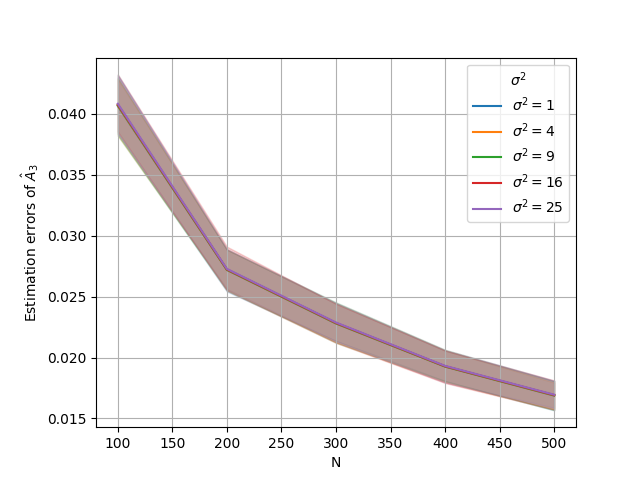}
\caption{Case~3 (power iteration): averaged $\mathrm{Err}$ for $\widehat A_1,\widehat A_2,\widehat A_3$ and the relative error of $\widehat{\sigma^2}$ across $(\sigma^2,N)$, with $95\%$ CIs over 50 replications. The error metric is the Procrustes-aligned $\sqrt{n_k m_k}$-normalized error defined at the start of the section.}
\label{fig:simu3}
\end{figure}

\section{Conclusion}\label{sec:conclusion}
In this paper, we propose a probabilistic tensor PCA model, extending the vector format PPCA model to general tensors. The proposed model is general and flexible, as it includes many tensor factor models as special cases and can be easily adapted to symmetric or potentially other tensor subspaces by fixing the bases of the subspaces. An sufficient and necessary condition for model identifiability and the existence of the MLE have been proved. Moreover, we proposed an EM algorithm and rank-$1$ power iteration algorithm for effective estimation. Particularly, we provide a non-trivial non-asymptotic statistical guarantees for the rank-$1$ power iteration algorithm, which further implies the asymptotic consistency of its estimators even just after the first iteration. The feasibility of these algorithms and their statistical guarantees are further supported by simulation studies. 

\paragraph*{Acknowledgments} 

We would like to thank Ming Yuan for  a helpful discussion during his visit to the University of Toronto in 2024. 

\putbib[ref.bib]
\end{bibunit}

\newpage
\begin{bibunit}[apalike]

\appendix

\section{Proofs of Section~\ref{sec:tensors}}\label{app:tensors_proofs}

\noindent \textbf{Proof of Lemma \ref{lem:normsequal}. }Let $\mathbf A=(A_1,\dots,A_r)$ and recall that $\sF_{\mathbf A}(T)=\mathbf A\cdot T$. Lemma~\ref{lem:tuckermat} gives the matrix–flattening
\[
   \mat(\sF_{\mathbf A}) \;=\;
   A_r \otimes_K A_{r-1} \otimes_K \cdots \otimes_K A_1 .
\]
For any two matrices $B,C$ the identity
$\|B\otimes_K C\|=\|B\|\,\|C\|$ holds for the spectral norm; an
$r$-fold induction therefore yields
\[
\|\sF_{\bs A}\|\;=\;   \|\mat(\sF_{\mathbf A})\|
   \;=\;\prod_{k=1}^{r}\|A_k\|.
\]
For the multilinear norm computation, take arbitrary rank-one tensors  
$\bs u=u_1\otimes\cdots\otimes u_r$ and  
$\bs v=v_1\otimes\cdots\otimes v_r$
with $\|u_k\|=\|v_k\|=1$ for all $k$. Since Tucker map respects rank-1 tensors,
\[
   \langle\sF_{\mathbf A}(\bs u),\bs v\rangle
     =\prod_{k=1}^{r} v_k^{\!\top}A_k u_k
     \;\le\;\prod_{k=1}^{r}\|A_k\|,
\]
where each inner product is bounded above by the largest singular
value $\|A_k\|$.  Equality is attained by choosing $u_k$ and $v_k$
to be the right and left singular vectors associated with
$\sigma_{\max}(A_k)=\|A_k\|$.  \qed

\noindent\textbf{Proof of Proposition \ref{prop:spmap}. }For the right implication, fix a collection $\mathbf A=(A_1,\ldots,A_r)$ and the corresponding Tucker map $\sF_{\bs A}(T)=\mathbf A\cdot  T$.  Note 
$$(\sF_{\bs A})_{i_1,\ldots, i_r, j_1,\ldots, j_r}\;=\;[\sF_{\bs A}( \bs e^{j_1,\ldots,j_r})]_{i_1,\ldots,i_r}\;=\;(A_1)_{i_1,j_1}\cdots (A_r)_{i_r,j_r},$$
and so
$$
(\pair(\sF_{\bs A}))_{\overline{i_1,j_1},\ldots,\overline{i_r,j_r}}\;=\;(A_1)_{i_1,j_1}\cdots (A_r)_{i_r,j_r}.
$$
In other words, $\pair(\sF_{\bs A})=\vec(A_1)\otimes \cdots\otimes \vec(A_r)$, and so it has rank 1. The left implication is similar. If $\pair(\sF)$ has rank 1 then 
$$(\sF( \bs e^{j_1,\ldots,j_r}))_{i_1,\ldots,i_r}\;=\;(A_1)_{i_1,j_1}\cdots (A_r)_{i_r,j_r}$$
for some matrices $A_1,\ldots,A_r$, which implies that $\sF=\sF_{\bs A}$, a Tucker map. \qed

\noindent\textbf{Proof of Lemma \ref{lem:tuckermat}. }
    Both $\mat(\sF_{\bs A})$ and  $A_r\otK\cdots \otK A_2\otK A_1$ are matrices with $n$ rows and $m$ columns. Take tuples $(i_1,\ldots,i_r)$ and $(j_1,\ldots,j_r)$. Equation \eqref{eq:matF} gives us an expression for $\mat(\sF)_{\overline{i_1,\ldots,i_r},\overline{j_1,\ldots,j_r}}$. By the proof of Proposition~\ref{prop:spmap}, for Tucker maps this gives
    $$
\mat(\sF)_{\overline{i_1,\ldots,i_r},\overline{j_1,\ldots,j_r}}\;=\;\prod_{k=1}^r (A_k)_{i_k,j_k}.
    $$
   It remains to show that the $(\overline{i_1, \ldots,i_r}, \overline{j_1, \ldots, j_r})$-th entry of $A_r\otK\cdots \otK A_2\otK A_1$ is also $\prod_{k=1}^r (A_k)_{i_k,j_k}$. First note that, by Definition~\ref{def:Kron}, 
    $
    (A_2\otK A_1)_{\overline{i_1,i_2},\overline{j_1,j_2}}=(A_2)_{i_2, j_2}(A_1)_{i_1,j_1}$. Thus the claim holds if $r=2$. In general, we will argue by induction. Suppose the claim is true for $r\leq s-1$ with $s\geq 2$. Directly by definition of the ranking function,
$$\overline{\overline{i_1,\ldots,i_{s-1}},i_s}\;=\;\overline{i_1,\ldots,i_s}.
    $$
Consequently, when $r = s$,    
\begin{eqnarray*}
    (A_s\otK A_{s-1}\otK \cdots \otK  A_1)_{\overline{i_1,\ldots,i_s},\overline{j_1,\ldots,j_s}}&=&(A_s\otK (A_{s-1}\otK\cdots \otK A_1))_{\overline{\overline{i_1,\ldots,i_{s-1}},i_s},\overline{\overline{j_1,\ldots,j_{s-1}},j_s}}\\    
    &=&(A_s)_{i_s,j_s} (A_{s-1}\otK\cdots \otK A_1)_{\overline{i_1,\ldots,i_{s-1}},\overline{j_1,\ldots,j_{s-1}}}. 
\end{eqnarray*}
By the inductive assumption, $(A_{s-1}\otK\cdots \otK A_1)_{\overline{i_1,\ldots,i_{s-1}},\overline{j_1,\ldots,j_{s-1}}}=\prod_{k=1}^{s-1}(A_k)_{i_k,j_k}$, which gives $(A_{s}\otK\cdots \otK A_1)_{\overline{i_1,\ldots,i_{s}},\overline{j_1,\ldots,j_{s}}}=\prod_{k=1}^{s}(A_k)_{i_k,j_k}$. Comparing with the formula for $\mat(\sF)$, we conclude the proof.   \qed

\noindent\textbf{Proof of Lemma \ref{lem:sqSigma}. }
   The $(\overline{i_1,\ldots,i_r}, \overline{j_1,\ldots,j_r})$-th entry of $\mat(\sM)$ is $\<\bs e^{i_1,\ldots,i_r},\sM \bs e^{j_1,\ldots,j_r}\>$, which is equal to $\<\sM \bs e^{i_1,\ldots,i_r},\bs e^{j_1,\ldots,j_r}\>$ or equivalently the $(\overline{j_1,\ldots,j_r}, \overline{i_1,\ldots,i_r})$-th entry of $\mat(\sM)$ due to the fact that $\sM$ is a self-adjoint operator. For the second statement, we simply observe that the equation $\sM T=\lambda T$ is equivalent to $\mat(\sM)\vec(T)=\lambda\vec(T)$.  \qed

\noindent\textbf{Proof of Lemma \ref{lem:linearGauss}. }
For every $T \in \V$, we have $\langle \sF Z, T \rangle = \langle Z, \sF^* T \rangle$, which has a Gaussian distribution with mean $\langle \mu, \sF^* T \rangle = \langle \sF\mu, T \rangle$ and variance $\langle \Sigma \sF^* T, \sF^* T \rangle = \langle \sF \Sigma \sF^* T, T \rangle$. \qed 

\noindent\textbf{Proof of Proposition \ref{prop:allold}. }
    By definition, $\<X,T\>$ has Gaussian distribution with mean $\<\mu,T\>$ and variance $\<\Sigma T,T\>$. Now note that 
    $\<X,T\>=\vec(X)^\top \vec(T)$, $\<\mu,T\>=\vec(\mu)^\top \vec(T)$.  Also, by Lemma~\ref{lem:matop} and Lemma~\ref{lem:sqSigma},
    $$
    \<\Sigma T,T\>\;=\;\vec(\Sigma T)^\top \vec(T)\;=\;\left(\mat(\Sigma)\vec(T)\right)^\top \vec(T)\;=\;\vec(T)^\top \mat(\Sigma)\vec(T).
    $$
    Showing that for any $\bs v\in \R^{n}$, $\bs v^\top \vec(X)$ is Gaussian with mean $\bs v^\top \vec(\mu)$ and variance $\bs v^\top \mat(\Sigma)\bs v$. The conclusion follows. \qed

\section{Computational details for the EM algorithm}\label{sec:comp_details}

In this section, we provide some technical details of the computational implementations needed for the EM-algorithm. We refer to Section~\ref{sec:tensorcomp} for necessary definitions. 

To optimize \eqref{eq:M} with respect to $A_k$, we first express it as a function of $A_k$ only. Denote by $\bs A^{(-k)}$ the the collection of matrices obtained from $\bs A$ by replacing $A_k$ with $I_{m_k}$. Denote by $\bs A^{(k)}$ the the collection of matrices obtained from $\bs A$ by replacing each $A_l$ for $l \neq k$ with $I_{m_l}$. For example,
$$
\bs A^{(-1)}\;=\;(I_{m_1},A_2,\ldots,A_r),\qquad \bs A^{(1)}\;=\;(A_1,I_{m_2},\ldots,I_{m_r}),
$$
and  $\bs A$ is a composition of maps defined by $\bs A^{(-1)}$ and $\bs A^{(1)}$. Thus, for every $k$,
$$
\sF_{\bs A} (\nu+\bs E_i)\;=\;\bs A^{(k)} \bs A^{(-k)}(\nu+\bs E_i).
$$
Denoting by $\cM_k$ the flattening map at mode $k$, also referred to the $k$-th mode matricization, we get
$$
\cM_k \left(\sF_{\bs A} (\nu+\bs E_i)\right)\;=\;A_k \cM_k\left(\bs A^{(-k)}(\nu+\bs E_i)\right)\;=\;A_k \bs U_{k,i},
$$
where we denote $$\bs U_{k,i}\;:=\;\cM_k\left(\bs A^{(-k)}(\nu+\bs E_i)\right).$$
Since the Frobenius norm is invariant under different reshaping of the tensor, we have
$$
\tfrac1N\sum_{i=1}^N \|X^{(i)}-\mu-\sF_{\bs A}\bs E_i\|^2_{\s{f}}\;=\;\tfrac1N\sum_{i=1}^N \|\cM_{k}(X^{(i)})-A_k \bs U_{k,i}\|^2_{\s{f}}.
$$
Next, we rewrite the trace term $\tr(\sF_{\bs A}\sV \sF^*_{\bs A})$ in \eqref{eq:M} explicitly in terms of $A_k$.
\begin{lem}
    We have
    $$
    \tr(\sF_{\bs A}\sV \sF_{\bs A}^*)\;=\;\left(\vec(A_1^\top A_1)^\top,\ldots,  \vec(A_r^\top A_r)^\top\right)\cdot \pair(\sV).
    $$
\end{lem}
\begin{proof}
First
    $$
    \tr(\sF_{\bs A}\sV \sF_{\bs A}^*)\;=\;\tr(\sF_{\bs A}^*\sF_{\bs A}\sV )\;=\;\<\sF_{\bs A}^*\sF_{\bs A},\sV\>,
    $$
    where the inner product is taken in the space $\R^{m_1\times \cdots\times m_r\times m_1\times \cdots \times m_r}$. Since it is invariant under tensor reshaping, we get
    $$
    \tr(\sF_{\bs A}\sV \sF_{\bs A}^*)\;=\;\<\pair(\sF_{\bs A}^*\sF_{\bs A}),\pair(\sV)\>.
    $$
    By Proposition~\ref{prop:spmap}, $\pair(\sF_{\bs A}^*\sF_{\bs A})$ is the outer product $\vec(A_1^\top A_1)\otimes\cdots \otimes \vec(A_r^\top A_r)$ and the result follows.
\end{proof}
Using this lemma, we can rewrite for any $k=1,\ldots,r$, 
$$\tr(\sF_{\bs A}\sV \sF_{\bs A}^*)
\;=\;\vec(A_k^\top A_k)^\top \left(\vec(A_1^\top A_1)^\top,\ldots,I_{m_k},\ldots  \vec(A_r^\top A_r)^\top\right)\cdot \pair(\sV).
$$
Reshape the $m_k^2$-dimensional vector $(\vec(A_1^\top A_1)^\top,\ldots,I_{m_k},\ldots  \vec(A_r^\top A_r)^\top)\cdot \pair(\sV)$ into an $m_k\times m_k$ matrix (using $\vec^{-1}:\R^{m_k^2}\to \R^{m_k\times m_k}$) and denote this by $W_k$. Then $\tr(\sF_{\bs A}\sV \sF_{\bs A}^*)=\tr(A_k W_k A_k^\top)$ and we can rewrite \eqref{eq:M} as a function of $A_k$
$$
\tfrac1N\sum_{i=1}^N \|\cM_{k}(X^{(i)})-A_k \bs U_{k,i}\|^2_{\s{f}}+\tr(A_k W_k A_k^\top).
$$
The optimal solution is easily found to satisfy
\begin{align}
\label{eq: solution_A_k}
\widehat A_k\;=\; \left(\frac1N \sum_{i=1}^N\cM_{k}(X^{(i)}) \bs U_{k,i}^\top\right)\cdot(W_k+\tfrac1N \sum_{i=1}^N  \bs U_{k,i}\bs U_{k,i}^\top)^{-1},
\end{align}
where the matrix $W_k+\tfrac1N \sum_{i=1}^N  \bs U_{k,i}\bs U_{k,i}^\top$ is positive-definite almost surely under continuous data.

The update of $\nu$, defined in \eqref{eq:nu} and the update of $\sigma^2$ in \eqref{eq:M2} can be vectorized in a similar way.

\section{Technical proofs for Section~\ref{sec:algorithm}}\label{app:proofsalg}

\begin{lem}
\label{lem: hanson-wright}
(Hanson-Wright inequality for Gaussian Chaos). Suppose that  $\xi \sim N_{\s{v}}(0, \I_{\s{v}})$, and $\sM\in \mathcal L(\V,\V)$ assumed non-zero. Then $\E \<\sM(\xi),\xi\>=\tr(\sM)$ and, for any $t\ge 0$, we have
\begin{align*}
\P\left( |\<\sM\xi,\xi\>- \tr(\sM)| > t  \right) \;\le\; 2 \exp \left( - c_0 \min\left\{ \frac{t^2}{\|\sM\|_{\s{f}}^2}, \frac{t}{\|\sM\|}\right\}\right)\quad\mbox{with}\qquad c_0:=0.145. 
\end{align*}
Moreover, for the one-side version of the inequality, we can remove the coefficient $2$ on the right hand side. 
\end{lem}
The proof of this result follows immediately from the standard Hanson-Wright inequality as   stated in Lemma~1 of  \cite{moshksar2021absolute} and Proposition~1 of \cite{moshksar2024refining}, which improved the constant from $0.125$ to about $0.1457$. To see this, note that
$$
\<\sM \xi,\xi\>-\tr(\sM)\;=\;\vec(\xi)^\top \mat(\sM)\vec(\xi)-\E \left(\vec(\xi)^\top \mat(\sM)\vec(\xi)\right),
$$
and $\vec(\xi)\sim N_n(0,I_n)$ with $\|\sM\|=\|\mat(\sM)\|$ and $\|\sM\|_{\s{f}}=\|\mat(\sM)\|_{\s{f}}$ by definition.

\begin{rem}\label{rem:HansonWrightN}
    Consider now independent random samples $\xi_1,\ldots,\xi_N$ from $N_{\s{v}}(0,\I_{\s{v}})$. Then Lemma~\ref{lem: hanson-wright} assures that 
    \begin{align*}
\P\left( \Big|\tfrac1N\sum_{i=1}^N\<\sM \xi_i , \xi_i\>- \tr(\sM)\Big| > t  \right) \;\le\; 2 \exp \left( - c_0 N \min\left\{ \frac{t^2}{\|\sM\|_{\s{f}}^2}, \frac{t}{\|\sM\|}\right\}\right).
\end{align*}
This is because $\sum_{i = 1}^N \langle M \xi_i, \xi_i \rangle$ can be rewritten as a matrix quadratic form 
$$
\left(\vec(\xi_1)^\top, \ldots, \vec(\xi_N)^\top\right) \left(I_{N} \otK \mat(M)\right) \left(\vec(\xi_1)^\top, \ldots, \vec(\xi_N)^\top\right)^\top.
$$
\end{rem}

For all the proofs in this section, we rely on the following construction. Consider the set $\B_{\s{v}}$ as defined in \eqref{eq:Bv}. The components of $\B_{\s{v}}$ are the $\ell_2$-unit balls in each $\R^{n_k}$ for $k \in [r]$.  Fix $\epsilon>0$, and let $\{v_1^{(1)},\ldots  v_1^{(h_1)}\}, \ldots, \{v_{r}^{(1)}, \ldots, v_{r}^{(h_{r})}\}$ be $\epsilon$-covers of these balls. Then the covering numbers satisfy $h_k \le (\frac{2}{\epsilon} +1)^{n_k} $, for $k \in [r]$; see, for instance, Example~5.8 in \cite{Wainwright_2019}. Let
\begin{equation}\label{eq:Ncover}
\cN_{\s{v}}\;=\; \left\{\bs v=v_1\otimes \cdots \otimes v_r\in \V:\;v_1\in \{v_1^{(1)},\ldots  v_1^{(h_1)}\},\ldots,v_r\in \{v_{r}^{(1)}, \ldots, v_{r}^{(h_{r})}\}\right\}    
\end{equation}
be the corresponding cover of $\B_{\s{v}}$. We have $|\cN_{\s{v}}|\leq (\tfrac{2}{\epsilon}+1)^{\sum_k n_k}$. Similarly, we can define $N_{\s{u}}$ with $|\cN_{\s{u}}|\leq (\tfrac{2}{\epsilon}+1)^{\sum_k m_k}$.

\begin{lem}
\label{lemma: isotropic_concentration}
Let $\eps_1, \ldots, \eps_N$ be \textit{i.i.d} tensors following  $N_{\s{v}} (0, \sigma^2 \I_{\s{v}})$, and denote $c_0=0.145$. We have 
\begin{align*}
& \quad \P\left(\Big\| \tfrac{1}{N}\sum_{i = 1}^N\eps_i \otimes \eps_i - \sigma^2 \I_{\s{v}}\Big\|_{\rm ml} > t\right)\;\le\; (10 r)^{2\sum_{k=1}^r n_k}\exp \left( - {c_0N}\min\{\tfrac{t^2}{4\sigma^4},\tfrac{t}{2\sigma^2}\}\right).
\end{align*}
\end{lem}
\begin{proof}
Denote $\sT = \frac{1}{N}\sum_{i = 1}^N\eps_i \otimes \eps_i-\sigma^2\I_{\s{v}}$ and let $(\hat{\bs u},\hat{\bs v})\in \B_{\s{v}}\times \B_{\s{v}}$ be the  optimizer in the definition of $\|\sT\|_{\rm ml}$ in \eqref{eq:spnorm}. Consider the covering $\cN_{\s{v}}$ of $\B_{\s{v}}$ as defined in \eqref{eq:Ncover}. By construction, there exists $(\bs u,\bs v)\in \cN_{\s{v}}\times \cN_{\s{v}}$ such that $\|u_1 - \hat u_1\| \le \epsilon, \ldots, \|u_r - \hat u_r\| \le \epsilon, \|v_1 - \hat v_{1}\| \le \epsilon, \ldots, \|v_r - \hat v_{r}\|\le \epsilon$. Therefore, for this $(\bs u,\bs v)$,
$$
\left\langle \sT (\hat{\bs u}),\hat{\bs v} \right\rangle \;=\; \left\<\sT(\bigotimes_{k=1}^r \hat u_k),\bigotimes_{k=1}^r \hat v_k\right\>\; = \;\quad \left \<\sT\Big(\bigotimes_{k=1}^r (u_k+(\hat u_k-u_k))\Big),\bigotimes_{l=1}^r ( v_k + (\hat v_k-v_k))\right\>. 
 $$
 Using multilinearity of the tensor product, we can expand this last term onto $2^{2r}$ terms. The first term is $\<\sT(\bs u),\bs v\>$. The last term is
 $$
 \left\<\sT \left(\bigotimes_{k=1}^r (\hat u_k-u_k)\right),\bigotimes_{k=1}^r (\hat v_k-v_k)\right\>\;\leq\;(\prod_{k=1}^r \|\hat u_k-u_k\| \|\hat v_k-v_k\|) \cdot \|\sT\|_{\rm ml}\;\leq\;\epsilon^{2r}\|\sT\|_{{\rm ml}}.
 $$
 We deal with the remaining terms in a similar way to get
$$ \langle \sT(\hat{\bs u}),\hat{\bs v} \rangle \; \le\; \langle \sT(\bs u),\bs v \rangle  + \left[\sum_{\ell=1}^{2r}\binom{2r}{\ell} \epsilon^\ell\right] \|\sT\|_{\rm ml} \;=\;\langle \sT(\bs u),\bs v \rangle + \left[ (1+\epsilon)^{2r} -1 \right] \|\sT\|_{\rm ml}.
$$
Taking $\epsilon = (\frac{3}{2})^{\frac{1}{2r}} - 1$ yields $(1+\epsilon)^{2r} -1 = \frac{1}{2}$, and so, there exists $(\bs u,\bs v)\in \cN_{\s{v}}\times \cN_{\s{v}}$ such that $\|\sT\|_{\rm ml} \le 2\<\sT(\bs u), \bs v\>$. This leads to
\begin{equation}\label{eq:Tmlboundaux1}
\P\left(\left\| \sT \right\|_{\rm ml} > t\right)\;\leq\;\P(\max_{\bs u\in \cN_{\s{v}},\bs v\in \cN_{\s{v}}} \<\sT(\bs u), \bs v\>>\tfrac{t}{2})\;\leq\;\sum_{\bs u\in \cN_{\s{v}}}\sum_{\bs v\in \cN_{\s{v}}} \P( \<\sT(\bs u), \bs v\>>\tfrac{t}{2}).    
\end{equation}

Let $\sT_i=\eps_i\otimes \eps_i-\sigma^2 \I_{\s{v}}$ so that $\sT=\tfrac1N\sum_{i=1}^N \sT_i$. For each $i=1,\ldots,N$, and taking $\sM=\bs u\otimes \bs v$ we get
$$
\<\sT_i(\bs u),\bs v\>\;=\;\<(\eps_i\otimes \eps_i)\bs u,\bs v\>-\sigma^2\<\bs u,\bs v\>\;=\;\<\sM\eps_i,\eps_i\>-\sigma^2\tr(\sM).
$$
By Lemma~\ref{lem: hanson-wright}, we have
$$
\P(\<\sT_i(\bs u), \bs v\>>\tfrac{t}{2})\; = \;\P(\<\sM \eps_i,\eps_i\>-\sigma^2\tr(\sM)>\tfrac{t}{2})\;\leq\;\exp\left(-c_0\min\{\frac{t^2}{4\sigma^4\|\sM\|^2_{\s{f}}},\frac{t}{2\sigma^2\|\sM\|}\}\right).
$$ Moreover, $\|\sM\|^2_{\s{f}}=\|\sM\|=1$. By Remark~\ref{rem:HansonWrightN}, it follows that 
\begin{equation}\label{eq:aux2}
\P(\<\sT(\bs u), \bs v\>>\tfrac{t}{2})\;\leq\;\exp\left(-c_0 N \min\{\frac{t^2}{4\sigma^4},\frac{t}{2\sigma^2}\}\right).    
\end{equation}
With the given choice of $\epsilon$, we can show that 
$$|\cN_{\s{v}}|\;\leq\;\left( 1 + \frac{2}{(3/2)^{1/2r} -1}\right)^{ 2 (\sum_{k=1}^rn_k) }\;\leq\;(10r)^{ 2 (\sum_{k=1}^rn_k) }.$$
Indeed, note that equivalently $(1 +\frac{2}{10r -1})^{2r} \le \frac{3}{2}$ for any $r \ge 1$. Consider the change of variable $x = \frac{10r -1}{2} \ge \frac{9}{2}$, it then suffices to show that the function $f(x) = (1+\frac{1}{x})^{\frac{2x+1}{5}} \le \frac{3}{2}$ for any $x\ge \frac{9}{2}$. Let $g(x) = 5\log\left(f(x)\right) = (2x+1) \log(1+\frac{1}{x})$. Note that $g'(x) = - \frac{2x+1}{x(x+1)} +2 \log(1+\frac{1}{x})$ and $g^{\prime\prime}(x) = \frac{1}{x^2(1+x)^2} \ge 0$. Therefore, $g^\prime(x)$ is an increasing function, leading to $g^\prime(x) \le \lim_{a\rightarrow +\infty} g^\prime(a) = 0$, and thus both $g(x)$ and $f(x)$ are decreasing functions. Numerical evaluation yields that $f(x)\le f(4.5)\approx 1.4938 < 1.5$. We also note that $f(x) > \lim_{x\rightarrow +\infty}f(x) = e^{0.4} \approx 1.4918$. 
Finally, using \eqref{eq:Tmlboundaux1}, \eqref{eq:aux2},  and the bound on $|\cN_{\s{v}}|$, we get  
$$
\P\left(\| \frac{1}{N}\sum_{i = 1}^N\eps_i \otimes \eps_i - \sigma^2 \I_{\s{v}}\|_{\rm ml} > t\right)\;\leq\;(10r)^{ 2 (\sum_{k=1}^rn_k) } \exp\left(-c_0 N \min\{\frac{t^2}{4\sigma^4},\frac{t}{2\sigma^2}\}\right)
$$
as claimed.
\end{proof}

\begin{lem}
\label{lemma: concentration_Y}
Suppose that $Z_1, \ldots, Z_N \in \U$ independently follow standard normal distribution and $Y_i = \bs A \cdot  Z_i$ with $A_k\in \R^{n_k\times m_k}$, $m_k \le n_k$ for $k = 1, \ldots, r$. We have 
$$ \P\left(\Big\| \frac{1}{N}\sum_{i = 1}^N Y_i \otimes Y_i - \sF_{\bs A}\sF_{\bs A}^*\Big\|_{\rm ml} > t\right)\;\le\;  \left(10r\right)^{2\sum_{k=1}^rm_k} \exp\left( - {c_0 N}\min\left\{\tfrac{t^2}{4\|\bs A\|^4},\tfrac{t}{2\|\bs A\|^2}\right\}\right),
$$
where $c_0=0.145$ and $\|\bs A\|:=\|A_1\|\cdots\|A_r\|$.
\end{lem}
\begin{proof}
Taking  $\sM=\bs v\otimes\tilde{ \bs v}$, where $\bs v=v_1\otimes  \cdots\otimes v_r$, $\tilde{\bs v}=\tilde{v}_1\otimes\cdots\otimes \tilde{v}_r$, and both of them lie in $\B_{\s{v}}$. Recall that $Y_i=\sF_{\bs A}(Z_i)$, which leads to  
\begin{align*}
\<\sM Y_i,Y_i\>\;=\<\sM\sF_{\bs A}Z_i,\sF_{\bs A}Z_i\>\;=\;\<\sF_{\bs A}^*\sM\sF_{\bs A}Z_i,Z_i\> \;=\;\left\<(\sF_{\bs A}^* \bs v ) \otimes (\sF_{\bs A}^* \tilde{\bs v} ) Z_i,Z_i\right\>.
\end{align*}
Note that 
\begin{eqnarray*}
\sF^*_{\bs A}\bs v\;\in\;\sF^*_{\bs A}\B_{\s{v}}&=&\{(A_1^\top v_1)\otimes\cdots\otimes (A_r^\top v_r):\;\|v_1\|,\ldots,\|v_r\|\leq 1\} \\
& \subseteq &  \{u_1\otimes \cdots\otimes u_r:\;\;\|u_1\|\leq \|A_1\|,\ldots,\|u_r\|\leq \|A_r\|\}\\
& = &  \|A_1\|\cdots \|A_r\|  \cdot \{u_1\otimes \cdots\otimes u_r:\;\;\|u_1\|,\ldots,\|u_r\|\leq 1\}\\
&=& \|\bs A\| \cdot \B_{\s{u}},
\end{eqnarray*}
which implies that,
\begin{align*}
\max_{\bs v, \tilde{\bs v} \in \B_{\s{v}}} \left(\<\sM Y_i,Y_i\> - \langle M, F_{\bs A} F_{\bs A}^* \rangle \right) & = \max_{\bs v, \tilde{\bs v} \in \B_{\s{v}}} \left(\left\<(\sF_{\bs A}^* \bs v ) \otimes (\sF_{\bs A}^* \tilde{\bs v} ) Z_i,Z_i\right\> - \langle \sF_{\bs A}^* \bs v , \sF_{\bs A}^* \tilde{\bs v}   \rangle \right)\\
& \le  \|\bs A\|^2 \max_{\bs u, \tilde{\bs u} \in \B_{\s{u}}} \left(\<( \bs u \otimes  \tilde{\bs u} ) Z_i,Z_i\> - \langle \bs u, \tilde{\bs u} \rangle \right). 
\end{align*}
By Lemma~\ref{lemma: isotropic_concentration}, 
\begin{align*}
\P\left(\Big\| \frac{1}{N}\sum_{i = 1}^N Y_i \otimes Y_i - \sF_{\bs A}\sF_{\bs A}^*\Big\|_{\rm ml} > t\right) & \le  \P\left( \|\bs A\|^2\Big\| \frac{1}{N}\sum_{i = 1}^N Z_i \otimes Z_i - \I_{U}\Big\|_{\rm ml} > t\right)\\
& \le (10r)^{ 2 (\sum_{k=1}^rm_k) } \exp\left(-c_0 N \min\{\frac{t^2}{4\|\bs A\|^4},\frac{t}{2\|\bs A\|^2}\}\right).
\end{align*}
\end{proof}

\begin{lem}
\label{lemma: concentration_corss_term}
Under the setup of TPCA model, if we denote $Y_i = \bs A \cdot  Z_i$ as in Lemma~\ref{lemma: concentration_Y}, for $i = 1, \ldots, N$, we have
$$\P\left( \Big\|\tfrac{1}{N} \sum_{i = 1}^NY_i \otimes \eps_i\Big\|_{\rm ml}>t\right)\;\le\; \left(10r\right)^{\sum_{k=1}^r (m_k+n_k)}\exp\left( - \tfrac{N}{3} \min\left\{\tfrac{t^2}{4\sigma^2 \|\bs A\|^2},\tfrac{t}{2 \sigma \|\bs A\|}\right\}\right), 
$$
where $\|\bs A\|=\|A_1\|\cdots\|A_r\|$.
\end{lem}
\begin{proof}
Denote $\sT_i=Y_i\otimes \eps_i$, $\sT=\tfrac1N\sum_{i=1}^N \sT_i$, and let $ \bs v, \tilde{\bs v}\in \B_{\s{v}}$ as defined in \eqref{eq:Bv}. For every $i=1,\ldots,N$, we have
$$
\<\sT_i(\bs v), \tilde{\bs v}\>\;=\;\<Y_i,\tilde{\bs v}\>\<\eps_i,\bs v\>\;=\;\<Z_i,\sF^*_{\bs A} \tilde{\bs v}\>\<\eps_i,\bs v\>.
$$
Note that as shown in the proof of Lemma~\ref{lemma: concentration_Y}, $\sF^*_{\bs A} \tilde{\bs v} \in \|\bs A\|\cdot \B_{\s{u}}$, which implies that, 
\begin{equation}\label{eq:TleqS}
\max_{\bs v,\tilde{\bs v} \in \B_{\s{v}}} \<\sT(\bs v), \tilde{\bs v}\>\;\leq\;\|\bs A\| \max_{\bs u \in \B_{\s{u}}, \bs v \in \B_{\s{v}}} \<\sS(\bs v),\bs u\>,  
\end{equation}
where $\sS=\tfrac1N\sum_{i=1}^N Z_i\otimes \eps_i$. 
Let $\cN_{\s{u}}, \cN_{\s{v}}$ be $\epsilon$-covers of $\B_{\s{u}}, \B_{\s{v}}$ as defined in \eqref{eq:Ncover}. Taking $\epsilon = (\frac{3}{2})^{\frac{1}{2r}} - 1$ and using the standard covering argument as in the previous proofs, we get
\begin{eqnarray*}
\P(\|\sT\|_{\rm ml}\geq t)&=& \P(\max_{\bs v,\tilde{\bs v}\in \B_{\s{v}}} \<\sT(\bs v), \tilde{\bs v}\>\geq t)\;\overset{\eqref{eq:TleqS}}{\leq}\;\P(\max_{\bs u \in \B_{\s{u}},\bs v\in \B_{\s{v}}} \<\sS(\bs v), \bs u\>\geq \tfrac{t}{\|\bs A\|})  \\
&\leq & \P(\max_{\bs u\in \cN_{\s{u}}, \bs v\in \cN_{\s{v}}} \<\sS(\bs v), \bs u\>\geq \tfrac{t}{2\|\bs A\|})\;\leq\;\sum_{\bs u\in \cN_{\s{u}}}\sum_{\bs v\in \cN_{\s{v}}}\P( \<\sS(\bs v),\bs u\>\geq \tfrac{t}{2\|\bs A\|}).
\end{eqnarray*}
Now for a fixed $\bs u\in \cN_{\s{u}}$ and $\bs v\in \cN_{\s{v}}$ consider the variables $U_i:=\<Z_i,\bs u\>$ and $V_i:=\<\eps_i,\bs v\>$. Note that $U_i,V_i$ are mean zero independent Gaussian variables with variances $1$ and $\sigma^2$, respectively. Consider the following moment generating function, for every $\lambda$,
   $$
   \E({e^{\lambda U_iV_i}})\;=\;\E \left(\E({e^{\lambda U_iV_i}}|U_i)\right)\;=\;\E(e^{\tfrac12 \lambda^2\sigma^2 U_i^2}),
   $$
   where $U_i^2$ is a chi-squared random variable with degree of freedom $1$. The known formula for the moment generating function of $U_i^2$ gives that 
   $$
   \E({e^{\lambda U_iV_i}})\;= \;\frac{1}{\sqrt{1-\sigma^2\lambda^2}},\qquad\mbox{for }|\lambda|<\tfrac1\sigma.
   $$
   We can further bound
   $$
  \frac{1}{\sqrt{1-\sigma^2\lambda^2}}\;\leq\;e^{\ln(2)\sigma^2\lambda^2}, \qquad\mbox{for }|\lambda|<\tfrac{1}{\sqrt{2}\sigma}.
   $$
   meaning that $U_iV_i$ is sub-exponential with parameters $( \sqrt{2\ln(2)}\sigma,\sqrt{2}\sigma)$ or, in particular, with parameters $( \sqrt{\tfrac32}\sigma,\sqrt{2}\sigma)$; the latter is taken for simpler algebra. Note that $\<\sS(\bs v),\bs u\>=\tfrac1N\sum_{i=1}^N U_i V_i$. By the sub-exponential tail bound (Proposition~2.9 in \cite{Wainwright_2019}) and Remark~\ref{rem:HansonWrightN},
   $$
   \P(\tfrac1N\sum_{i=1}^N U_i V_i>\tfrac{t}{2\|\bs A\|})\;\leq\;\exp(-{N}\min\{\tfrac{t^2}{12\sigma^2\|\bs A\|^2},\tfrac{t}{4\sqrt{2}{\sigma}\|\bs A\|}\})\;\leq\;\exp(-\tfrac{N}{3}\min\{\tfrac{t^2}{4\sigma^2\|\bs A\|^2},\tfrac{t}{2{\sigma}\|\bs A\|}\}),
   $$
   where we used the fact that $\tfrac43\sqrt{2}<2$.
   It then follows that
   $$
   \P(\|\sT\|_{\rm ml}\geq t)\;\leq\;\left(10r\right)^{\sum_{k=1}^r (m_k+n_k)}\exp\left( - \tfrac{N}{3}\min\left\{\tfrac{t^2}{4\sigma^2 \|\bs A\|^2 }, \tfrac{t}{2\sigma \|\bs A\| }\right\}\right),
   $$
   as claimed.
\end{proof}

\subsection*{Proof of Theorem~\ref{theorem:concentration_spectral}}
\begin{proof} For any realization $X_i$ of $X$, let $Z_i,\eps_i$ be the corresponding realizations of $Z$ and $\eps$.
Denote $Y_i = \bs A \cdot  Z_i$, for $i = 1,\ldots, N$. Recall the definition of the spectral norm in \eqref{eq:spnorm}. Using stochastic representation of $X$ in \eqref{eq:XPPCA}, the form of $\Sigma$ in \eqref{eq:Sigma}, and the triangle inequality, we get 
$$
\|\tfrac{1}{N} \sum_{i = 1}^N X_i \otimes X_i - \Sigma\|_{\rm ml} \;\le\;   \|\tfrac{1}{N} \sum_{i = 1}^N \eps_i \otimes \eps_i - \sigma^2\I_{\s{v}}\|_{\rm ml}+\|\tfrac{1}{N} \sum_{i = 1}^N Y_i \otimes Y_i - \sF_{\bs A}\sF_{\bs A^*}\|_{\rm ml}  + 2\|\tfrac{1}{N} \sum_{i = 1}^N Y_i \otimes \eps_i \|_{\rm ml}.
$$
It then suffices to employ Lemmas~\ref{lemma: isotropic_concentration}, \ref{lemma: concentration_Y} and \ref{lemma: concentration_corss_term} to bound the three terms on the right hand side of the above display separately.

In the bound in Lemma~\ref{lemma: isotropic_concentration} take $t=t_1$ and  denote the right-hand side expression of the bound by $\delta_1$. Solving for $t_1$, we easily arrive at expression
$$
\min\{\tfrac{t^2_1}{4\sigma^4},\tfrac{t_1}{2\sigma^2}\}\;=\;\frac{1}{c_0 N}\left(2\log(10r)\sum_{k=1}^r n_k+\log\tfrac{1}{\delta_1}\right)\;:=\;\alpha_1.
$$
Note that for any $\alpha,\beta>0$ it holds that $\alpha=\min\{\beta,\beta^2\}$ if and only if $\beta=\max\{\alpha,\sqrt{\alpha}\}$. We can now take $\beta=\tfrac{t_1}{2\sigma^2}$ and $\alpha=\alpha_1$ to solve for $t_1$ in terms of $\delta_1$ and other parameters. We conclude that 
$$
\|\tfrac1N\sum_{i=1}^N \eps_i\otimes \eps_i-\sigma^2\I_{\s{v}}\|_{\rm ml}\;\leq\;2\sigma^2 \max\{\alpha_1,\sqrt{\alpha_1}\},\quad\mbox{with probability at least }1-\delta_1.
$$
Note that if we take $\delta_1=N^{-2\log(10r)\sum_k n_k}$, 
the expression for $\alpha_1$ yields
$$
\frac{1}{c_0 N}\left(2\log(10r)\sum_{k=1}^r n_k+\log\tfrac{1}{\delta_1}\right)\;=\;\frac{2\log(10r)\sum_{k=1}^r n_k}{c_0 N}\left(1+\log(N)\right),
$$
which converges to zero when $N \rightarrow \infty$ and $n_{\max}$ satisfies $\tfrac{n_{\max}\log N}{N}\to 0$.

We proceed in exactly the same way for the other two terms. In the bound in Lemma~\ref{lemma: concentration_Y} take $t=t_2$ and  denote the right-hand side expression of the bound by $\delta_2$. Solving for $t_2$, we arrive at the expression
$$
\min\{\tfrac{t^2_2}{4\|\bs A\|^4},\tfrac{t_2}{2\|\bs A\|^2}\}\;=\;\frac{1}{c_0 N}\left(2\log(10r)\sum_{k=1}^r m_k+\log\tfrac{1}{\delta_2}\right)\;:=\;\alpha_2.
$$
We can now take $\beta=\tfrac{t}{2\|\bs A\|^2}$ and $\alpha=\alpha_2$ to solve for $t_2$ as above. We conclude that 
$$
\|\tfrac1N\sum_{i=1}^N Y_i\otimes Y_i-\sF_{\bs A}\sF_{\bs A}^*\|_{\rm ml}\;\leq\;2\|\bs A\|^2 \max\{\alpha_2,\sqrt{\alpha_2}\}\quad\mbox{with probability at least }1-\delta_2.
$$
Note that if we take $\delta_2=N^{-2\log(10r)\sum_k m_k}$ the expression for $\alpha_2$ yields
$$
\frac{1}{c_0 N}\left(2\log(10r)\sum_{k=1}^r m_k+\log\tfrac{1}{\delta_2}\right)\;=\;\frac{2\log(10r)\sum_{k=1}^r m_k}{c_0 N}\left(1+\log(N)\right),
$$
which converges to zero when $N\to \infty$ and $m_{\max}$ satisfies $\tfrac{ m_{\max}\log N}{N}\to 0$.

Finally, denote the probability bound in Lemma~\ref{lemma: concentration_corss_term} by $\delta_3$ and take $t = \frac{1}{2}t_3$. Solving for $t_3$ we arrive at expression
$$
\min\{\tfrac{t^2_3}{16\sigma^2\|\bs A\|^2},\tfrac{t_3}{4\sigma\|\bs A\|}\}\;=\;\frac{3}{ N}\left(\log(10r)\sum_{k=1}^r (m_k+n_k)+\log\tfrac{1}{\delta_3}\right)\;:=\;\alpha_3.
$$
We can now take $\beta=\tfrac{t_3}{4\sigma\|\bs A\|}$ and $\alpha=\alpha_3$ to solve for $t_3$ as above. We conclude that 
$$
2\|\tfrac1N\sum_{i=1}^N Y_i\otimes \eps_i\|_{\rm ml}\;\leq\;4\sigma\|\bs A\| \max\{\alpha_3,\sqrt{\alpha_3}\}\quad\mbox{with probability at least }1-\delta_3.
$$
Note that if we take $\delta_3=N^{-\log(10r)\sum_k (m_k+n_k)}$, the expression for $\alpha_3$ yields
$$
\frac{3}{N}\left(\log(10r)\sum_{k=1}^r (m_k+n_k)+\log\tfrac{1}{\delta_3}\right)\;=\;\frac{3\log(10r)\sum_{k=1}^r (m_k+n_k)}{N}\left(1+\log(N)\right),
$$
which converges to zero when $N\to \infty$ and $\tfrac{n_{\max}\log N}{N}\to 0$.

Finally, putting the pieces together, with probability at least $1- \delta_1-\delta_2-\delta_3$, we have $\|\frac{1}{N} \sum_{i = 1}^N X_i \otimes X_i - \Sigma\|_{\rm ml} \le t_1 + t_2 +t_3$. 
\end{proof}

\noindent \textbf{Proof of Theorem~\ref{theorem:rank-1poweriteration}} 
\begin{proof}
Consider the decomposition of $\sS_N$ given in \eqref{eq:S_Ndecomp}, $\pair(\sS_N)=\omega\;b_1\otimes \cdots \otimes b_r+\pair(\Phi)$, where $\Phi=\sigma^2\I_{\s{v}}+\sS_N-\Sigma$ and $\omega=\prod_{k=1}^r \|B_k\|_{\s{f}}$. 

To proof the theorem, without loss of generality, we analyze $\sin\theta_k^{(1)}$ for $k=r$. The analysis for general $k$ is analogous but with more tedious notation. Let $\tilde{b}_r$ be the result given by the first step power iteration as given in \eqref{equ: b_k_l_tilde}:
\begin{equation}\label{eq:b1tilde}
\begin{aligned}
\tilde{b}_r &= ({b_1^{(0)}}^\top,\ldots, {b_{r-1}^{(0)}}^\top, I_{n_r^2}) \cdot  \left(\omega\;b_1\otimes \cdots\otimes b_r+\pair(\Phi)\right)\\
&= \omega \prod_{k=1}^{r-1}\cos(\theta^{(0)}_k)b_r + ({b_1^{(0)}}^\top,\ldots, {b_{r-1}^{(0)}}^\top, I_{n_r^2}) \cdot  \pair(\Phi).
\end{aligned}    
\end{equation}
Denote by $\tilde B_r=\vec^{-1}(\tilde{b}_r)$ the associated matrix and let $$\tilde\Delta=\vec^{-1}\left(({b_1^{(0)}}^\top,\ldots, {b_{r-1}^{(0)}}^\top, I_{n_r^2}) \cdot  \pair(\Phi)\right)\in \R^{n_r\times n_r}.$$
We first upper bound the spectral norm of the matrix $\tilde\Delta$. Since $\tilde\Delta$ is positive semi-definite by construction, we can rewrite its spectral norm as
$$
\|\tilde \Delta\|\;=\;\max_{\|u_r\|\leq 1} \<\tilde\Delta,u_r u_r^\top\>.
$$
Let $\sum_{j =1}^{m_k} \lambda_{k, j}^{(0)} u_{k, j}^{(0)} (u_{k, j}^{(0)})^\top$ be the eigenvalue decomposition of $B_k^{(0)}$, for $k = 1, \ldots, r-1$. This gives
$$
b_k^{(0)}=\sum_{j=1}^{m_k}\lambda_{k,j}^{(0)}\vec(u_{k, j}^{(0)} (u_{k, j}^{(0)})^\top).
$$
By multilinearity, we can rewrite $\<\tilde\Delta,u_r u_r^\top\>$ as
$$
\sum_{j_1=1}^{m_1}\cdots \sum_{j_{r-1}=1}^{m_{r-1}} \lambda^{(0)}_{1,j_1}\cdots \lambda^{(0)}_{r-1,j_{r-1}}\left(\vec(u_{1,j_1}^{(0)}(u_{1,j_1}^{(0)})^\top,\ldots,\vec(u_{r}u_{r}^\top)\right)\cdot \pair(\Phi).
$$
\begin{rem}
Consider any $\sT\in \mathcal L(\V,\V)$ and $\bs u,\bs v\in \B_{\s{v}}$ as defined in \eqref{eq:Bv}. Then
\begin{equation}\label{eq:Tuvusefulbound}
\<\sT(\bs u),\bs v\>\;=\;\left(\vec(u_1v_1^\top)^\top,\ldots,\vec(u_rv_r^\top)^\top\right)\cdot \pair(\sT)\;\leq\;\|\sT\|_{\rm ml}.    
\end{equation}    
\end{rem}
Using this remark and the fact that $\|\Phi\|_{\rm ml}\leq \psi$ (as $\bs E(\psi)$ holds), we get that, under our assumed event, 
$$
\left(\vec(u_{1,j_1}^{(0)}(u_{1,j_1}^{(0)})^\top,\ldots,\vec(u_{r}u_{r}^\top)\right)\cdot \pair(\Phi)\;\leq\;\|\Phi\|_{\rm ml}\;\leq\;\psi.
$$
Moreover, all involved eigenvalues are non-negative and so
$$
\|\tilde\Delta\|\;\leq\;\psi\sum_{j_1=1}^{m_1}\cdots  \sum_{j_{r-1}=1}^{m_{r-1}} \lambda^{(0)}_{1,j_1}\cdots \lambda^{(0)}_{r-1,j_{r-1}} \;=\;\psi \prod_{k=1}^{r-1}\tr(B_k^{(0)})\;\leq\;\psi \prod_{k=1}^{r-1} \sqrt{m_k}\;=\;\psi\sqrt{\tfrac{m}{m_r}},
$$
where $m=m_1\cdots m_r$ and the last inequality follows by the fact that each $B_k^{(0)}$ has rank at most $m_k$ with $\|B_k^{(0)}\|_{\s{f}}=1$ and inequality \eqref{eq:lowrankFbound}.

Next, we study the eigendecomposition step of $\tilde{B}_r\in \R^{n_r \times n_r}$. Denote $\widehat{B}={\rm Eig}_{m_k}(\tilde B_{r})$ so that $B_r^{(1)}$ defined in \eqref{equ: b_k_l_hat} is just the normalized version of $\widehat{B}$. Denote 
$$c\;:=\;\tfrac{\omega}{\|B_r\|_{\s{f}}} \prod_{k=1}^{r-1}\cos(\theta^{(0)}_k).$$
The decomposition of $\tilde b_r$ in \eqref{eq:b1tilde} gives $\tilde B^{(1)}_r\;=\;c B_r+\tilde\Delta$. 
Let  $\hat\Delta$ be such that 
$$
\widehat{B}\;=\;c B_r+\hat\Delta.
$$ 
To bound the norm of $\hat\Delta$ note that, by the triangle inequality,
$$
 \|\hat\Delta \|\;\le\; \|\widehat{B} - \tilde B^{(1)}_r\| + \|\tilde B^{(1)}_r- cB_r \| \;=\;\|\widehat{B} - \tilde B^{(1)}_r\|+\|\tilde\Delta\|.
 $$
By Eckhart-Young Theorem (Theorem~2.4.8 in \cite{golub2013matrix}), $\widehat{B}$ is the minimizer of the spectral distance to $\tilde B^{(1)}$ among all matrices of rank at most $m_r$. Thus,
$$
\|\widehat{B} - \tilde B^{(1)}_r\|\;\leq\;\|cB_r-\tilde B^{(1)}_r\|\;=\;\|\tilde \Delta\|.
$$
It follows that 
$$
\|\hat\Delta\|\;\leq\;2\|\tilde\Delta\|\;\leq\;2\psi\sqrt{\tfrac{m}{m_r}}.
$$
Since $\widehat{B}$ and $B_r$ both have ranks at most $m_r$, it follows that ${\rm rank}(\hat\Delta)\leq 2m_r$ and so, by \eqref{eq:lowrankFbound},
\begin{equation}\label{eq:hatDeltaF}
\|\hat{\Delta}\|_{\s{f}}\;\leq\;\sqrt{{\rm rank}(\hat{\Delta})}\|\hat{\Delta}\| \;\leq \; \sqrt{2m_r} \cdot 2 \sqrt{\frac{m}{m_r}} \psi \;=\; 2\sqrt{2m} \psi.
\end{equation}
We are now ready to bound $\sin \theta_r^{(1)}$. By definition, 
\begin{eqnarray*}
    \sin \theta^{(1)}_r&=& \sqrt{1-\frac{\langle \widehat{B}, B_r\rangle^2}{\|\widehat{B}\|_{\s{f}}^2\|B_r\|_{\s{f}}^2}}\;=\;\frac{\sqrt{\|\widehat{B}\|_{\s{f}}^2\|B_r\|_{\s{f}}^2-\langle \widehat{B}, B_r\rangle^2}}{\|\widehat{B}\|_{\s{f}}\|B_r\|_{\s{f}}}.
\end{eqnarray*}
Note that
$$
\|\widehat{B}\|^2_{\s{f}}\|B_r\|^2_{\s{f}}\;=\;c^2\|B_r\|_{\s{f}}^4+2c\|B_r\|^2_{\s{f}}\<\hat\Delta,B_r\>+\|B_r\|^2_{\s{f}}\|\hat\Delta\|_{\s{f}}^2
$$
and
$$
\<\widehat{B},B_r\>^2\;=\;c^2\|B_r\|^4_{\s{f}}+2c\|B_r\|_{\s{f}}^2\<\hat \Delta,B_r\>+\<\hat\Delta,B_r\>^2.
$$
It follows that 
$$
\sin \theta^{(1)}_r\;=\;\frac{\sqrt{\|B_r\|^2_{\s{f}}\|\hat\Delta\|_{\s{f}}^2-\<\hat\Delta,B_r\>^2}}{\|\widehat{B}\|_{\s{f}}\|B_r\|_{\s{f}}}\;\leq\;\frac{\|\hat\Delta\|_{\s{f}}}{\|\widehat{B}\|_{\s{f}}}\;\leq\;\frac{\|\hat\Delta\|_{\s{f}}}{c\|B_r\|-\|\hat{\Delta}\|_{\s{f}}},
$$
where in the last step we use the triangle inequality: 
$$
\|\widehat{B}\|_{\s{f}}\;\geq \;c\|B_r\|_{\s{f}}-\|\hat{\Delta}\|_{\s{f}}.
$$
Using \eqref{eq:hatDeltaF} and the fact that the function $\tfrac{x}{a-x}$ for $a>0$ is increasing in $x$, we get
$$
\sin \theta^{(1)}_r\;\leq\;\frac{2\sqrt{2m}\psi}{c\|B_r\|_{\s{f}}-2\sqrt{2m}\psi}
$$
By Assumption~\ref{assp: tau}, $\sin \theta^{(0)}_k\leq \tau$ for all $k=1,\ldots,r$, or equivalently, $\cos \theta^{(0)}_k\geq \sqrt{1-\tau^2}$. This gives
$$
c\|B_r\|_{\s{f}}\;=\;\omega \prod_{k=1}^{r-1}\cos(\theta^{(0)}_k)\;\geq\; \omega(1-\tau^2)^{\frac{r-1}{2}}.
$$
Consequently, 
$$
\frac{2\sqrt{2m}\psi}{c\|B_r\|_{\s{f}}-2\sqrt{2m}\psi}\;\leq\;\frac{2\sqrt{2m}\psi}{\omega(1-\tau^2)^{\frac{r-1}{2}}-2\sqrt{2m}\psi}\;\leq\;\frac{4\sqrt{2m}\psi}{\omega(1-\tau^2)^{\frac{r-1}{2}}},
$$
which concludes the first part of the proof.

Finally, we turn to bound $|\omega - \widehat{\omega}^{(1)}|$, where $\widehat{\omega}^{(1)}$ is defined in \eqref{equ:omega_hat} (take $L=1$). By definition, 
\begin{align*}
\widehat{\omega}^{(1)} \;=\; \left(\widehat{b}_1^{\top}, \ldots, \widehat{b}_r^{ \top} \right)\cdot \pair(\sS_N)\;=\; \omega \prod_{k =1}^r \cos \theta_k^{(1)} + \left(\widehat{b}_1^{ \top}, \ldots, \widehat{b}_r^{ \top} \right)\cdot\pair(\Phi). 
\end{align*}
By the similar argument as bounding the spectral norm of $\tilde{\Delta}$, we obtain
\begin{align*}
\left(\frac{\widehat{b}_1^{ \top}}{\|\widehat{b}_1\|}, \ldots, \frac{\widehat{b}_r^{ \top}}{\|\widehat{b}_r\|}, \right)\cdot\pair(\Phi) \le \sqrt{m} \psi, 
\end{align*}
and so
\begin{align*}
|\omega - \widehat{\omega}^{(1)}|&\le \omega\left(1 -  \prod_{k =1}^r \sqrt{ 1- \frac{32m\psi^2}{ (1 -\tau^2)^{r-1}  \omega^2}} \right) + \sqrt{m}\psi\\
& \le \omega \cdot \frac{r}{2} \cdot \frac{32m\psi^2}{ (1-\tau^2)^{r-1}\omega^2} + \sqrt{m}\psi, 
\end{align*}
where the last inequality uses the facts that $r\ge 2$ and $(1 - x)^{\alpha} \ge 1 - \alpha x$ for $0\le x\le 1$ and $\alpha \ge 1$; here applied with $\alpha=\tfrac{r}{2}$ and $x=\tfrac{32m\psi^2}{(1-\tau^2)^{r-1}\omega^2}$. This gives that 
\begin{align*}
\frac{|\omega - \widehat{\omega}^{(1)}|}{\omega}\;\le\; \frac{16mr\psi^2}{ (1-\tau^2)^{r-1}\omega^2} + \frac{\sqrt{m}\psi}{\omega}  
\end{align*}
as claimed.
\end{proof}

\noindent \textbf{Proof of Corollary~\ref{cor: cor_A_consistency}}
\begin{proof}
Recall that $B_k^{(1)}$ is the output after the first iteration of our algorithm defined in \eqref{equ: b_k_l_hat}. By Theorem~\ref{theorem:rank-1poweriteration},  we have 
\begin{align}\label{eq:sqrt2bound}
\left\| B_k^{(1)} - \frac{B_k}{\|B_k\|_{\s{f}}}\right\|_{\s{f}}\;=\;\|b_k^{(1)}-b_k\|  \;\le\;  \sqrt{2}f_1(\tau),
\end{align}
and $ |\widehat{\omega} - \omega| \le f_2(\tau) \omega$, for any $k = 1, \ldots, r$.  Note that for any $x\ge 0, y>0$ and $r \in N^{+}$, we have 
\begin{align*}
|x - y| = \frac{|x^r - y^r|}{\sum_{i = 0}^{r-1} x^i y^{r-1+i}} \le \frac{|x^r - y^r|}{ y^{r-1}}, 
\end{align*}
where we denote $0^0 = 1$ as convention. Now taking $x = \hat{\omega}^{\frac{1}{r}}> 0$ and $y = \omega^{\frac{1}{r}} > 0$ in the above display yields that

\begin{align*}
|\widehat{\omega}^{\frac{1}{r}} - \omega^{\frac{1}{r}}| \le |\widehat{\omega} - \omega| \omega^{-\frac{r-1}{r}}  \le f_2(\tau) \omega^{\frac{1}{r}},
\end{align*}
 Moreover, by Assumption~\ref{assp:scale}, we have $\|B_1\|_{\s{f}} =\ldots = \|B_r\|_{\s{f}} = \omega^{1/r}$, and so after a single iteration ($L=1$), our estimate of $B_k$ would be $\widehat B_k:=\widehat{\omega}^{\frac1r}B_k^{(1)}$ (based on \eqref{equ:A_hat})
$$
\|\widehat{B}_k-B_k\|_{\s{f}}\;\leq\;\|\widehat{\omega}^{\frac1r}B_k^{(1)}-\omega^{\frac1r}B_k^{(1)}\|_{\s{f}} +\|\omega^{\frac1r}B_k^{(1)}-B_k\|_{\s{f}}\;=\;|\widehat\omega^{\frac1r}-\omega^{\frac1r}|+\omega^{\frac1r}\|b_k^{(1)}-b_k\|,
$$
which, by \eqref{eq:sqrt2bound} and the above bound on $|\widehat\omega^{\frac1r}-\omega^{\frac1r}|$, can be bounded by 
$$
\left( f_2(\tau)+\sqrt{2}f_1(\tau)\right)\omega^{\frac1r}.
$$
By Lemma~6 of \cite{ge2017no}, there exists an orthogonal matrix $O \in \R^{m_k \times m_k}$, such that 
\begin{align*}
\frac{1}{\sqrt{n_k m_k}}\|\widehat{A}_k - A_k O\|_{\s{f}} \;\;\le\;\; \frac{ \left[ f_2(\tau) +  \sqrt{2}f_1(\tau)\right]\omega^{\frac{1}{r}}}{2(\sqrt{2} - 1) \sqrt{n_k m_k}\lambda_{m_k}(B_k)}, 
\end{align*}
for any $k = 1, \ldots, r$.
\end{proof}

\noindent\textbf{Proof of Theorem~\ref{theorem: sigma_2_consistency}}
\begin{proof}
We first control $\frac{1}{n}|\tr(J_1)|$. In the proof of Corollary~\ref{cor: cor_A_consistency}, assuming $\bs E(\psi)$ holds,  we have $\|\widehat{B}_k - B_k\|_{\s{f}} \le f_3(\tau) \omega^{\frac{1}{r}}/\sqrt{2}$, where $f_3(\tau) = 2f_1(\tau) + \sqrt{2}f_2(\tau)$. Let $\Delta_k = \widehat{B}_k - B_k$. It then follows that $\Delta_k$ is symmetric and has rank at most $2m_k$.  Moreover, by Cauchy-Schwarz inequality
\begin{align*}
|\tr(\Delta_k)| \le \sum_{j = 1}^{2m_k} \sigma_j(\Delta_k) \le \sqrt{2m_k} \sqrt{\sum_{j = 1}^{2m_k} \left(\sigma_j(\Delta_k)\right)^2 }= \sqrt{2m_k} \|\Delta
_k\|_{\s{f}} \le \sqrt{m_k}f_3(\tau) \omega^{\frac{1}{r}}, 
\end{align*}
where $\sigma_j(\Delta_k)$ is the $j$-th leading singular value of $\Delta_k$, for $k = 1, \ldots, r$. Therefore, 
\begin{align*}
\frac{1}{n}|\tr(J_1)|&=\frac{1}{n}\left| \tr\left(  \widehat{B}_r \otK\ldots \otK \widehat{B}_1 - B_r\otK \ldots \otK B_1\right)\right|\\
& =\frac{1}{n}\left| \prod_{k =1}^r \tr(\widehat{B}_k) - \prod_{k = 1}^r \tr(B_k) \right|\\
& =\frac{1}{n}\left| \prod_{k =1}^r \tr(B_k+ \Delta_k) - \prod_{k = 1}^r \tr(B_k) \right|\\
& \le \frac{\sqrt{m}}{n} \left[ \prod_{k =1}^r \left(\|B_k\|_{\s{f}} +f_3(\tau) \omega^{\frac{1}{r}}\right) - \prod_{k = 1}^r \|B_k\|_{\s{f}} \right]\\
& = \frac{\sqrt{m}\omega}{n} \left[ \left(1 + f_3(\tau) \right)^r - 1 \right]. 
\end{align*}

Next, we apply the Hanson-Wright inequality to control $\frac{1}{n} |\tr(J_2)|$. Note that $\vec(X_i)$ shares the same distribution of $\left[\mat(\Sigma)\right]^{^{\frac{1}{2}}}\xi_i$, where $\xi_1, \ldots, \xi_N \overset{i.i.d}{\sim} N_n(0, I_n)$. Therefore, 
\begin{align*}
& \quad \P\left(\frac{1}{n} \left \langle \mat\left(\sS_N\right) - \mat(\Sigma) , I_n\right \rangle > t\right)\\
& = \P\left( \left \langle \frac{1}{N}\sum_{i = 1}^N \xi_i \xi_i^\top - I_n, \mat(\Sigma)\right \rangle > nt\right)\\
& = \P\left( (\xi_1^\top, \ldots, \xi_N^\top) I_N \otimes_K \mat(\Sigma)(\xi_1^\top, \ldots, \xi_N^\top)^\top -  \tr\left(I_N \otK \mat(\Sigma)\right) > nNt\right)\\
&\le \exp\left(- c_0 \min \left\{ \frac{n^2N^2t^2}{\|I_N\otK \mat(\Sigma)\|_{\s{f}}^2}, \frac{nNt}{\|I_N\otK \mat(\Sigma)\|}\right\}\right). 
\end{align*}
where $c_0 = 0.145$ as before. Moreover, $\|I_N\otK \mat(\Sigma)\|_{\s{f}}^2 = N \|\Sigma\|_{\s{f}}^2 \le N [m \|\Sigma\|^2 + (n-m)\sigma^4]$, and $\|I_N\otK \mat(\Sigma)\| = \|\Sigma\| = \|A_1\|^2\ldots \|A_r\|^2 + \sigma^2$, leading to 
\begin{align*}
& \quad \P\left(\frac{1}{n} \left \langle \mat\left(\sS_N  \right) - \mat(\Sigma) , I_n\right \rangle > t\right)\\
&\le \exp\left(- \frac{c_0 nNt}{\| \Sigma\|} \min \left\{ \frac{nt}{m\|\Sigma\| + (n-m)\sigma^2/ \|\Sigma\|}, 1\right\}\right). 
\end{align*}
Taking
\begin{align*}
& \alpha_4 = \frac{2\|\Sigma\|\log(Nn)}{c_0N\left(m\|\Sigma\| + (n-m)\sigma^2/\|\Sigma\|\right)}, \text{ and }\\
&t_4 = \frac{m\|\Sigma\| + (n-m)\sigma^2/\|\Sigma\|}{n} \max\{\alpha_4, \sqrt{\alpha_4}\}, 
\end{align*}
we have $\frac{1}{n}\|\tr(J_2)\| \le t_4$ with probability at least $1 - \exp\left(-2\log(Nn)\right) = 1 - \frac{1}{(Nn)^2}$. 
\end{proof}

\putbib[ref.bib]
\end{bibunit}

\end{document}